\documentclass[oneside,a4paper,reqno]{amsart}
\usepackage{pdfsync, enumerate}
\usepackage{stmaryrd}
\usepackage{mathrsfs}
\usepackage{multicol}
\usepackage{amsmath, amsthm, amscd, amssymb, latexsym, eucal}
\usepackage{mathtools} 
\usepackage[all]{xy}
\def\serieslogo@{} \def\@setcopyright{} \makeatother
\usepackage{tikz-cd}
\tikzcdset{scale cd/.style={every label/.append style={scale=#1},
    cells={nodes={scale=#1}}}}
\usepackage{multienum}

\usepackage[colorinlistoftodos]{todonotes}

\usepackage{hyperref}
\usepackage{color}
\usepackage{cite}

\makeatletter
\renewcommand*\env@matrix[1][c]{\hskip -\arraycolsep
  \let\@ifnextchar\new@ifnextchar
  \array{*\c@MaxMatrixCols #1}}
\makeatother

\usepackage{color}

\numberwithin{equation}{section}
\newtheorem{thm}{Theorem}[section]
\newtheorem*{main-thm}{Main Theorem}
\newtheorem*{Auslander-thm}{Auslander's Theorem}

\newtheorem{cor}[thm]{Corollary}
\newtheorem{lem}[thm]{Lemma}
\newtheorem{prop}[thm]{Proposition}

\newtheorem*{thmA}{Theorem~A}
\newtheorem*{thmB}{Theorem~B}
\theoremstyle{definition}
\newtheorem{defn}[thm]{Definition}
\newtheorem{rem}[thm]{Remark}
\newtheorem{exam}[thm]{Example}

\newtheorem*{problemA}{Problem~A}
\newtheorem*{problemB}{Problem~B}

\newcommand{\A}{\mathscr A}
\newcommand{\B}{\mathscr B}
\newcommand{\C}{\mathscr C}

\newcommand{\D}{\mathscr D}

\newcommand{\N}{\mathcal N}

\newcommand{\T}{\mathcal T}
\newcommand{\U}{\mathcal U}
\newcommand{\V}{\mathcal V}

\newcommand{\mc}{\mathsf{C}}

\newcommand{\mK}{\mathsf{K}}

\newcommand{\mD}{\mathsf{D}}

\newcommand{\mKb}{\mathsf{K^b}}
\newcommand{\mDsg}{\mathsf{D_{sg}}}
\newcommand{\mDb}{\mathsf{D^b}}

\def\a{\alpha}
\def\b{\beta}

\def\e{\varepsilon}

\def\i{\iota}

\DeclareMathOperator*{\Ker}{\mathsf{Ker}}
 \DeclareMathOperator*{\Image}{\mathsf{Im}}
\DeclareMathOperator*{\Coker}{\mathsf{Coker}}

 \DeclareMathOperator{\pd}{\mathsf{pd}}

\DeclareMathOperator*{\Mod}{\mathsf{Mod}-\!}

 \DeclareMathOperator*{\smod}{\mathsf{mod}-\!}

 \DeclareMathOperator*{\umod}{\underline{\mathsf{mod}}-\!}

\DeclareMathOperator*{\proj}{\mathsf{proj}}
\DeclareMathOperator*{\Inj}{\mathsf{Inj}}
\DeclareMathOperator*{\Proj}{\mathsf{Proj}}

\DeclareMathOperator{\Hom}{\mathsf{Hom}}

\DeclareMathOperator{\Ext}{\mathsf{Ext}}
\DeclareMathOperator*{\Tor}{\mathsf{Tor}}

\newcommand{\Dsg}{\mathsf{D}_\mathsf{sg}}

\newsavebox{\proofbox}
\savebox{\proofbox}{\begin{picture}(7,7)%
  \put(0,0){\framebox(7,7){}}\end{picture}}

\usepackage{latexsym}
\usepackage{pstricks}
\usepackage{comment}

\begin{document}
%\linenumbers

%\setprotcode\font
%    {\it \setprotcode\font}
%    {\bf \setprotcode\font}
%    {\bf \it \setprotcode\font}
%    \pdfprotrudechars=1

\title[Equivariant recollements and singular equivalences]
{Equivariant recollements and singular equivalences}

\author[Karakikes]{Miltiadis Karakikes}
\address{Department of Mathematics, National and Kapodistrian University of Athens Panepistimioupolis, 15784 Athens, Greece}
\email{miltoskar@math.uoa.gr}

\author[Kontogeorgis]{Aristeides Kontogeorgis}
\address{Department of Mathematics, National and Kapodistrian University of Athens Panepistimioupolis, 15784 Athens, Greece}
\email{kontogar@math.uoa.gr}

\author[Psaroudakis]{Chrysostomos Psaroudakis}
\address{Department of Mathematics, Aristotle University of Thessaloniki, Thessaloniki, 54124, Greece}
\email{chpsaroud@math.auth.gr}

\date{\today}
 

 %Presheaves and sheaves, stacks, descent conditions (category-theoretic aspects)

\begin{abstract}
In this paper we investigate equivariant recollements of abelian (resp.\ triangulated) categories. We first characterize when a recollement of abelian (resp.\ triangulated) categories induces an equivariant recollement, i.e.\ a recollement between the corresponding equivariant abelian (resp.\  triangulated) categories. We further investigate singular equivalences in the context of equivariant abelian recollements. In particular, we characterize when a singular equivalence induced by the quotient functor in an abelian recollement lifts to a singular equivalence induced by the equivariant quotient functor. As applications of our results: (i) we construct equivariant recollements for the derived category of a quasi-compact, quasi-separated scheme where the action comes from a subgroup of the automorphism group of the scheme and (ii) we establish new singular equivalences between certain skew group algebras. 
\end{abstract}

\maketitle

\setcounter{tocdepth}{1} \tableofcontents

\section{Introduction and the Main Results}

Equivariant categories appear naturally in various settings and are omnipresent in algebraic geometry, algebraic topology and representation theory. A common feature of ``equivariant mathematics" is that whenever a group acts on a category of interest, studying sheaves over the orbit space is equivalent to studying sheaves over the equivariant category. For example, let $X$ be a variety and $G$ a finite group. It is known that a $G$-action on $X$ induces a $G$-action on the bounded derived category $\mathsf{D}^{\mathsf{b}}(X)$ of coherent sheaves on $X$ and moreover the equivariant derived category $\mathsf{D}^{\mathsf{b}}(X)^G$ can be considered as the derived category of coherent sheaves on the quotient stack $[X/G]$. The starting point of group actions and derived categories was the pioneering work of Bernstein and Lunts \cite{BernsteinLunts}, where they constructed the derived category of equivariant sheaves on locally compact topological $G$-spaces. It should be noted that the study of equivariant sheaves stems from the seminal paper of Grothendieck \cite{Tohoku} who first studied the category of $G$-equivariant sheaves of a space $X$ and found relations to the category of sheaves of the quotient space.

Let $f\colon X\to Y$ be a continuous map between locally compact spaces. It is well known that between the unbounded derived categories of sheaves $\mathsf{D}(X)$ and $\mathsf{D}(Y)$ there are derived functors: $f_*$, $f^*$, $f_{!}$, $f^{!}$ together with the duality, the hom and the tensor product. The main objective of Bernstein and Lunts' book was to obtain an equivariant version of these functors having first established a suitable notion of equivariant derived category. It should be noted that the latter functors are part of what is called {\em Grothendieck's six functor formalism} which by the fundamental work of Beilinson, Bernstein and Deligne \cite{BBD} can be encoded via recollement diagrams of triangulated categories. In the same paper, gluing of t-structures along a recollement has been introduced and the associated hearts form a recollement of abelian categories. This is a diagram of abelian categories and additive functors:
\begin{equation*}
\begin{tikzcd}
 \A \arrow[rr, "\mathsf{i}" description]& & \B \arrow[ll, bend left, "\mathsf{p}"] \arrow[ll, bend right, "\mathsf{q}"'] \arrow[rr, "\mathsf{e}" description] & & \C \arrow[ll, bend left, "\mathsf{r}"] \arrow[ll, bend right, "\mathsf{l}"']
\end{tikzcd}\eqno \mathsf{R}_{\mathsf{ab}}(\A,\B,\C)
\end{equation*}
which satisfies certain properties. Roughly speaking, a recollement $\mathsf{R}_{\mathsf{ab}}(\A,\B,\C)$ means that the sequence $\A \to \B\to \C$  is simultaneously a localisation and a colocalisation sequence. Such diagrams have been studied extensively in representation theory in connection to certain homological questions, see for instance \cite{Psaroud:survey} for a summary. Motivated by the work of Bernstein and Lunts on equivariant derived categories we formulate the next natural question for abelian recollements.

\begin{problemA}
Given a recollement $\mathsf{R}_{\mathsf{ab}}(\A,\B,\C)$ of abelian categories and $G$ a finite group acting on $\B$, under what conditions can we construct a recollement $\mathsf{R}_{\mathsf{ab}}(\A^G,\B^G,\C^G)$ of equivariant abelian categories?
\end{problemA}

Our first main result, proved in Theorem~\ref{main1}, provides necessary and sufficient conditions for constructing the equivariant $\mathsf{R}_{\mathsf{ab}}(\A^G,\B^G,\C^G)$ from $\mathsf{R}_{\mathsf{ab}}(\A,\B,\C)$.

\begin{thmA}
Let $\mathsf{R}_{\mathsf{ab}}(\A,\B,\C)$ be a recollement of abelian categories and $G$ a finite group acting on $\B$. Then the following are equivalent:
\begin{itemize}
\item[(i)] $G$ acts on $\C\simeq \B/\A$ and $\mathsf{e}$ is a $G$-functor. 
\item[(ii)] $\A$ is a $G$-invariant subcategory of $\B$. 
\end{itemize}
If either of the above conditions holds true, then $\mathsf{R}_{\mathsf{ab}}(\A^G, \B^G, \C^G)$ is a recollement of equivariant abelian categories.
\end{thmA}

Motivated by the above result we say that $\mathsf{R}_{\mathsf{ab}}(\A,\B,\C)$ \textsf{lifts to a} $G$-\textsf{equivariant recollement} $\mathsf{R}_{\mathsf{ab}}(\A^G,\B^G,\C^G)$ if a  group $G$ acts on $\B$ such that it satisfies one of the equivalent conditions of Theorem~A (Definition~\ref{equivariantrecollement}). In the context of Theorem~A, the equivariant recollement $\mathsf{R}_{\mathsf{ab}}(\A^G,\B^G,\C^G)$ induces an equivalence of categories between $\B^G / \A^G$ and $\C^G\simeq (\B/\A)^G$. This was first explored in the paper of Chen, Chen and Zhou \cite{CCZ} using monadic techniques. From the abelian context of Theorem~A we can pass to the triangulated one and ask if a triangulated analogue of Theorem~A holds. Indeed this is true and proved in Theorem~\ref{main2} with some extra assumptions concerning the triangulated structure. We remark that the case of semi-orthogonal decompositions has been studied by Sun \cite{ChaoSun}. 

In recent years, there has been a lot of attention on equivariant singularity categories. One reason is that equivariant singularity categories appear quite naturally as they can describe derived categories of varieties. For instance, let $X$ be a smooth variety and $s$ a regular section of a vector bundle $\mathcal{E}$. Then the bounded derived category of coherent sheaves on the zero scheme of $s$ is triangle equivalent to a certain equivariant singularity category, for more details see \cite{Isik}. On the other hand, certain homotopical invariants of equivariant singularity categories have been recently investigated. More precisely, Brown and Dyckerhoff studied in \cite{BrownDyckerhoff} the topological K-theory spectrum of the differential graded singularity category of a weighted projective hypersurface over the complex numbers. Furthermore, homological invariants like the singular Hochschild cohomology have been investigated via singular equivalences of Morita type with level, see~\cite{Wang:SingularHH}.

Motivated by these results, and by the recent developments on singular equivalences, i.e.\ triangle equivalences between singularity categories, it is natural to investigate triangle equivalences between equivariant singularity categories. In the abelian recollement context, we formulate the following problem.

\begin{problemB}
Let $\mathsf{R}_{\mathsf{ab}}(\A,\B,\C)$ be a recollement of abelian categories which lifts to a $G$-equivariant recollement $\mathsf{R}_{\mathsf{ab}}(\A^G,\B^G,\C^G)$. Assume that the quotient functor $\mathsf{e}\colon\B \to \C$ induces a singular equivalence between $\B$ and $\C$: 
\[
\mDsg (\mathsf{e})\colon \mDsg (\B)\xrightarrow[]{\simeq} \mDsg (\C).
\]
Can we construct a singular equivalence between $\mDsg (\B^G)$ and $\mDsg (\C^G)$, and between the equivariant singularity categories $\mDsg (\B)^G$ and $\mDsg (\C)^G$?
\end{problemB}

In this direction, we prove in Theorems~\ref{main3} and~\ref{main4} and Theorem~\ref{main6} the following result which constitutes the second main result of this paper. Under certain conditions, it provides the desired equivariant singular equivalences.

\begin{thmB}
Let $\mathsf{R}_{\mathsf{ab}}(\A, \B, \C)$ be a recollement of abelian categories that lifts to a $G$-equivariant recollement with $|G|$ invertible in $\B$. Assume that $\B$ and $\C$ have enough projectives and that the singularity categories $\mDsg(\B^G)$ and $\mDsg(\C^G)$ are idempotent complete. Then the following statements are equivalent:
\begin{itemize}
\item[(i)] The functor $\mathsf{e} \colon \B \to \C$ induces a singular equivalence:
\[
\mDsg (\mathsf{e})\colon \mDsg (\B)\xrightarrow[]{\simeq} \mDsg (\C).
\]

\item[(ii)] The $G$-functor $e^G\colon \B^G \to \C^G$ induces a singular equivalence:
\[
\mDsg (\mathsf{e}^G)\colon \mDsg (\B^G)\xrightarrow[]{\simeq} \mDsg (\C^G).
\]
\end{itemize}
\end{thmB}

The above result lifts the right hand side of $\mathsf{R}_{\mathsf{ab}}(\A, \B, \C)$ first to the bounded derived categories and then to the singularity categories. A natural question is when one can construct a recollement at the level of bounded derived categories. In \cite{Psaroud}, necessary and sufficient conditions were given for such a lifting. In this respect, assuming that the recollement  $\mathsf{R}_{\mathsf{ab}}(\A, \B, \C)$ lifts to a $G$-equivariant recollement, we construct two recollements of triangulated categories: namely, a recollement of derived equivariant categories and a recollement of equivariant derived categories. We show that these two triangulated recollements are compatible in a canonical way; we refer to Theorem~\ref{main5} for more details.

We describe the contents of the paper section by section. In Section~\ref{Preliminaries}, we recall notions and results on finite group actions on additive categories and discuss properties of equivariant categories. Moreover, we recall examples of geometric actions induced by a subgroup of automorphisms of a variety $X$ and of algebraic actions arising as a subgroup of ring automorphisms. In this case, the equivariant category of modules is isomorphic to category of modules of the skew group ring. We also expose some equivariant machinery we need regarding triangulated categories.

In Section~\ref{sectionrecolabel}, we investigate lifts of recollements of abelian and triangulated categories to G-equivariant recollements (Definitions~\ref{equivariantrecollement} and~\ref{equivarianttriangrecollement}) and provide necessary and sufficient conditions for these lifts (Theorems~\ref{main1} and~\ref{main2}). Moreover, we recall a well-known example of a recollement of module categories induced by an idempotent element of a ring. We show that, when the idempotent is G-invariant, there is a natural equivalence between the equivariant recollement induced by an action of automorphisms on the ring and the recollement induced by the corresponding idempotent of the skew group algebra (see subsection~\ref{examplewithidempotents}).

In Section~\ref{section: Yoneda extensions}, we show that if a G-functor is a k-homological embedding (Definition~\ref{homembeddeff}), then the induced equivariant functor is also a k-homological embedding. In Section~\ref{equivariant derived cats}, we investigate two different ways to construct triangulated recollements given a $G$-equivariant recollement of abelian categories $\mathsf{R_{ab}}(\A, \B, \C)$. One construction is taking the equivariant categories first and then their respective bounded derived categories, resulting in $\mathsf{R_{tr}}( \mDb(\A^G), \mDb(\B^G), \mDb(\C^G))$. The other construction consists of taking the bounded derived categories first and then the equivariant categories, yielding $\mathsf{R_{tr}}( \mDb(\A)^G, \mDb(\B)^G, \mDb(\C)^G)$. In Theorem~\ref{main5}, we show that the these two recollements are connected via a comparison functor which is an equivalence if $|G|$ is invertible in $\B$ (Theorem~\ref{derivedequivariantequiv}). 
In the process of doing that, we study the derived functors of equivariant functors in Proposition~\ref{finitenessoflocaldimension} and show that they satisfy certain homological properties. We also extend \cite[Theorem~7.2]{Psaroud} to the equivariant setting by showing in Propositions~\ref{derivedrecol-left-part} and~\ref{derivedrecol-right-part} that a recollement of abelian categories lifts to a recollement of bounded derived categories if and only if the corresponding G-equivariant recollement has the same property.

In Section~\ref{singularcats}, we investigate singularity categories (in the sense of Buchweitz \cite{Buch} and Orlov \cite{Orlov}) in the equivariant setting and singular equivalences in the context of recollements. In particular, we extend the result of Psaroudakis, Skarts\ae{}terhagen and Solberg \cite[Theorem~5.2]{PSS} to the equivariant setting (Theorem~\ref{main3})  by showing that the equivariant functor $\mathsf{e}^G$ of the quotient functor $\mathsf{e} \colon \B \to \C$ induces a singular equivalence provided that $\mDsg(\B^G)$ and $\mDsg(\C^G)$ are idempotent complete. We also prove that the converse holds (Theorem~\ref{main4}); i.e.,\ if $\mathsf{e}^G \colon \B^G \to \C^G$ induces a singular equivalence, then $\mathsf{e}$ also induces a singular equivalence, without the assumption of idempotent completions. In Corollary~\ref{recollementofmodulesandequiv} we give equivalent conditions for the quotient functor of the recollement of module categories of Artin algebras induced by a $G$-invariant idempotent element to be a singular equivalence.

The final Section~\ref{sectionexamples} is devoted to examples and applications of our main results. The first application is of geometric nature, while the rest are algebraic. In subsection~\ref{geomexample}, we show that given a recollement of unbounded derived categories of quasi-coherent sheaves on a scheme $X$ (satisfying certain conditions),
an open subspace $U$, and a suitable subcategory supported on the complement of $U$, 
a group action by automorphisms of $X$ induces an action on these categories, yielding an equivariant recollement.
The second application, subsection~\ref{triangmatrixexample}, concerns equivariant triangular matrix rings; we show that these are also triangular matrix rings and obtain results regarding equivariant singular equivalences (Corollary~\ref{globaldimensionoftriangmatrixrings}). 
The final subsection, \ref{equivariantmoritatypewithlvl}, extends the work of Qin \cite[Theorem~4.1]{YQ} regarding singular equivalences of Morita type with level to the equivariant setting (Corollary~\ref{equivariantmoritatypewithlvlcor}). 
Finally, we derive in Corollary~\ref{HHC} a new isomorphism of Gerstenhaber algebras between the singular Hochschild cohomology of skew group algebras.

\subsection*{Acknowledgments} 
The research work of the first author was supported by the Hellenic Foundation for Research and Innovation (HFRI) under the 4th Call for HFRI Ph.D. Fellowships (Fellowship Number: 9270). For the second author: The research project is implemented in the framework of H.F.R.I call ``Basic research Financing (Horizontal support of all Sciences)" under the National Recovery and Resilience Plan ``Greece 2.0" funded by the European Union -- NextGenerationEU (H.F.R.I. Project Number: 14907). For the third author: The research project is implemented in the framework of H.F.R.I call ``Basic research Financing (Horizontal support of all Sciences)" under the National Recovery and Resilience Plan ``Greece 2.0" funded by the European Union -- NextGenerationEU (H.F.R.I. Project Number: 16785).

We thank the anonymous referee for valuable suggestions and comments. We also thank Martin Kalck for useful discussions on equivariant singularity categories. 

\section{Equivariant Preliminaries}
\label{Preliminaries}

In this section we collect all necessary equivariant preliminaries that are used extensively throughout this paper. All functors are covariant and all groups appearing in the paper are finite.

\subsection{Group Actions on Categories}

Let $G$ be a finite group and $\D$ be a category. 

\begin{defn}\label{defaction}
A {\bf right action} $(\rho,\theta)$ of G on $\D$ consists of
\begin{itemize}
\item an auto-equivalence $\rho_g\colon \D \rightarrow \D$ for all g in $G$, and
\item an isomorphism of functors $\theta_{g,h}\colon \rho_g \circ \rho_h \xrightarrow[]{\simeq} \rho_{hg}$ for all $g$, $h$ in $G$,
    \end{itemize} 
satisfying, for all $g,h,k\in G$, the following \textbf{2-cocycle condition}:
\begin{equation}\label{cocycle}
\begin{tikzcd}
\rho_g\rho_h\rho_k \arrow[r, "\rho_g \theta_{h,k}"] \arrow[d, "\theta_{g,h}\rho_k"' ] & \rho_g \rho_{kh} \arrow[d, "\theta_{g,kh}"] \\
\rho_{hg}\rho_k \arrow[r, "\theta_{hg,k}"'] & \rho_{khg}
\end{tikzcd}
\end{equation}
The natural isomorphisms $\theta_{g,h}$ are called {\bf composition isomorphisms.}
\end{defn}

Similarly, one can define a left group action on $\D$. Note that transitioning from a right $G$-action to a left $G$-action can be achieved by setting $\rho'_g = \rho_{g^{-1}}$ and $\theta'_{g,h}=\theta_{g^{-1}, h^{-1}}$. 
Note that each group action admits a natural isomorphism $u \colon \rho_{1_G} \xrightarrow{\simeq} \mathsf{Id}$ called the \textbf{unit} of the action. 

For a right group action and an object $X$ we adopt the notation ${X^g} \coloneqq \rho_g(X)$, while for a morphism $f \colon X \to Y$ we write $f^g\coloneqq \rho_g(f) \colon {X^g } \to {Y^g}$. 

Suppose that we have a right action of a finite group $G$ on two categories $\D$ and $\hat{\D}$. We write $ (\D,\rho, \theta)$ for the action of $G$ on $\D$ and $(\hat{\D},\hat{\rho},\hat{\theta})$ for the action on $\hat{\D}$.

\begin{defn}\label{G-functordef}
A {\bf $G$-functor} $(F,\sigma)\colon (\D,\rho, \theta) \rightarrow (\hat{\D},\hat{\rho},\hat{\theta})$ consists of a functor $F\colon \D\rightarrow \hat{\D}$ together with a family $\sigma$ of natural isomorphisms $\{ \sigma_g\colon F \circ \rho_g \rightarrow \hat{\rho}_g \circ F\}_{g\in G}$ compatible with the associativity conditions of both sides, that is, the following diagram commutes:
\begin{equation}\label{associativity}
\begin{tikzcd}
F\rho_g\rho_h \arrow[r, "F \theta_{g,h}"] \arrow[d, "\sigma_{g}\rho_h"' ] & F \rho_{hg} \arrow[dd, "\sigma_{hg}"] \\
\hat{\rho}_g F \rho_h \arrow[d,"\hat{\rho}_g \sigma_h"'] \\
\hat{\rho}_{g}\hat{\rho}_h F \arrow[r, "\hat{\theta}_{g,h} F"'] & \hat{\rho}_{hg}F\\
\end{tikzcd}
\end{equation}
\end{defn}

Note that there exists a variation of the above notion given in \cite[Definition~2.18]{ChaoSun}.
%which keeps track of the units of the action asking that a $G$-functor satisfies the following commutative diagram:
%\begin{equation*}
%\begin{tikzcd}
%F \rho_e \arrow[rr, "\sigma_e"] \arrow[dr, "F %u"']& & \hat{\rho}_e F \arrow[dl, "\hat{u} F"] \\
%& F
%\end{tikzcd}
%\end{equation*}
%where $u$ and $\hat{u}$ are the units of the %actions on $\D$ and $\hat{\D}$, respectively.
In this paper we work with $G$-functors as introduced in Definition~\ref{G-functordef}. 

\begin{defn}
    Let $(F,\sigma^F)$ and $(H,\sigma^H)$ be two $G$-functors $(\D,\rho, \theta) \rightarrow (\hat{\D},\hat{\rho},\hat{\theta})$. A natural transformation $\eta \colon  F \Rightarrow H$ is called {\bf $G$-natural transformation} if it commutes with the action, i.e.\ for all $g \in G$ the following diagram is commutative:
\begin{equation}\label{G-natural}
    \begin{tikzcd}
F \rho_g \arrow[r, "\eta \rho_g"] \arrow[d,"\sigma^F_{g}"'] & H \rho_g \arrow[d, "\sigma^H_{g}"] \\
{\hat{\rho}_g}F \arrow[r, "{\hat{\rho}_g} \eta"'] & {\hat{\rho}_g}H
    \end{tikzcd}
\end{equation}
We write $\eta \colon  (F , \sigma^F) \Rightarrow_G (H, \sigma^H)$ for the $G$-natural transformation.
\end{defn}

The following well known example will be used throughout this work.

\begin{exam}\label{exampleonsheaves1}
Let $G$ be a finite group acting on a scheme $X$ by automorphisms. Then $G$ induces an action on the category $\mathsf{Coh}(X)$ of coherent sheaves over $X$. Indeed, $G$ acts on coherent sheaves by pullbacks $\rho_g \coloneqq g^* \colon  \mathsf{Coh}(X) \to \mathsf{Coh}(X)$, and for the composition there exist canonical isomorphisms 
$\theta_{g,h}\colon  g^* \circ h^* \xrightarrow[]{\simeq}(hg)^*$. 
This is obviously a right action. We could obtain a left action if we acted by pushforwards or, equivalently, by transitioning under the correspondence  $g_{*}= (g^{-1})^*$. Similarly, we have an action on the category $\mathsf{Qcoh} (X)$ of quasi-coherent sheaves of $X$.
\end{exam}

Notice that every group action on a category $\D$ yields a subgroup of $\mathsf{Aut}(\D)$, which is the group of isomorphism classes of auto-equivalences of $\D$. Two equivalences are identified if and only if there is a natural isomorphism between them. However, the converse is not always true. Even for $\mathbb{C}$-linear categories the failure is obstructed by a class in $H^3(G, \mathbb{C}^*)$   (see \cite[Theorem~2.1]{BeckOber}). For a particular example we refer to \cite[Paragraph~3.6]{BeckOber}.

\subsection{Definition and First Properties}

Throughout this subsection we assume that we have a category $\D$ with a right action of a finite group $G$, i.e.\ a triple $(\D, \rho, \theta)$ of Definition~\ref{defaction}. We define, now, the equivariant category $\D^G$.
 
A {\bf $G$-equivariant object} of $\D$ is a pair $(X, \chi)$ where $X$ is an object of $\D$ and $\chi$ is a family of isomorphisms $ \{ \chi_g\colon X \xrightarrow[]{\simeq} {X^g}\}_{g\in G}$, satisfying the following commutative diagram for all $g, h$ in $G$:
    \begin{equation}
    \label{eqcocyclecond}
        \begin{tikzcd}
            X \arrow[r, "\chi_g"] \arrow[rrr, bend right=20, "\chi_{hg}"'] & X^g \arrow[r, "(\chi_h)^g"] & (X^h)^g \arrow[r, "\theta_{g,h}^X"] & X^{hg}
        \end{tikzcd}
    \end{equation}
The family $\chi$ is called the \textbf{equivariant structure} or  \textbf{linearization} of $X$. We refer to the diagram $(\ref{eqcocyclecond})$ as the \textbf{ cocycle condition} on linearizations. An object $X \in \D$ is called $G$-\textbf{invariant} if $X^g \simeq X$ for all $g$ in $G$. It is clear that $G$-equivariant objects are $G$-invariant but in general a $G$-invariant object does not always admit a linearization. Even for $G$-invariant simple objects of a $\mathbb{C}$-linear category, it is known that they admit linearization if and only if they have trivial cohomology class in the second cohomology group $H^2(G, \mathbb{C}^*)$ (see \cite[Lemma~1]{Ploog}). 

Let $(X,\chi)$ and $(Y, \psi)$ be two $G$-equivariant objects. A morphism of equivariant objects $f\colon (X,\chi) \rightarrow (Y, \psi)$ is a morphism $f\colon X\to Y$ in $\D$ that commutes with the linearizations, i.e.\ the following diagram is commutative for all $g$ in $G$:
\begin{equation}\label{equivarianthomgroups}
        \begin{tikzcd}
            X \arrow[r, "f"] \arrow[d,"\chi_g"'] & Y \arrow[d, "\psi_g"] \\
            {X^g} \arrow[r, "{f^g}"] & {Y^g}
        \end{tikzcd}
    \end{equation}  
Equivalently, $f$ is invariant under the (right) action of $G$ on $\Hom_{\D}(X,Y)$ where $G$ acts as follows: $f. g \coloneqq (\psi_g)^{-1} \circ {f^g} \circ \chi_g$. Hence, we have $\Hom_{{\D}^G}((X,\chi),(Y,\psi)) = \Hom_{\D}(X,Y)^G$, which is the invariant subset of $\Hom_{\D}(X,Y)$ under the action of $G$ induced by $\chi$ and $\psi$. 
If the $G$-action on $\D$ is strict, then this is a $G$-action on each Hom-set in the usual sense, i.e.\ we have that $ f.1_G = f$ and $(f .h).g = f. (hg)$. For the former notice that $\chi_{1_G} = \mathsf{Id}_X$ and $\psi_{1_G} = \mathsf{Id}_Y$ in the case of strict action.
However, when the action is not strict, the diagram~$(\ref{equivarianthomgroups})$ describes a $G$-action on Hom-sets ``up to composition isomorphisms $\theta$".
To ease the reader's mind, every $G$-action on a category $\D$ is equivalent to a strict $G$-action on $\D$, see \cite[Theorem~5.4]{EvShinder}.

\begin{defn}
The \textbf{equivariant category} $\D^G$ is defined to be the category with objects the $G$-equivariant objects $(X, \chi)$ and morphisms those in $\D$ that satisfy diagram $(\ref{equivarianthomgroups})$. The category $\D^G$ is also called the \textbf{equivariantization} of $\D$ with respect to the action of $G$.
\end{defn}

\begin{rem}\label{0-object-equiv}
Let $\A$ be an abelian category. If a finite group $G$ acts on $\A$, then the equivariant category $\A^G$ is also abelian. 
Indeed, $0\to (X,\chi)\to (Y, \psi)\to (Z, \zeta)\to 0$ is a short exact sequence in $\A^G$ if and only if $0\to X\to Y\to Z\to 0$ is a short exact sequence in $\A$. Note that the zero object in $\A^G$ has a unique (trivial) linearization $\phi^0 = (\phi^0_g\colon  0 \xrightarrow[]{} { 0^g} =0)_{g \in G}$. This is the zero object of the equivariant category.

One can construct explicitly and canonically the kernels and cokernels of morphisms in $\A^G$ when $\A$ is abelian. We briefly describe the construction.
Regarding the cokernel of $f\colon  (X,\chi) \to (Y, \psi)$, we have the object $(\Coker (f), \phi) $ where the linearization is obtained by the commutative diagram:
    \begin{equation*}
        \begin{tikzcd}
            X \arrow[r, "f"] \arrow[d,"\chi_g"'] & Y \arrow[d, "\psi_g"] \arrow[r, "j"] & \Coker (f) \arrow[d, dashrightarrow, "\phi_g"]\\
            {X ^g } \arrow[r, "{f^g}"'] & {Y^g} \arrow[r, "{j^g}"'] & {\Coker(f)^g}
        \end{tikzcd}
\end{equation*}  
where the dashing arrows are induced by the universal property of the cokernel. Indeed, note that ${\Coker (f)^g}$ is the cokernel of ${f^g}$ since the composition ${j^g} \circ {f^g} = { (j \circ f)^g}= 0 $ is the zero morphism. The commutativity of the diagram implies ${j^g} \circ \psi_g \circ f = {j^g} \circ {f^g} \circ \chi_g = 0 $. Thus, the universal property of the cokernel, yields a unique morphism $\phi_g$ such that the diagram commutes. This is actually an isomorphism since it has an inverse $\phi_{g}^{-1}$, which we obtain using similar arguments and the inverses of $\chi_g$ and $\psi_g$. In order to show that the family $\{ \psi_g\}_{g\in G} $ is a linearization of $\Coker (f)$ one has to use the following diagram:
    \begin{equation*}
        \begin{tikzcd}
            X \arrow[ddd, bend right=70, "\chi_{hg}"'] \arrow[r, "f"] \arrow[d,"\chi_g"'] & Y \arrow[d, "\psi_g"] \arrow[r, "j"] & \Coker(f) \arrow[d, dashrightarrow, "\psi_g"] \arrow[ddd, bend left=70, dashrightarrow, "\phi_{hg}"]\\
            {X ^g} \arrow[r, "{f^g}"] \arrow[d, "{(\chi_h)^g }"']& {Y^g} \arrow[d, "{(\psi_h)^g}"] \arrow[r, "{j^g}"] & {\Coker(f)^g} \arrow[d, dashrightarrow, "{(\phi_h)^g}"] \\
            ( X^h)^g \arrow[d, "\theta_{g,h}"'] \arrow[r, "(f^h)^g "] & {(Y^h)^g} \arrow[d, "\theta_{g,h}"] \arrow[r, "{(j^h)^g}"] & { (\Coker(f)^h)^g} \arrow[d, "\theta_{g,h}"] \\
            {X^{hg}} \arrow[r, "{f^{hg}}"']& { Y^{hg}} \arrow[r, "{j^{hg}}"'] & { \Coker(f)^{hg}}
        \end{tikzcd}
    \end{equation*} 
Since ${\Coker (f)^{hg}}$ is a cokernel of $f$, by the universal property of $\Coker(f)$, we have $\theta_{g,h} \circ {(\phi_h)^g} \circ \phi_g=\phi_{hg}$. 
Clearly, the composition $(X,\chi) \xrightarrow[]{f} (Y, \psi) \to (\Coker (f), \phi) $ is the zero morphism, thus this object is the cokernel of $f$. Similarly, if $f\colon  (X,\chi) \to (Y, \psi)$ is a morphism in $\A^G$, then its kernel is the object $(\Ker (f), \chi_{|_{\Ker (f)}})$. 
\end{rem}

As a special case of the following remark we see that the above construction of kernels (resp.\ cokernels and images) in the equivariant category yields essentially a unique up to isomorphism kernel (resp.\ cokernel and image).

\begin{rem}
If $f\colon  X \to Y$ is an isomorphism and $X$ admits a linearization $\chi$, then $Y$ also admits a linearization $\psi $ given by the following commutative diagram:
    \begin{equation*}
        \begin{tikzcd}
            X \arrow[d, "f"] \arrow[r, "\chi_g"] \arrow[rrr, bend left=30, "\chi_{hg}"] & {X^g} \arrow[d, "{f^g}"] \arrow[r, "{(\chi_h)^g}"] & (X^h)^g \arrow[d, "(f^h)^g"] \arrow[r, "\theta_{g,h}"] & {X^{hg}} \arrow[d, "{f^{hg}}"]\\
            Y \arrow[r, dashrightarrow, "\psi_g"] \arrow[rrr, dashrightarrow, bend right=30, "\psi_{hg}"'] & {Y^g} \arrow[r, dashrightarrow, "{(\psi_h)^g}"] & {(Y^h)^g} \arrow[r, "\theta_{g,h}"] & {Y^{hg}}
        \end{tikzcd}
    \end{equation*}
where $\psi_g \coloneqq {f^g } \circ \chi_g \circ f^{-1} $. Thus the left and middle squares and the outer bended parallelogram are by definition commutative. The right square is commutative since $\theta_{g,h}$'s are natural isomorphisms for all $g,h \in G$. Observe that the commutativity of the bottom part of the diagram, i.e.\ $\phi_{hg} = \theta_{g,h} \circ  {(\psi_h)^g} \circ \psi_g$, follows from the commutativity of the rest of the diagram. 
\end{rem}

\begin{exam}\label{exampleonsheaves2}
Continuing Example~\ref{exampleonsheaves1}, we obtain the equivariant category of coherent sheaves $\mathsf{Coh}^G(X)$ which is abelian. Note that the quotient scheme $X/G$ exists if and only if every $G$-orbit is contained in an affine open by \cite{SGA1}[Expos\'e V, Proposition 1.8]. Assuming that the quotient $X/G$ is a scheme and that $G$ acts freely, then we have the following equivalence between the equivariant category of coherent sheaves and the category of coherent sheaves of the quotient scheme:
\[
\mathsf{Coh}(X/G) \xrightarrow[]{\simeq} \mathsf{Coh}^G(X).
\]
The above equivalence is given by pullback of sheaves along the quotient map $\pi\colon  X \to X/G$. Under these assumptions, the same equivalence holds for the categories of quasi-coherent sheaves, i.e.\ $\mathsf{Qcoh}(X/G) \simeq \mathsf{Qcoh}^G(X)$. If the action is not free or the quotient scheme does not exists, then we have to use the quotient stack $[X/G]$ to avoid obstructions. We remark that, if $X$ is quasi-projective, then every $G$-orbit is contained in an affine open (see \cite[Chapter~II, Paragraph~7, Remark]{Mumford}).
\end{exam}

In the following lemma we show that a $G$-functor induces a functor between the corresponding equivariant categories, which we call \textbf{equivariant functor}. 

\begin{lem}
\label{gfunctorinducesequivariant}
A $G$-functor $(F,\sigma)\colon (\D,\rho, \theta) \rightarrow (\hat{\D},\hat{\rho},\hat{\theta})$ induces an equivariant functor $F^G\colon \D^G \rightarrow \hat{\D}^G$. 
\begin{proof}
Let $(X,\chi)$ be an object in $\D^G$. We define $F^G(X,\chi) = (FX, \chi')$, where 
$ \chi'=\{\chi'_g\colon FX \xrightarrow{F\chi_g} F({X^g}) \xrightarrow[]{\sigma_g}{(FX)^g} \}_{g\in G} $ 
is the induced family of isomorphisms which satisfies the following cocycle condition:
\begin{equation*}
\begin{tikzcd}
            FX \arrow[r, "\chi'_g"] \arrow[rrr, bend right=20, "\phi'_{hg}"'] & {(FX)^g} \arrow[r, "(\chi'_h)^g"] & { ((FX)^h)^g} \arrow[r, "\hat{\theta}_{g,h}"] & {(FX)^{hg}}
\end{tikzcd}
\end{equation*}
Let $f\colon (X, \chi) \rightarrow (Y, \psi)$ be a morphism in $\D^G$. Then we define $F^G(f)$ to be the morphism $F(f)\colon (FX, \chi')\rightarrow(FY, \psi')$ in $\hat{\D}^G$. Indeed, from the commutative square
\[
\begin{tikzcd}
            X \arrow[r, "f"] \arrow[d, "\chi_g"'] & Y \arrow[d, "\psi_{g}"] \\
            {X^g} \arrow[r, "{f^g}"'] & {Y^g}
        \end{tikzcd}
\]        
by applying $F$, we obtain the upper commutative square diagram 
\[
\begin{tikzcd}
            FX \arrow[r, "F(f)"] \arrow[d, "F\chi_{g}"'] \arrow[dd, bend right=80, "\chi'_{g}"'] & FY \arrow[d, "F\psi_{g}"] \arrow[dd, bend left=80, "\psi'_{g}"] \\
            F({X^g}) \arrow[r, "F({f^g})"] \arrow[d, "\sigma^{X}_g"'] & F({Y^g}) \arrow[d, "\sigma^{Y}_g"] \\
            {(FX)^g} \arrow[r, "{(Ff)^g }"] & {(FY)^g}
\end{tikzcd}
\]       
Left and right ``handles" are commutative by definition and the bottom square is commutative since $\sigma_g$ is a natural transformation. We conclude that $F^G(f)$ is a morphism in $\hat{\D}^G$.
\end{proof}
\end{lem}

Clearly, if the functor $F$ in the above lemma is additive, then so is the induced equivariant functor $F^G$. Similarly, one can show the following. 

\begin{lem}\label{equivnaturaltransf}
A $G$-natural transformation $\eta\colon (F,\sigma^F) \Rightarrow_G (H,\sigma^H) $, induces an equivariant natural transformation $\eta^G \colon  F^G \Rightarrow H^G$ such that $\eta^G_{(X,\chi)} =\eta_X$.
\end{lem}

\begin{rem}
\label{compositionGfunctors}
Let $(F,\sigma^F)\colon (\D,\rho, \theta) \rightarrow (\hat{\D},\hat{\rho},\hat{\theta})$ and $(H,\sigma^H)\colon  (\hat{\D},\hat{\rho},\hat{\theta}) \rightarrow (\Tilde{\D}, \Tilde{\rho}, \Tilde{\theta})  $ be two $G$-functors  and $F^G\colon \D^G\rightarrow \hat{\D}^G, \ H^G\colon \hat{\D}^G \rightarrow \Tilde{\D}^G$ the induced equivariant functors. Then we observe that $(H \circ F, \sigma^{H \circ F})$ is a $G$-functor, with $\sigma^{H\circ F}_g = \sigma^H_g F \circ H \sigma^F_g$ for all $g \in G$, and $(H\circ F)^G = H^G \circ F^G$.
\end{rem}

\begin{rem} 
\label{inclusionfunctor}
Let $\C$ be a $G$-invariant subcategory of $\D$, i.e.\ $ C^g \in \C$ for all $C\in \C $ and $g\in G$. This means equivalently that $\rho_g(\C) \simeq \C$. In this case, the action of $G$ on  $\D$ restricts to an action of $G$ on $\C$. The inclusion functor $\C \hookrightarrow \D$ is naturally a $G$-functor.
\end{rem}

We have the following generalization:

\begin{lem}\label{Ginvariantsubcats}
    Let $(F,\sigma)\colon (\D,\rho, \theta) \rightarrow (\hat{\D},\hat{\rho},\hat{\theta})$ be a $G$-functor and $\C$ and $\hat{\C}$  be  $G$-invariant subcategories of $\D$ and $\hat{\D}$, respectively, such that $F(\C)\subset \hat{\C} $. Then $F^G(\C^G) \subset \hat{\C}^G$.
\end{lem}
\begin{proof}
Let $(X, \chi) \in \C^G$. Then $F^G(X, \chi) =  (FX, \chi')$ where $ \chi'=\{\chi'_g\colon FX \xrightarrow{F\chi_g} F({X^g}) \xrightarrow[]{\sigma_g}{(FX)^g} \}_{g\in G}$, as described in Lemma~\ref{gfunctorinducesequivariant}. Notice that $FX, F(X^g), (FX)^g$ are all objects of $ \mathcal{\hat{\C}}$ for all $g\in G$. We infer that $F^G(X, \chi)$ lies in $\hat{\C}$.
\end{proof}

\begin{lem}\label{ffequiv}
Suppose that $G$ acts on two additive categories $\D$ and $\D'$ and let $(F,\sigma) \colon  (\D, \rho, \theta) \rightarrow (\hat{\D}, \hat{\rho}, \hat{\theta})$ be a $G$-functor. If $F$ is fully faithful, then $F^G$ is also fully faithful.
\end{lem}
\begin{proof}
We need to check that there is the following bijection induced by $F^G$:
\[
F^G_{(X,Y)} \colon  \Hom_{\D^G}((X,\chi),(Y,\psi) ) \xrightarrow[]{\simeq} \Hom_{\hat{\D}^G}(F^G(X,\chi),F^G(Y,\psi) )
\]
We can write $\Hom_{\D}(X,Y)^G$ instead of $\Hom_{\D^G}((X,\chi),(Y,\psi) )$, where the $G$-action is induced by $\chi, \psi$. We can also write $\Hom_{\hat{\D}}(FX,FY)^G$ where the $G$-action is induced by $\chi' , \psi'$ which are described in Lemma~\ref{gfunctorinducesequivariant}. Since $F_{(X,Y)} \colon  \Hom_{\D}(X, Y ) \xrightarrow[]{\simeq} \Hom_{\hat{\D}}(FX, FY)$ is a $G$-module isomorphism, it restricts to a $G$-module isomorphism of their $G$-invariant parts.
\end{proof}

The following useful lemma shows that adjoints can be lifted to the equivariant setting, provided that at least one of them is a $G$-functor.

\begin{lem}\textnormal{(\!\!\cite[Lemma~2.19]{ChaoSun})}
\label{adjointequiv}
Let $\D$ and $\hat{\D} $ be two additive categories with $G$-actions. Suppose that $F \dashv H$ is an adjoint pair, with $F\colon  \D \rightarrow \hat{\D}$ and $H \colon  \hat{\D} \rightarrow \D$. If either of the two functors is a $G$-functor, then the other one becomes naturally a $G$-functor. The unit $\upsilon \colon \mathsf{Id} \Rightarrow HF$ and the counit $\nu\colon  FH \Rightarrow \mathsf{Id}$ become $G$-natural transformations. Moreover, have that $F^G \dashv H^G$, with unit $\upsilon^G$ and counit $\nu^G$.
\end{lem}

\begin{cor}\label{G functor equivalence}
Let $(F,\sigma^F)\colon (\D,\rho, \theta) \rightarrow (\hat{\D},\hat{\rho},\hat{\theta})$ be a $G$-functor. If $F\colon \D \xrightarrow[]{} \hat{\D}$ is an equivalence, then so is $F^G \colon  \D^G \xrightarrow[]{} \hat{\D}^G$.
\end{cor} 

We explain below how to transfer $G$-actions via an equivalence of categories.

\begin{lem}\label{inducedequivariantequivalence}
 Assume that $F\colon  \D \xrightarrow{\simeq} \mathcal{\hat{D}}$ is an equivalence of categories and let $G$ be a finite group acting on $\D$. Then, there is an induced action on $\mathcal{\hat{D}}$ rendering $F$ a $G$-functor. Moreover, $F^G \colon  \D^G \to \hat{\D}^G$ is also an equivalence of categories.
\end{lem}
\begin{proof}
Let $F^{-1}$ denote a quasi-inverse of $F$.
For any $g \in G$ define an auto-equivalence $\hat{\rho}_g \colon  \mathcal{\hat{D}} \to \mathcal{\hat{D}}$ as the composition $F \circ \rho_g\circ F^{-1}$. Now $\hat{\theta}$ is induced by $\theta$, i.e.\ $\hat{\theta}_{g,h}$  is the composition
\[ 
\hat{\rho}_g \circ \hat{\rho}_h = F \rho_g F^{-1} \circ F \rho_h F^{-1} \xrightarrow{F \rho_g \e \rho_h F^{-1}} F\rho_g \circ \rho_h F^{-1} \xrightarrow{F \theta_{g,h} F^{-1}} F \rho_{hg} F^{-1} = \hat{\rho}_{gh}
\]
where $\e$ is the counit of the adjunction $F^{-1} \dashv F$. This composition is, by construction, a natural isomorphism.
The cocycle condition~$(\ref{cocycle})$ follows immediately. Moreover, $F$ becomes a $G$-functor by the definition of the action on $\mathcal{\hat{D}}$. By Corollary~\ref{G functor equivalence}, the functor $F^G$ is an equivalence.
\end{proof}

%%%%%%%%%%%%%%%%%%%%%%%%%%%%%%%%%%%%%%%%%%%%%%%%%%%%%%%%%%%
%\subsection{Forgetful and Induction Functors} 
%%%%%%%%%%%%%%%%%%%%%%%%%%%%%%%%%%%%%%%%%%%%%%%%%%%%%%%%%%%

There is a useful bi-adjoint pair of functors between the categories $\D$ and $\D^G$, namely the \textbf{forgetful} functor denoted by $\mathsf{For}$ and the \textbf{induction} functor denoted by $\mathsf{Ind}$.
The forgetful functor $\mathsf{For} \colon  \D^G \rightarrow \D$ is defined by forgetting the linearizations of an equivariant object, that is, $\mathsf{For}(X, \chi) = X$. Moreover, the induction functor $\mathsf{Ind} \colon  \D \rightarrow \D^G$ is defined by $\mathsf{Ind}(X) = (\bigoplus_{h \in G}  X^h, \phi ) $, where $\{\phi_g\colon  \bigoplus_{h\in G}  X^h \rightarrow \bigoplus_{h \in G} (X^h)^g\}_{g\in G} $ is the collection of isomorphisms $\theta_{g,h}^{-1} \colon X^{hg} \rightarrow ( X^h)^g $.
For a proof of the adjuction see \cite[Lemma~3.8]{Elagin}. Given a $G$-functor $(F,\sigma)\colon (\D,\rho, \theta) \rightarrow (\hat{\D},\hat{\rho},\hat{\theta})$, then by Lemma~\ref{gfunctorinducesequivariant}, it induces an equivariant functor $F^G\colon \D^G \rightarrow \hat{\D}^G$ that commutes with the forgetful functor, that is, the following diagram is commutative:
\begin{equation}\label{comdiagramFor}
\begin{tikzcd}
\D^G \arrow[r, "F^G"] \arrow[d, "\mathsf{For}"'] & \hat{\D}^G \arrow[d, "\mathsf{For}"] \\
            \D \arrow[r, "F"] & \hat{\D}
\end{tikzcd}
\end{equation}
The equivariant functor commutes with the induction functor up to natural isomorphism $\oplus \sigma_g \colon  F^G \circ \mathsf{Ind} \xrightarrow[]{\simeq } \mathsf{Ind} \circ F$.  When $\C$ is a $G$-invariant subcategory of $\D$, we have an induced action of $G$ on $\C$ and the equivariant category $\C^G$ can be naturally identified with the full subcategory $\mathsf{For}^{-1}(\C) $ of $\D^G$.

%%%%%%%%%%%%%%%%%%%%%%%%%%%%%%%%%%%%%%%%%%%%%%%%%%%%%%%%%%%
\subsection{Equivariant Category of Modules}\label{skewgroupalgebra} 
%%%%%%%%%%%%%%%%%%%%%%%%%%%%%%%%%%%%%%%%%%%%%%%%%%%%%%%%%%%

In this subsection we discuss the equivariant module category. Let $\Mod R$ denote the category of all right $R$-modules of a unital and associative ring $R$. For $m \in M$ and $r \in R$ we write $m \cdot r$ for the $R$-module action. Consider a finite group $G$ and a group homomorphism $\rho \colon G \rightarrow \mathsf{Aut}(R)^{op}$, where the multiplication of $\mathsf{Aut}^{op}$ is defined by $g \cdot_{op} h \coloneqq hg$. We also write $\rho(g)(r) \coloneqq r^g$ for $g \in G$ and $r \in R$. The opposite multiplication implies that, for $g,h \in G$:
\[
(r^h)^g  = \rho(g)\big(\rho(h)(r)\big) =\big(\rho(g) \cdot_{op} \rho(h)\big)(r) = \big(\rho(h) \rho(g) \big) (r) = \rho(hg)(r) = r^{hg}
\] 
hence, this is a right $G$-action on $R$.
Also, for $g \in G$ and $M \in  \Mod{R}$, the \textbf{twisted module} $M^g$ is defined such that $M^g = M$ as an abelian group but with a different $R$-action: $m \bullet_g r \coloneqq m \cdot r^{g^{-1}}$ for $m \in M$ and $r \in R$.
This yields an automorphism $(-)^g \colon  \Mod{R} \rightarrow \Mod{R}$ sending a module $M$ to $M^g$ and a morphism $f$ to itself. We have the data of a strict right $G$-action on $\Mod{R}$. Indeed $(M^h)^g = M^{hg}$; we use the notation $(\bullet_h)_g$ for the multiplication of $(M^h)^g$, since it is twisted twice, and observe that 
\[
m (\bullet_h)_g r = m \bullet_h r^{g^{-1}} = m \cdot (r^{g^{-1}})^{h^{-1}} = m \cdot r^{g^{-1}h^{-1}} = m \cdot r^{(hg)^{-1}} = m \bullet_{hg} r
\]
For this action, we say $G$ \textbf{acts on $\Mod R$ by automorphisms}.

A \textbf{skew group algebra} $RG$ (or twisted group ring, sometimes denoted by $R\#G$) is the free left $R$-module with elements in $G$ as a basis (i.e.\ formal sums $\sum_{g\in G} r_g g$) and multiplication defined by $(rg)(r' g') \coloneqq r r'^{g} gg'$ for $r,r' \in R$ and $g, g' \in G$. The skew group algebra is not the usual group algebra, which is denoted by $R[G]$, unless $G$ acts trivially on $R$. Moreover, $r \mapsto r 1_G$ is a ring homomorphism but $RG$ is not an $R$-algebra unless the action is trivial. It is canonically an $R^G$-algebra by the composition $R^G \hookrightarrow R \to RG$, where $R^G$ denotes the ring of invariants.  

A right $RG$-module $M$ is a right $R$-module equipped with a right $G$-action, $m \mapsto mg$, compatible with the $R$-action, i.e.\ such that $(m g) \cdot r = (m \cdot r^g)g$.
There is a canonical isomorphism:
\begin{equation}\label{isomorphism of equivariant module categories}
    \Phi \colon  \Mod RG \xrightarrow[]{\simeq} (\Mod R)^G
\end{equation}
such that $\Phi(M) = (M, \mu)$ for an object $M \in  \Mod{RG}$, where $\mu$ is a family $\{ \mu_g \}_{g\in G}$ defined by $\mu_g\colon  M \xrightarrow[]{\simeq} M^g$, $m  \mapsto m g$ (thus $\mu_g \mu_h = \mu_{hg}$) and $\Phi(f) = f$ for a morphism $f \in \Mod RG$. Notice that $\mu_g$ is indeed an $R$-module homomorphism: 
\[
mr \mapsto (m \cdot r)g = (mg) \cdot r^{g^{-1}} = (mg)\bullet_g r
\]
The functor $\Phi$ restricts to an equivalence $\smod RG \simeq (\smod R)^G$ and, if $|G|$ is invertible in $R$, restricts also to an equivalence $(\proj R)^G \simeq \proj (RG)$. Here, we assume that the ring $R$ is noetherian, so that the category $\smod R$ of finitely generated right $R$-modules is abelian. Then, the equivariant category $(\smod R)^G$ is abelian and therefore $\smod RG$ is also abelian. For more details we refer to the classical paper of Reiten and Riedtmann \cite{RR}, see also \cite[Proposition~2.48]{Demonet}.

\begin{rem}
We summarise below some remarks on skew group algebras.
\begin{enumerate}
\item The equivalence $\Phi$ actually says that in order to have a linearization of an $R$-module $M$, it suffices to equip it with a compatible $G$-action. Moreover, the set of all linearizations of an $R$-module is in one-to-one correspondence with the different compatible $G$-actions it can be equipped with.
\item If the group action on $R$ is trivial, then every $R$-module has a linearization. Indeed, for any module $M$, if the group acts trivially ($mg = m $), then the compatibility condition always holds since $m \cdot r = (m \cdot r)g = (mg) \cdot r^{g^{-1}}$.
This implies that any $R$-module becomes naturally an $R[G]$-module.
\item For a left $G$-action on the category of $R$-modules, one would have to use a group homomorphism $G \to \mathsf{Aut}(R)$, use $m \bullet_g r = mr^g$ for the $R$-module structure of ${^g \! M}$, use the rule $gr = r^{g^{-1}}g$ for the skew group ring and use the linearizations $m \mapsto mg^{-1}$.
\end{enumerate}
\end{rem}

\begin{rem}
In view of the geometric action of Examples~\ref{exampleonsheaves1} and~\ref{exampleonsheaves2} and the algebraic action of the previous Section~\ref{skewgroupalgebra} it is natural to wonder whether the two actions are related and, if they are, how. It turns out that the affine geometric case coincides with the algebraic one. Thus one can view the geometric action as a generalization of the algebraic one or even think of the equivariant  (quasi-)coherent sheaves as sheafified modules over the skew group algebras (this is mentioned in \cite[Section~2]{BCA}).
Indeed, let $R$ be a commutative noetherian ring. Observe that automorphisms of $R$ correspond to automorphisms of the affine scheme $\mathsf{Spec}(R)$. Following the construction of pullbacks by automorphisms (see for example \cite{Ueno1, Ueno2}), the induced $G$-action by automorpshisms on $\Mod R$ coincides with the induced $G$-action by pullbacks on $\mathsf{Qcoh}(R)$ using the canonical equivalence $\Mod R \simeq \mathsf{Qcoh}(\mathsf{Spec}(R))$ (resp.\  for $\smod R \simeq \mathsf{Coh}(\mathsf{Spec}(R))$).
\end{rem}

\subsection{Action on Triangulated Categories}
\label{actionontriangcats}
In this subsection, we briefly expose some machinery on actions for triangulated categories that we will need in the sequel. For more details we refer to \cite{ChaoSun, Elagin, Ploog, XiaoWuChen2} and references therein. Throughout this subsection we consider right group actions on categories.

When a group acts on a triangulated category the action has to be compatible with the triangulated structure. We recall the following definition.

\begin{defn}\label{admissibleaction}
\textnormal{(\!\!\cite[Definition~3.1]{ChaoSun}}
A $G$-action $(\rho, \theta)$ on a triangulated category $\T$ with translation auto-equivalence $[1]$, is said to be \textbf{admissible} if each $\rho_g$ is a triangle auto-equivalence and is equipped with natural isomorphisms $\tau_{g} \colon  [1] \circ \rho_g \xrightarrow[]{\simeq} \rho_g \circ [1] $, i.e.\ the shift $[1]$ is a $G$-functor. 
\end{defn}

Note that $[1]^G$ is the translation functor of $\T^G$. We have an equality  $\mathsf{For} \circ [1]^G = [1] \circ \mathsf{For}$ and a natural isomorphism $\oplus_{g \in G} \rho_g \colon  \mathsf{Ind}\circ [1] \xrightarrow[]{\simeq} [1]^G \circ \mathsf{Ind}$.

\begin{defn}
Given an admissible $G$-action on a triangulated category $\T$, a \textbf{canonical triangulated structure} on $\T^G$ is such that the forgetful functor is triangulated.
\end{defn} 
It is well known that $\T^G$ is not always triangulated, but by work of Chao Sun \cite{ChaoSun} (which generalized Chen's work \cite{XiaoWuChen2}) we have the following result. Recall that $|G|$ is invertible in an additive category $\D$ if for any morphisms $f\colon X \to Y$ there exists a morphism $g\colon  X \to Y$ such that $f= ng$.

\begin{prop}\label{canonicalpretriangstructure}
\textnormal{(\!\!\cite[Proposition~3.3]{ChaoSun}}
 If G acts admissibly on a triangulated category $\T$ and $|G|$ is invertible in $\T$, there exists a unique canonical pre-triangulated structure on $\T^G$. In particular, exact triangles in $\T^G$ are those who are mapped to exact triangles in $\T$ via the forgetful functor. Moreover, if $\T^G$ admits a triangulated structure, then it is unique.
\end{prop}

In this case we say that $\T^G$ is \textbf{canonically triangulated}, implying the uniqueness of the structure. 
For an example of a non triangulated equivariant category where $|G|$ is non invertible, see \cite[Remark~4.6,~(2)]{XiaoWuChen2}.
The following is the triangulated analogue of Lemma~\ref{gfunctorinducesequivariant}.

\begin{lem}
\label{triangadjointequiv}
Let $G$ act admissibly on triangulated categories $\T$ and $\widehat{\T}$ with $|G|$ invertible in both $\T$ and $\hat{\T}$. Let $(F, \sigma) \colon  \T \to \widehat{\T}$ be a triangulated $G$-functor. Assume that $\T^G$ and $\widehat{\T}^G$ are canonically triangulated. Then the induced equivariant functor $F^G\colon \T^G \to \widehat{\T}^G $ is triangulated.
\end{lem}
\begin{proof}
We denote by $\hat{\rho}_g$ the auto-equivalence in $\widehat{\T}$ corresponding to $g \in G$, by $\widehat{[1]}$ the shift functor of $\widehat{\T}$ and by $\widehat{\tau}$ its $G$-functor structure.

By diagram~$(\ref{comdiagramFor})$ we have  $\mathsf{For} \circ F^G = F \circ \mathsf{For}$. It follows that $F^G$ preserves distinguished triangles. Also, the action is admissible, thus $[1]$ is a $G$-functor and we have a natural isomorphism $F \circ [1] \simeq_{\eta} \widehat{[1]} \circ F$ which is also a $G$-natural isomorphism, since the following diagram commutes:
\begin{equation*}
\begin{tikzcd}
F \circ [1] \circ \rho_g \arrow[r, "\eta \rho_g"] \arrow[d, "\sigma'"']&  \widehat{[1]} \circ F \circ g \arrow[d, "\sigma''" ] \\
\hat{\rho}_g \circ F \circ [1] \arrow[r, "\rho_g \eta"'] & \hat{\rho}_g \circ \widehat{[1]} \circ F
\end{tikzcd}
\end{equation*}
Here we used that $F \circ [1]$ and $\widehat{[1]}\circ F$ are $G$-functors with $G$-structures $\sigma'\coloneqq \sigma [1]\circ F\tau$ and $\sigma''\coloneqq  \widehat{\tau} F \circ \widehat{[1]}\sigma$, respectively, described in Remark~\ref{compositionGfunctors}.
By Lemma~\ref{equivnaturaltransf}, we have a natural isomorphism $F^G \circ [1] \simeq _{\eta^G} [1]\circ F^G$, which concludes the proof.
\end{proof}

We recall below the result of Chao Sun about actions on Verdier quotients. 

\begin{thm} \textnormal{(\!\!\cite[Theorem~3.9]{ChaoSun})}
\label{quotienttheorem}
Suppose $\T$ is a triangulated category with an admissible $G$ action, $|G|$ is invertible in $\T$ and $\T^G$ admits a canonical triangulated structure. Let $\U$ be a $G$-invariant subcategory of $\T$. Then:
\begin{itemize}
\item[(i)] The Verdier quotient $\T/ \U$ carries an admissible $G$ action and $(\T/ \U) ^G $ admits a canonical triangulated structure.

\item[(ii)] There is a natural triangle functor $\T^G / \U^G \xrightarrow{} (\T/ \U)^G$ that is an equivalence up to retracts. In particular, if $\T^G / \U^G$ is idempotent complete, then it is an equivalence.
\end{itemize}
\end{thm}

We finish this section by verifying that the quotient functor is a $G$-functor. 

\begin{rem}\label{quotientGfunctor}
With the same assumption as in Theorem~\ref{quotienttheorem}, the quotient functor $Q\colon \T \xrightarrow{} \T/ \U$ is canonically a $G$-functor.

Let $X$ be an object in $\mathcal{T}$ and $\hat{\rho}_g$ the induced auto-equivalences of the action on the quotient $\T / \U$. The image of $X$ under the quotient functor is denoted also by $X$.
The family of natural isomorphisms $\{ \sigma_g \colon Q \circ \rho_g \xrightarrow[]{\simeq} \hat{\rho}_g \circ Q \}_{g\in G} $ where $\sigma_g = \mathsf{Id}$ make the following diagram commutative:
\begin{equation*}
\begin{tikzcd}
 Q \circ \rho_g (X) = {X^g} \arrow[r, "\mathsf{Id}"] \arrow[d, " Q\circ \rho_g (f)"'] & \hat{\rho}_g \circ  Q(X) = {X^g} \arrow[d, " \hat{\rho}_g \circ Q(f)"] \\
 Q \circ \rho_g (Y) = {Y^g} \arrow[r, "\mathsf{Id}"] & \hat{\rho}_g \circ  Q(Y) = {Y^g}
\end{tikzcd}
\end{equation*}
for each $f\colon X \to Y$ in $\T$.
The horizontal maps are identities because the action on the quotient is induced by the action on $\T$. Moreover, $ Q\circ \rho_g (f) =  \hat{\rho}_g \circ  Q(f)$, since on the left hand side of the equation we have ${X^g} \xleftarrow[]{\mathsf{Id}_{X^g} } {X^g}  \xrightarrow[]{f^g} {Y^g}$ while on the right hand side we have ${X^g} \xleftarrow[]{\mathsf{Id}_X^g } {X^g} \xrightarrow[]{f^g} {Y^g}$ which are identical by $\mathsf{Id}_X^g = \mathsf{Id}_{X^g}$.
The family $\{ \sigma_g\}_{g\in G}$ obviously satisfies the associativity condition~$(\ref{associativity})$.
\end{rem}

\section{Equivariant Recollements}
\label{sectionrecolabel}

In this section we investigate finite group actions on recollements of abelian and triangulated categories. We study finite group actions induced by automorphisms of rings in the context of recollements of module categories. Throughout this section, all group actions are right actions. We write $F\colon \D \to \hat{\D}$ for a $G$ functor between two categories with $G$ actions, omitting the notations $(F,\sigma)$ for the $G$-functor and $(\D,\rho, \theta)$ for the action when the context is clear.

%%%%%%%%%%%%%%%%%%%%%%%%%%%%%%%%%%%%%%%%%%%%%%%%%%%%%%%%%%%%%%%%%%%%%%
\subsection{Epi on Linearizations}
%%%%%%%%%%%%%%%%%%%%%%%%%%%%%%%%%%%%%%%%%%%%%%%%%%%%%%%%%%%%%%%%%%%%%%
The following lemma is very useful and important for what follows, but first we introduce some new terminology.

\begin{defn}\label{epionlin}
Suppose a finite group $G$ acts on two categories $\A$ and $\B$. Let $F\colon \A \rightarrow \B$ be a $G$-functor. We say that $F^G$ is \textbf{epi on linearizations} if for all $(FX,\chi')\in \B^G$, there exists $(X, \chi) \in \A^G$ such that $F^G(X,\chi) = (FX, \chi')$.
%i.e.\ for every linearization $\psi$ of some object $FX$ in the image of $F$, we can find some linearization $\chi$ so that the image of the equivariant object $(X, \chi)$ under the equivariant functor $F^G$ is exactly $(FX,\chi')$. \textcolor{red}{sxolio...}
\end{defn}

The property of being epi on linearizations is analogous to $\Image F^G = (\Image F)^G$ when the right hand side makes sense, i.e.\ when there is an induced action on $\Image F$. 

\begin{lem}\label{epilinear} 
Suppose a finite group $G$ acts on two additive categories $\A, \B$ (which can be in particular abelian or triangulated with an admissible action) and let $F\colon \A \to \B$ be a fully faithful $G$-functor. Then, $F^G$ is epi on linearizations. 
\begin{proof}
Let $(FX, \chi') $ be an object of $ \B^G$ and denote by $\sigma$ the natural isomorphism $F \rho_g \xrightarrow{\simeq} \rho_g F$.
We need to find some isomorphism $\chi_g\colon X \to X^g$ such that the composition $FX \xrightarrow{F \chi_g} F(X^g) \xrightarrow{\sigma^X_g} (FX)^g$ is equal to $\chi'_g$. Notice that $(\sigma^X_g)^{-1} \chi'_g \in \mathsf{Iso}(FX, F(X^g))$ and since $F_{X,X^g}$ induces a bijection of the isomorphism subgroups, there exists $\chi_g \in \mathsf{Iso}(X, X^g)$ such that $F(\chi_g) = (\sigma_g^X)^{-1}\chi'_g$. Now we have to show that the family of isomorphisms $\{\chi_g\}_{g \in G}$ satisfies the cocycle condition in order to be a linearization of $X$, then we have by construction that $F^G(X,\chi) = (FX, \chi')$.

Indeed, it inherits this property from $\chi'_g$. Consider the following diagram which illustrates the construction of $\chi_g$: 

 \begin{equation*}
        \begin{tikzcd}
            X \arrow[rr, "\chi_g"'] \arrow[rrrrrr, bend left=20, "\chi_{hg}"] && X^g  \arrow[rr, "(\chi_h)^g"']  && (X^h)^g \arrow[rr, "\theta_{g,h}^X"']  && X^{hg} \\
            && F(X^g) \arrow[dd, "\sigma^X_g"'] \arrow[rr, dashrightarrow, "F((
            \chi_h)^g)"]		   && F ((X^h)^g) \arrow[d, "\sigma^{X^h}_g"] \arrow[rr, "F \theta^X_{g,h}"] && F(X^{gh}) \arrow[dd, "\sigma^X_{hg}"]	\\
            &&	&& (F (X^h))^g \arrow[d, "(\sigma^X_h)^g"]  \\
            FX \arrow[rr, "\chi'_g"] \arrow[rruu, "F(\chi_g)=(\sigma^X_g)^{-1}\chi'_g"] \arrow[rrrrrr, bend right=20, "\chi'_{hg}"']  && (FX)^g \arrow[rr, "(\chi'_h)^g"']  \arrow[rru, "( (\sigma^X_h)^{-1}\chi'_h)^g" description] && ((FX)^h)^g  \arrow[rr, "(\theta'_{g,h})^{FX}"] && (FX)^{hg}
        \end{tikzcd}
    \end{equation*}
The above diagram commutes. Indeed, the left triangle commutes by the definition of the linearization $\chi'_g$. Similarly, the middle triangle is obtained by applying $\rho_g$ to the underlying diagram which is the definition of $\chi'_h$. Note that $( (\sigma^X_h)^{-1}\chi'_h)^g = (F\chi_h)^g$. Consequently, the dashing arrow is $ (\sigma^{X^h}_g)^{-1} \circ (F\chi_h)^g \circ \sigma_g^X = F( (\chi_h)^g)$ obtained by applying the natural transformation $\sigma_g$ on the isomorphism $\chi_h$. The commutativity of the right square is precisely the condition that $F$ is a $G$-functor, applied on the object $X$. The lower bended diagram is the cocycle condition of $\chi'$.

Notice that $F\theta_{g,h}^X$ lies in $\mathsf{Iso}\big(F((X^{h})^g), F(X^{hg}) \big)$, hence, corresponds through $F$ to the isomorphism $\theta_{g,h}^X \in \mathsf{Iso}((X^h)^g, X^{hg})$. We also have:
\begin{align*}
F(\chi_{hg}) = (\sigma_{hg}^X)^{-1} \chi'_{hg} = (\sigma_{hg}^X)^{-1} \theta'_{g,h} \circ (\chi'_h)^g \circ \chi'_g & = F(\theta_{g,h}^X) \circ F((\chi_h)^g) \circ F(\chi_g) \\
& = F(\theta_{g,h} \circ (\chi_h)^g \circ \chi_g)
\end{align*}
and therefore $\chi_{hg} = \theta^X_{g,h} \circ (\chi_h)^g \circ \chi_g$ by the bijection of the isomorphism groups.
\end{proof}
\end{lem}

\begin{rem}\label{subcatlinear}
We remark that we apply the above Lemma when $\A$ is a $G$-invariant subcategory of $\B$. The inclusion $\iota \colon \A \hookrightarrow B$ is epi on linearizations. In particular, any object $(X,\chi)$ of $\B^G$ is an object of $\A^G$ if and only if $X$ lies in $\A$.
\end{rem}

\subsection{Recollement of Abelian Categories}\label{recollementsofabeliancats}

In this section we examine when a recollement $\mathsf{R}_{\mathsf{ab}}(\A,\B,\C)$ of abelian categories induces an equivariant recollement $\mathsf{R}_{\mathsf{ab}}(\A^G,\B^G,\C^G)$ of abelian categories. Recall that a recollement of abelian categories is a diagram of the form
\[
\begin{tikzcd}
 \A \arrow[rr, "\mathsf{i}" description]& & \B \arrow[ll, bend left, "\mathsf{p}"] \arrow[ll, bend right, "\mathsf{q}"'] \arrow[rr, "\mathsf{e}" description] & & \C \arrow[ll, bend left, "r"] \arrow[ll, bend right, "l"']
\end{tikzcd}\eqno \mathsf{R}_{\mathsf{ab}}(\A,\B,\C)
\]
with $(\mathsf{l,e,r})$ and $(\mathsf{q,i,p})$ adjoint triples, $\mathsf{i,l,r}$ fully faithful functors and $\Image (\mathsf{i}) =\Ker (\mathsf{e})$. Recollements of triangulated categories were introduced in  \cite{BBD} and were studied later also in the case of abelian categories. For an overview of abelian recollements and their applications to representation theory, we refer to \cite{Psaroud:survey}.

We show below how to construct a recollement of abelian equivariant categories.

\begin{prop}\label{prop1}
Let $\mathsf{R}_{\mathsf{ab}}(\A,\B,\C)$ be an abelian recollement. Let $G$ be a finite group acting on $\A$, $\B$ and $\C$. If one of the functors of each of the adjoint triples in $\mathsf{R}_{\mathsf{ab}}(\A,\B,\C)$  is a $G$-functor, then all three are $G$-functors and we have a recollement of equivariant abelian categories:
\begin{equation*}
\begin{tikzcd}
 \A^G \arrow[rr, "\mathsf{i}^G" description]& & \B^G \arrow[ll, bend left, "\mathsf{p}^G"] \arrow[ll, bend right, "\mathsf{q}^G"'] \arrow[rr, "\mathsf{e}^G" description] & & \C^G \arrow[ll, bend left, "\mathsf{r}^G"] \arrow[ll, bend right, "\mathsf{l}^G"']
\end{tikzcd}\eqno \mathsf{R}_{\mathsf{ab}}(\A^G,\B^G,\C^G)
\end{equation*}
\end{prop}
\begin{proof}
By Lemma~\ref{adjointequiv} an adjoint triple $L \dashv F \dashv G $ lifts to an equivariant adjoint triple if (at least) one of these functors is a $G$-functor. So, for $(\mathsf{q}^G, \mathsf{i}^G, \mathsf{p}^G)$ and $(\mathsf{l}^G,\mathsf{e}^G, \mathsf{r}^G)$ to be equivariant triples it suffices that one functor of each of these adjoint triples $(\mathsf{l,e,r})$ and $(\mathsf{q, i,p})$ is a $G$-functor. The condition that $\mathsf{l}^G, \mathsf{r}^G, \mathsf{i}^G$ are fully faithful is also automatically fulfilled by Lemma~\ref{ffequiv}.

It remains to prove that $\Image (\mathsf{i}^G) = \Ker (\mathsf{e}^G)$. It suffices to show the equality on objects of these subcategories.
For an object $(X,\chi)$ in $\B^G$, observe that it lies in $\Ker(\mathsf{e}^G)$ if and only if $X$ lies in $\Ker(\mathsf{e})$. This is easy to prove using the uniqueness of the zero object, see Remark~\ref{0-object-equiv}. Now, the inclusion $\Image (\mathsf{i}^G) \subseteq \Ker (\mathsf{e}^G)$ follows immediately from $\Image(\mathsf{i}) = \Ker(\mathsf{e})$. For the other inclusion, we have to show that for an object of the form $(\mathsf{i}X, \chi') \in\Ker (\mathsf{e}^G)$ with $\mathsf{i}X \in \Ker (\mathsf{e}) = \Image (\mathsf{i})$ for some $X \in \A$, there exists some linearization $\chi$ of $X$, equivalently an equivariant object $(X,\chi) \in \A^G$ such that $\mathsf{i}^G (X, \chi) = (\mathsf{i}X,\chi')$. This follows from $\mathsf{i}$ being epi on linearizations by Lemma~\ref{epilinear}.
\end{proof}

Since $\A$ can be interpreted as a full subcategory of $\B$, it is natural to ask whether the action on $\B$ induces an action on $\A$. It turns out that this is closely related to $\mathsf{e}$ being a $G$-functor.

\begin{prop}\label{actiontosubcategory}

Let $\mathsf{R}_{\mathsf{ab}}(\A,\B,\C)$ be an abelian recollement. Let $G$ be a finite group acting on $\B$ and $\C$. Assume also that one of the functors $\mathsf{l}, \mathsf{e}$ or $\mathsf{r}$ is a $G$-functor. Then we have an induced action on $\A$ and $\mathsf{R}_{\mathsf{ab}}(\A^G, \B^G, \C^G)$ is a recollement of abelian categories.
\end{prop}
\begin{proof}
Since $F\colon \A \xrightarrow[]{\simeq }  \mathsf{Im}(\mathsf{i}) = \mathsf{Ker}(\mathsf{e})$, if we show that $\mathsf{Ker}(\mathsf{e})$ is a $G$-invariant subcategory of $\B$, then $\A$ inherits a $G$-action by the equivalence $F$, by Lemma~\ref{inducedequivariantequivalence}. 
Note that $F$ is a $G$-functor by construction.
Moreover, the inclusion $\iota \colon \mathsf{Ker}(\mathsf{e}) \hookrightarrow \B$ is a $G$-functor by Remark~\ref{inclusionfunctor}, hence $\mathsf{i}= \iota \circ F$ becomes a $G$-functor being the composition of $G$-functor. This is also fully faithful being a composition of fully faithful functors.
Then, by Proposition~\ref{prop1}, we conclude that $\mathsf{R}_{\mathsf{ab}}(\A^G, \B^G, \C^G)$ is a recollement of abelian categories.

Since either of $ \mathsf{l,e,r}$ is a $G$-functor, then all are $G$-functors. Since $\mathsf{e}$ is a $G$-functor, it follows that $\Ker (\mathsf{e})$ is a $G$-invariant subcategory. Indeed, if $X$ lies in $\Ker (\mathsf{e})$, then $X^g$ lies in $\Ker(\mathsf{e})$ if and only if $\mathsf{e}(X^g) =0$, which holds since $\mathsf{e}(X^g) \simeq  (\mathsf{e}X)^g=0$.
\end{proof}

%The other direction holds also true, but first a small remark.

\begin{rem}\label{SerreSubcatEquivariant}
Recall that a Serre subcategory $\A$ of an abelian category $\B$, is a  full subcategory, closed under subobjects, quotients and extensions. In particular, in a recollement $\mathsf{R_{ab}}(\A, \B, \C)$, the subcategory $\A$ is Serre. In the above Proposition we showed that under certain assumptions on a $G$-action on $\B$, then $\A^G$ is a Serre subcategory of $\B^G$. However, this holds in general whenever $\A$ is a $G$-invariant Serre subcategory of an abelian category $\B$.
For the closed under extensions condition, we must check that given a short exact sequence where the outer terms lie in $\A^G$, the middle term also also lies in $\A^G$:
$$ 0 \rightarrow (X, \chi) \rightarrow (Y, \psi) \rightarrow (Z, \zeta) \xrightarrow{} 0 $$
Using the forgetful functor, which is exact, we obtain a short exact sequence where the outer terms are in $\A$ and therefore, since it is a Serre subcategory, the middle term $Y$ is also in $\A$. Thus, $(Y,\psi)$ lies in $\A^G$, by Remark~\ref{subcatlinear}.
Similarly we can show that $\A^G$ is closed under subobjects and quotients.
\end{rem}

The above remark shows that the Gabriel quotient $\B^G / \A^G$ exists and motivates the following result.

\begin{prop}\label{actiontoquotient}
Let $\mathsf{R}_{\mathsf{ab}}(\A,\B,\C)$ be an abelian recollement.  Let $G$ be a finite group acting on $\B$ such that $\A$ is a $G$-invariant subcategory of $\B$. Then there is an induced action on $\C$ rendering $\mathsf{e}$ a $G$-functor. Then $\mathsf{R}_{\mathsf{ab}}(\A^G,\B^G,\C^G)$ is a recollement of abelian categories.
%\begin{equation*}
%\begin{tikzcd}
 %\A \arrow[rr, "i" description]& & \B \arrow[ll, bend left, "p"] \arrow[ll, bend right, "q"'] \arrow[rr, "e" description] & & \C \arrow[ll, bend left, "r"] \arrow[ll, bend right, "l"']
%\end{tikzcd}\eqno \mathsf{R}_{\mathsf{ab}}(\A,\B,\C)
%\end{equation*}
%and a finite group $G$ acts on $\B$ with $\A$ being G-invariant (Serre) subcategory of $\B$. Then there is an induced action on $\C$ rendering $e$ a $G$-functor. Then $\mathsf{R}_{\mathsf{ab}}(\A^G,\B^G,\C^G)$ is a recollement of abelian categories.
%\par
%Moreover, $\B^G / \A^G \simeq (\B/\A)^G$.
\end{prop}

\begin{proof}
Since $\A$ is a $G$-invariant subcategory, we have that $\mathsf{i}$ is a $G$-functor (see Remark~\ref{inclusionfunctor}). Denote by $Q \colon \B \to \B/\A$ the quotient functor. Recall that there exists an equivalence $F \colon \B / \A \xrightarrow{\simeq} \C$ such that $\mathsf{e}=F\circ Q$. If we show that there is an induced action on $\B/\A$ such that the quotient functor is a $G$-functor, then we have that $\C$ inherits this action through $F$ in a natural way. Moreover, $F$ is a $G$-functor by Lemma~\ref{inducedequivariantequivalence} and hence $\mathsf{e} = F \circ Q$ becomes also a $G$-functor as a composition of two $G$-functors. Then by Proposition~\ref{prop1} the result follows.

The quotient functor $Q \colon \B \to \B/\A$ is identity on objects, thus for the functor $\rho_g\colon \B/\A \to \B / \A$ on objects we define $B^g =  Q(B)^g = Q(B^g)$.
Now, for morphisms of $\B/\A$ we define that $\rho_g$ maps a morphism $f$ to $f^g$, as it would in $\B$. In this way, it preserves the isomorphisms in $\B/\A$ and thus $\rho_g$ is well defined. Indeed, an isomorphisms $f\colon X \rightarrow Y$ in $\B /\A$ is a morphism in $\B$ such that $\Ker f,\ \Image f \in \A$. That is an exact sequence of the form:
$$0 \rightarrow \Ker f \rightarrow X \xrightarrow[]{f} Y \rightarrow \Image f \rightarrow 0$$
Now applying the exact functor $\rho_g$, we obtain the exact sequence
$$0 \rightarrow \Ker (f)^g \rightarrow X^g \xrightarrow[]{f^g} Y^g \rightarrow \Image (f)^g \rightarrow 0$$
Since $\A$ is $G$-invariant, we also have that $\Ker(f)^g,\ \Image(f)^g \in \A$, so $f^g\colon  X^g \rightarrow Y^g$ is an isomorphism in $\B/\A$. Notice that $\rho_g$ is an auto-equivalence of $\B/\A$ since any auto-equivalence of $\B$ that restricts to an auto-equivalence of $\A$ yields an auto-equivalence of the quotient (this a consequence of \cite[Corollary~III.1.2]{Gabriel}).

The composition natural isomorphism $\theta'$ for the action on $\B/\A$ is induced by the corresponding one of the action on $\B$ by composing it with $Q$, i.e.\ $\theta'_{g,h} = Q \theta_{g,h}$, for all $g,h \in G$. To have a clearer picture of this, for some $X \xrightarrow[]{f} Y$ in $\B /\A$ we have that:
\begin{equation*}
\begin{tikzcd}
 (X^h)^g = Q((X^h)^g) \arrow[d, "Q((f^h)^g)"'] \arrow[r, "Q \circ \theta_{g,h}^X", "\simeq"'] & Q(X^{hg})= X^{hg} \arrow[d, "Q(f^{hg})"] \\
(Y^h)^g = Q((Y^h)^g) \arrow[r, "Q \circ \theta_{g,h}^Y", "\simeq"'] & Q(Y^{hg})= Y^{hg}
\end{tikzcd}
\end{equation*}
which is commutative, since we are applying $Q$ to the underlying commutative diagram which is due to the action on $\B$. The cocycle condition~$(\ref{cocycle})$ for this action immediately holds. Note that the definition of the action renders $Q$ a $G$-functor in a canonical way (the family of 2-isomorphisms $\sigma^Q$ is simply the identity by the definition of the action).
\end{proof}

\begin{rem}\label{quotientequivalenceofequivariantcats}
The above proof, in fact, shows that there exists an equivalence of categories $\B^G/\A^G \simeq (\B/\A)^G$. This fact was proved by Chen, Chen and Zhou (see \cite[Corollary~4.4]{CCZ}) using monadic techniques. The main difference is that the above proof relies on intrinsic tools related to the equivariant category, in particular, the epi on linearizations property that fully faithful $G$-functors admit. Thus, with this direct approach, we avoid the machinery of (co)monads.
%on the one hand requires less knowledge to be used, and on the other hand gives us a better understanding of the subject. We think that similar proofs, regarding equivariant categories, that use monadic techniques can be also provided by studying directly such properties of the equivariant categories.
\end{rem}

We finish this section by combining the previous results into one theorem.

\begin{thm} 
\label{main1}
    Let $\mathsf{R}_{\mathsf{ab}}(\A,\B,\C)$ be an abelian recollement. 
%\begin{equation*}
%\begin{tikzcd}
 %\A \arrow[rr, "i" description]& & \B \arrow[ll, bend left, "p"] \arrow[ll, bend right, "q"'] \arrow[rr, "e" description] & & \C \arrow[ll, bend left, "r"] \arrow[ll, bend right, "l"']
%\end{tikzcd}\eqno \mathsf{R}_{\mathsf{ab}}(\A,\B,\C)
%\end{equation*}
Let $G$ be a finite group acting on $\B$. Then the following are equivalent:
\begin{itemize}
\item[(i)] $G$ acts on $\C\simeq \B/\A$ and $\mathsf{e}$ is a $G$-functor. 
\item[(ii)] $\A$ is a $G$-invariant subcategory of $\B$.
\end{itemize}
If either of the above conditions holds true, then $\mathsf{R}_{\mathsf{ab}}(\A^G, \B^G, \C^G)$ is a recollement of abelian categories.
\end{thm}
\begin{proof}
The implication (i) $\Longrightarrow$ (ii) is Proposition~\ref{actiontosubcategory} while the converse implication (ii) $\Longrightarrow$ (i) is Proposition~\ref{actiontoquotient}.
\end{proof}

Motivated by this we introduce the following notion.

\begin{defn}
\label{equivariantrecollement}
A recollement of abelian categories $\mathsf{R}_{\mathsf{ab}}(\A,\B,\C)$ \textbf{lifts to a} $G$-\textbf{equivariant recollement} $\mathsf{R}_{\mathsf{ab}}(\A^G,\B^G,\C^G)$ if a  group $G$ acts on $\B$ such that it satisfies either of the equivalent conditions of Theorem~\ref{main1}.
\end{defn}

%%%%%%%%%%%%%%%%%%%%%%%%%%%%%%%%%%%%%%%%%%%%%%%%%%%%%%%%%%%
\subsection{Recollement Induced by an Idempotent}\label{examplewithidempotents} 
%%%%%%%%%%%%%%%%%%%%%%%%%%%%%%%%%%%%%%%%%%%%%%%%%%%%%%%%%%%
In this subsection we introduce a direct application of Theorem~\ref{main1} on the well known example about recollements induced by idempotent elements of rings.

Let $R$ be a ring and $e$ an idempotent element in $R$. We denote by $\Mod R$ the category of right $R$-modules. This induces a functor $\mathsf{e} \coloneqq (-)e\colon  \Mod R \xrightarrow[]{} \Mod eRe$, which is multiplication from the right side by $e$. This functor is isomorphic to the functors $\Hom_R (eR, -) \simeq eR \otimes_R-$ which have fully faithful left and right adjoints, forming the following adjoint triple: 
\[
- \otimes_{eRe}eR \ \dashv \  \Hom_R(eR, -)\simeq -\otimes_R Re\ \dashv \  \Hom_{eRe}(Re,-) 
\]
The kernel of $(-)e$ is the category $\Mod R/ReR$. This yields a recollement of module categories:
\begin{equation} \label{modulerecollement}
\begin{tikzcd}
\Mod R/ReR \arrow[rr, "\mathsf{inc}"]& & \Mod R \arrow[ll, bend left, "\Hom_R(R/eRe \text{,}-)"] \arrow[ll, bend right, "- \otimes_R R/ReR"'] \arrow[rr, "\mathsf{e}"] & & \Mod eRe \arrow[ll, bend left, "\Hom_{eRe}(Re\text{,}-)"] \arrow[ll, bend right, "- \otimes_{eRe}eR"']
\end{tikzcd}
\end{equation}
which is said to be \textbf{induced by the idempotent} $e$.

Given a finite group $G$ and a group homomorphism $\rho\colon  G \rightarrow \mathsf{Aut}(R)^{op}$, there is and induced (right) strict action on $\Mod R$ as described in subsection~\ref{skewgroupalgebra}. We prove below that $G$ acts by automorphisms on the outer terms of the recollement and therefore we can construct a recollement of equivariant module categories.

\begin{prop}
Let $R$ be a ring and $e$ an idempotent element of $R$. Let  $G$ be a finite group acting by automorphisms on $R$ such that $e \in R^G$. Then $G$ acts by automorphisms on $\Mod eRe$ and on $\Mod R/ReR$. Moreover, $\mathsf{e}$ is a strict $G$-functor and the recollement induced by the idempotent $e$ lifts to a $G$-equivariant recollement $\mathsf{R_{ab}} ( (\Mod R/ReR)^G, (\Mod R)^G, (\Mod eRe)^G )$.
\begin{proof}
The assumpion that $e$ is an element of $R^G$ implies that a group homomorphism $\rho\colon  G \rightarrow \mathsf{Aut}(R)^{op}$ induces a group homomorphism $\rho\colon  G \rightarrow \mathsf{Aut}(eRe)^{op}$ since we have $(ere)^g = er^ge$ which is an automorphism of $eRe$. Thus, we have an induced action by automorphisms of $eRe$ on $\Mod eRe$. Note that $(Me)^g = M^g e$, since $m \bullet_g e = m \cdot e^{g^{-1}} = m \cdot e $. Moreover, the equality $(Me)^g = M^ge$, yields that the functor $(-)e$ is a strict $G$-functor in the sense that the family of natural isomorphisms needed to define a $G$-functor are identities.

By Theorem~\ref{main1}, $G$ acting on $\Mod eRe$ such that $\mathsf{e}=(-)e$ is a $G$-functor is equivalent to $\Mod R/ReR$ being $G$-invariant. Obviously, the induced action on this subcategory is also strict. Note that this action on $\Mod R/ReR$ is also induced by automorphisms. Indeed, the homomorphism  $G \to \mathsf{Aut}(R/ReR)^{op} $ is well defined, since if  $r_1-r_2=r'er'' \in ReR$, then ${r_1}^g - {r_2}^g = (r_1-r_2)^g = (r'er'')^g=(r')^ge(r'')^g$, which lies in $ReR$.
\end{proof}
\end{prop}

\begin{rem}
If we assume that $e$ is not $G$-invariant, then $\Ker (e) = \Mod R/ReR$ is not $G$-invariant either. 
In fact, the action is not well defined since two elements $r \neq r'$ such that $r+ReR = r'+ReR  \in R/ReR$ might yield $r^g + ReR \neq r'^g+ReR$. Hence, even for $R/ReR$ as a module over itself, the $(R/ReR)$-module structure on $(R/ReR)^g$ is not well defined.
This is due to the fact that an automorphisms of $R$ does not induce an automorphism of $R/ReR$ unless it preserves $e$.
\end{rem}

Since we have a $G$-equivariant recollement of categories of modules, it is natural to compare it to the recollement of the categories of modules of the skew group algebras. It turns out that this recollement is also induced by the corresponding idempotent.

\begin{prop}\label{equivariant recollement induced by an idempotent}
Let $R$ be a ring and $e$ an idempotent element of $R$. Let $G$ be a finite group acting by automorphisms on $R$ such that $e \in R^G$. Then $G$ acts by automorphisms on $\Mod R$ and on $\Mod eRe$ such that there exists an equivariant recollement. Moreover, the corresponding recollement of the skew group algebras is also induced by the idempotent $e' \coloneqq e1_G$.
\begin{proof}
Recall from equation~$(\ref{isomorphism of equivariant module categories})$ that $\Phi \colon \Mod RG \to (\Mod R)^G $ is a canonical isomorphism of categories. We denote, also, by $\Phi $ the isomorphism $\Mod (eRe)G \to (\Mod eRe)^G$.
Observe that the skew group algebra $(eRe)G$ is isomorphic to $e'RGe'$, thus we can identify their categories of modules. Denote by $\mathsf{e'} \coloneqq (-)e'$.
We have the following square
\begin{equation*}
\begin{tikzcd}
\Mod RG \arrow[r, "\mathsf{e'}"] \arrow[d, "\Phi"'] & \Mod e'RGe' \\
(\Mod R)^G \arrow[r, "{\mathsf{e}}^G"'] &  (\Mod eRe)^G \arrow[u, "\Phi^{-1}"']
\end{tikzcd}
\end{equation*}
which is commutative. Indeed,
\[ 
\Phi ^{-1} \circ \mathsf{e}^G \circ \Phi M = \Phi^{-1} \circ \mathsf{e}^G (M,\mu) = \Phi^{-1}(Me, \mu') = M e1_G = (M)e' 
\]
where the families of linearizations are
\begin{itemize}
\item[(i)] $\mu_g\colon  M \xrightarrow[]{\simeq} M^g$ given by $m \mapsto mg$, and
\item[(ii)] $\mu'_g\colon  Me \xrightarrow[]{\simeq} M^g e \xrightarrow{=} (Me)^g $ (since $e$ induces a strict $G$-functor) given by $m\cdot e \mapsto mg \cdot e = (m\cdot e^{g})g=(m\cdot e)g$ (since $e$ is $G$-invariant).
\end{itemize}
We infer that we have the following recollement
\begin{equation} \label{equivariantmodulerecollement}
\begin{tikzcd}
\Mod RG/RG e' RG \arrow[rr, "\mathsf{inc}"]& &  \Mod RG \arrow[ll, bend left, "\Hom_{RG}(RG/e'RGe'\text{,}-)"] \arrow[ll, bend right, "- \otimes_{RG} RG/RGe'RG"'] \arrow[rr, "\mathsf{e'}"] & &  \Mod e'RG e' \arrow[ll, bend left, "\Hom_{e'RGe'}(RGe'\text{,}-)"] \arrow[ll, bend right, "- \otimes_{e'RGe'}e'RG"']
\end{tikzcd}
\end{equation}
which is equivalent through the canonical isomorphisms $\Phi$ of equation~$(\ref{isomorphism of equivariant module categories})$ to the $G$-equivariant recollement $\mathsf{R_{ab}} ( (\Mod R/ReR)^G, (\Mod R)^G, (\Mod eRe)^G )$.
\end{proof}
\end{prop}

\subsection{Recollement of Triangulated Categories}
\label{recollementsoftriangulatedcats}
In this section we examine the conditions under which a recollement $\mathsf{R}_{\mathsf{tr}}(\U,\T,\V)$ of triangulated categories induces an equivariant recollement $\mathsf{R}_{\mathsf{tr}}(\U^G,\T^G,\V^G)$ of triangulated categories. Note that this has been proved for semi-orthogonal decompositions in \cite{ChaoSun}. We provide a proof for recollements of triangulated categories for completeness.

\begin{rem}\label{invertibilityonsubcats}
If $G$ is a finite group acting on all three categories of a recollement $\mathsf{R}_{\mathsf{tr}}(\U,\T,\V)$ and if $|G|$ is invertible in $\T$, then it is invertible in $\U$ and $\V$.
    
Indeed, using the fully faithful functors of the recollement this is immediate. For example, $|G|$ is invertible in $\U$ because for any $f $ in $\Hom_{\U}(X,Y)$ we have $i(f)$ lies in $ \Hom_{\T}(\mathsf{i}X,\mathsf{i}Y)$. Then, by invertibility in $\T$ there exists some $f'\colon \mathsf{i}X \to \mathsf{i}Y$ such that $\mathsf{i}(f)=nf'$. Then again by fully faithfulness of $\mathsf{i}$ we have that $f = n \mathsf{i}^{-1}(f')$.
\end{rem}

\begin{rem}\label{identificationofTRIANGsubcategories}
Using Lemma~\ref{triangadjointequiv}, we have canonical identifications in recollements of triangulated categories. That is, if $F \colon  \D \to \hat{\D}$ is a triangulated equivalence and $G$ is a finite group acting admissibly on $\D$, then there is an admissible $G$-action on $ \hat{\D}$ rendering $F$ a triangulated $G$-functor. Moreover, if the equivariant categories are canonically triangulated, then $F^G$ is triangulated. Then, in a recollement of triangulated categories $\mathsf{R}_{\mathsf{tr}}(\U,\T,\V)$ we can identify $\U$ with the full subcategory $\Ker (\mathsf{e})$ of $\T$ and $\V$ with the Verdier quotient $\T/\U$.
\end{rem}

 \begin{rem}\label{trianequivrem}
In general, the equivariant category of a triangulated category is not always triangulated; see subsection~\ref{actionontriangcats}.
A sufficient condition for $\T^G$ to be canonically triangulated is that $|G|$ is invertible in $\T$ and that $\T$ admits a dg-enhancement  \cite[Theorem~6.9, Corollary~6.10]{Elagin}. If $\mathsf{R_{tr}}(\U, \T, \V)$ is a recollement of triangulated categories, then by  \cite[Remark~3.2]{SC} we know that if $\T$ admits a dg-enhancement then also $\U$ and $\V$ admit a dg-enhancement (respectively as a full subcategory and as the Verdier quotient). Thus, if a finite group $G$ acts on $\U$, $\T$ and $\V$ with $|G|$ invertible in $\T$ (hence in all three categories by Remark~\ref{invertibilityonsubcats}), then the equivariant categories $\U^G$, $\T^G$ and $\V^G$ are canonically triangulated. 
\end{rem}
Without the assumption on the existence of a dg-enhancement for $\T$, we don't know in general if $\T^G$ being triangulated implies that $\U^G$ and $\V^G$ are triangulated. 

\begin{prop}\label{prop2}
Let  $\mathsf{R}_{\mathsf{tr}}(\U,\T,\V)$ be a triangulated recollement. Let $G$ be a finite group acting admissibly on $\U$, $\T$ and $\V$. Assume also that $|G|$ is invertible in $\T$ and that $\T^G$ admits the canonical triangulated structure. If one of the functors of each of the adjoint triples in  $\mathsf{R}_{\mathsf{tr}}(\U,\T,\V)$ is a $G$-functor, then all three are $G$-functors, $\U^G,\V^G$ admit canonical triangulated structures and we have a recollement of equivariant triangulated categories:
\begin{equation*}
\begin{tikzcd}
 \U ^G \arrow[rr, "\mathsf{i}^G" description]& & \T ^G \arrow[ll, bend left, "\mathsf{p}^G"] \arrow[ll, bend right, "\mathsf{q}^G"'] \arrow[rr, "\mathsf{e}^G" description] & & \V ^G \arrow[ll, bend left, "\mathsf{r}^G"] \arrow[ll, bend right, "\mathsf{l}^G"']
\end{tikzcd}\eqno \mathsf{R}_{\mathsf{tr}}(\U^G,\T^G,\V^G)
\end{equation*}
\end{prop}
\begin{proof}
The invertibility of $|G|$ in $\T$ yields its invertibility in $\U,\V$ by Remark~\ref{invertibilityonsubcats}. Thus, $\U^G$ and $\V^G$ admit a canonical pre-triangulated structure by Proposition~\ref{canonicalpretriangstructure}. By Lemma~\ref{triangadjointequiv} all equivariant functors are triangulated. Observe that, using the fully faithful triangle functor $\mathsf{i}^G \colon \U^G \to \T^G$, the octahedron axiom holds for $\U^G$, thus it is canonically triangulated. Similarly, $\V^G$ is canonically triangulated using either $\mathsf{l}^G$ or $\mathsf{r}^G$.
%Then $\V^G$ being identified with the Verdier quotient $(\T / \U)^G$ admits a canonical triangulated structure by Theorem~\ref{quotienttheorem}. 
  The rest of the proof is the same as the proof of Proposition~\ref{prop1}.
\end{proof}

The next result is the triangulated analogue of  Proposition~\ref{actiontosubcategory} using Remark~\ref{identificationofTRIANGsubcategories} and Proposition~\ref{prop2}. The proof is similar to the one given in Proposition~\ref{actiontosubcategory} and is left to the reader.

\begin{prop}\label{triangactiontosubcategory}
Let  $\mathsf{R}_{\mathsf{tr}}(\U,\T,\V)$ be a triangulated recollement. Let $G$ be a finite group acting admissibly on $\T $ and $\V $ such that either one of $\mathsf{l,e}$ or $\mathsf{r}$ is a $G$-functor. Assume also that $|G|$ is invertible in $\T$ and that $\T^G$ admits a canonical triangulated structure. Then we have an induced action on $\U$ and $\mathsf{R}_{\mathsf{tr}}(\U^G,\T^G,\V^G)$ is a recollement of triangulated categories.
\end{prop}

The following is the triangulated analogue of Proposition~\ref{actiontoquotient}.

\begin{prop}\label{actiontoquotienttriang}
Let  $\mathsf{R}_{\mathsf{tr}}(\U,\T,\V)$ be a triangulated recollement. Let $G$ be a finite group acting admissibly on $\T $ such that $\U $ is a $G$-invariant subcategory of $\T $.  Assume also that $|G|$ is invertible in $\T $ and that $\T ^G$ admits a canonical triangulated structure. Then there is an induced action on $\V$ rendering $\mathsf{e}$ a $G$-functor and $\mathsf{R}_{\mathsf{tr}}(\U^G,\T^G,\V^G)$ is a recollement of triangulated categories.
\end{prop}

\begin{proof}
By Theorem~\ref{quotienttheorem} we have that the quotient $\T /\U $ has a canonical admissible $G$-action and the equivariant category $\mathcal{(T/U)}^G$ admits the canonical triangulated structure. Since the quotient category is triangle equivalent to $\V $, say by $F\colon  \mathcal{T/U} \xrightarrow[]{\sim} \V $, then $\V $ inherits a $G$-action and $F$ becomes a $G$-functor, i.e.\ $F^G$ is a triangulated equivalence (see Lemma~\ref{inducedequivariantequivalence} and Lemma~\ref{triangadjointequiv}). By Remark~\ref{quotientGfunctor} the quotient functor $Q \colon  \T  \to\mathcal{T/U}$ is naturally a $G$-functor. Now, $\mathsf{e}$ becomes also a $G$-functor being the composition of $G$-functors $\mathsf{e}=Q\circ F$ by the universal property of the quotient. Note that $\mathsf{i}$ is a $G$-functor as a composition of such, as described in Remark~\ref{identificationofTRIANGsubcategories}. Now the result follows from Proposition~\ref{prop2}.
\end{proof}

Combining the previous results yields the following:

\begin{thm} 
\label{main2}
 Let  $\mathsf{R}_{\mathsf{tr}}(\U,\T,\V)$ be a triangulated recollement.
%\begin{equation*}
%\begin{tikzcd}
% \U  \arrow[rr, "i" description]& & \T  \arrow[ll, bend left, "p"] \arrow[ll, bend right, "q"'] \arrow[rr, "e" description] & & \V  \arrow[ll, bend left, "r"] \arrow[ll, bend right, "l"']
%\end{tikzcd} \eqno \mathsf{R}_{\mathsf{tr}}(\U,\T,\V)
%\end{equation*}
Let $G$ be a finite group acting admissibly on $\T $, $|G|$ is invertible in $\T $ and $\T ^G$ is canonically triangulated. Then the following are equivalent:
\begin{itemize}

\item[(i)] $G$ is acting admissibly on $\V \simeq \T /\U $ and $\mathsf{e}$ is a $G$-functor. 

\item[(ii)] $\U $ is a $G$-invariant triangulated subcategory of $\T $.
\end{itemize}
If either of the above conditions holds true, then $\mathsf{R}_{\mathsf{tr}}(\U^G,\T^G,\V^G)$ is a recollement of triangulated categories.
\begin{proof}
The implication (i) $\Longrightarrow$ (ii) is Proposition~\ref{triangactiontosubcategory} while the converse implication (ii) $\Longrightarrow$ (i) is Proposition~\ref{actiontoquotienttriang}.
\end{proof}
\end{thm}

\begin{defn}\label{equivarianttriangrecollement}
    We say that a recollement of triangulated categories $\mathsf{R}_{\mathsf{tr}}(\U,\T,\V)$ \textbf{lifts to a} $G$-\textbf{equivariant recollement} $\mathsf{R}_{\mathsf{tr}}(\U^G,\T^G,\V^G)$ if a  group $G$ acts admissibly on $\T $, $|G|$ is invertible in $\T$, $\T^G$ is canonically triangulated and either of the equivalent conditions of Theorem~\ref{main2} are satisfied.
\end{defn}

We finish this section with the following remark on short exact sequences of abelian (respectively, triangulated) categories.

\begin{rem}\label{equivariantSES}
Let $0\to \A \xrightarrow{\mathsf{i}} \B \xrightarrow{\mathsf{e}} \C \to 0$ be a short exact sequence of abelian categories, i.e.\ the category $\C$ is equivalent to the Gabriel quotient 
$\B / \A$ where $\A$ is a Serre subcategory of $\B$. Assume that a finite group $G$ acts on $\A, \B$ and $\C$ such that $\mathsf{i}$ and $\mathsf{e}$ are $G$-functors. Then by \cite[Corollary~4.4]{CCZ} (or Remark~\ref{quotientequivalenceofequivariantcats}) we have an equivalence $\B^G/\A^G \simeq (\B /\A)^G$, which shows that we have a short exact sequence of equivariant abelian categories $0 \to \A^G \xrightarrow{\mathsf{i}^G} \B^G \xrightarrow{\mathsf{e}^G} \C^G \to 0$.

In the case of short exact sequences of triangulated categories the above does not carry over verbatim. Assume that $0\to \U \xrightarrow{\mathsf{i}} \T \xrightarrow{\mathsf{e}} \V \to 0$ is a short exact sequence of triangulated categories and let $G$ be a finite group acting on $\U, \T$ and $\V$ such that $|G|$ is invertible in all three categories. Assume further that all three equivariant categories are triangulated. Then, by Theorem~\ref{quotienttheorem}, the Verdier quotient $\T^G / \U^G$ is equivalent to $(\T/\U)^G$ up to retracts. Consider the following commutative diagram:
\begin{equation*}
    \begin{tikzcd}
0 \arrow[r] & \U^G \arrow[d, "\mathsf{Id}"]  \arrow[r, "\mathsf{i}^G"] & \T^G  \arrow[d, "\mathsf{Id}"]  \arrow[r, "Q"] & \T^G / \U^G \arrow[r]  \arrow[d, "F"] & 0  \\
0 \arrow[r] & \U^G \arrow[r, "\mathsf{i}^G"] & \T^G \arrow[r, "\mathsf{e}^G"] & (\T / \U)^G 
    \end{tikzcd}
\end{equation*}
%where we have identified $\V= (\T /\U)$ (hence $\V^G = (\T /\U)^G$) and $e$ (hence $e^G$) with the (equivariant) quotient functor.
where $\V$ is identified with $\T / \U$. We can think that the natural equivalence up to retracts $F \colon \T^G / \U^G \to (\T / \U)^G$ measures how far is the above sequence from being exact.
Nevertheless, if we enhance the short exact sequence with a right adjoint $\mathsf{r}\colon \V \to \T$, then the composition $\mathsf{e} \circ \mathsf{r}$ becomes a localization functor with $\Ker (\mathsf{er})=\U $ by \cite[Proposition~4.9.1~(4)]{Krause: localizationtheory} and, moreover, $\mathsf{r}$ is fully faithful by \cite[Corollary~2.4.2]{Krause: localizationtheory}. 
Hence, the same holds for the equivariant functors $\mathsf{r}^G$ and $\mathsf{e}^G$, i.e.\ the composition $\mathsf{e}^G \circ \mathsf{r}^G$ is a localization functor, thus $F \colon \T^G / \U^G \to (\T / \U)^G$ is an equivalence by \cite[Corollary~2.4.2]{Krause: localizationtheory}. In particular, this is true for the recollement of equivariant triangulated categories. Therefore, the left-exact sequence of equivariant triangulated categories is a short exact sequence if either we assume that the Verdier quotient $\T^G / \U^G$ is idempotent complete, or if we assume that $\T$ admits a localization functor $L \colon \T \to \T$ with $\Ker L = \U$, equivalently $\mathsf{e}$ admits a right adjoint, equivalently $\mathsf{i}$ admits a right adjoint, by \cite[Proposition~4.9.1~(2)]{Krause: localizationtheory}.

Therefore, the essential difference between subsections~\ref{recollementsofabeliancats} and~\ref{recollementsoftriangulatedcats} is that the results of the former can be formulated for exact sequences of abelian categories without any extra assumptions, while for the latter we need extra assumptions so that $F \colon \T^G / U^G \to (\T/\U)^G$ is an equivalence of triangulated categories.
\end{rem}

\section{Equivariant Yoneda Extensions}
\label{section: Yoneda extensions}

In this section we investigate homological embeddings in equivariant abelian categories. In this section, all finite group actions are assumed to be right actions.

Recall that the Yoneda extension group is
\[
\Ext^n_{\A}(X,Y) = \big\{\theta\colon 0 \rightarrow Y \rightarrow A_n \rightarrow \dots \rightarrow A_1 \rightarrow X \rightarrow 0 \ | \ \theta \ \text{exact sequence} \big\} / \sim
\]
where $\theta \sim \theta'$ if and only if there exists an extension $\theta''$ and chain maps $\mu$ and $\nu$ that are identity on the sides in the following way:
\begin{equation*}
    \begin{tikzcd}
 \theta:    & 0 \arrow[r] & Y \arrow[d, equal] \arrow[r] & A_n  \arrow[r] & \dots \arrow[r] & A_1 \arrow[r] & X \arrow[d, equal] \arrow[r] & 0 \\
 \theta'': \arrow[u, "\mu"'] \arrow[d, "\nu"]     & 0 \arrow[r] & Y \arrow[d, equal] \arrow[r] & A''_n \arrow[d, "\nu_n"] \arrow[r] \arrow[u, "\mu_n"'] & \dots \arrow[r] & A''_1 \arrow[d, "\nu_1"]\arrow[r] \arrow[u, "\mu_1"'] & X \arrow[d, equal] \arrow[r] & 0 \\
      \theta' : & 0 \arrow[r] & Y \arrow[r] & A'_n \arrow[r] & \dots \arrow[r] & A'_1 \arrow[r] & X \arrow[r] & 0
    \end{tikzcd}
\end{equation*}

We also need to recall the next homological property of a functor.

\begin{defn}\label{homembeddeff}\textnormal{(\!\!\cite[Definition~3.6]{Psaroud})}
An exact functor $\mathsf{i}\colon \A \rightarrow \B$ between abelian categories is called a \textbf{$k$-homological embedding}, for $k \geq 0$, if the map
\[
\mathsf{i}^n_{X,Y}\colon \Ext^n_{\A}(X,Y) \xrightarrow[]{} \Ext^n_{\B}(\mathsf{i}X,\mathsf{i}Y)
\]
is invertible for all $X, Y$ in $\A$ and $0 \leq n \leq k$. The functor $\mathsf{i}$ is called a \textbf{homological embedding} if it is a $k$-homological embedding for all $k \geq 0$.
\end{defn}

We will need the following observation due to Oort~\cite{Oort}.

\begin{rem} 
\label{remhomemb}
If $\mathsf{i} \colon \A \rightarrow \B$ is a $k$-homological embedding, then the map 
\[
\mathsf{i}^{k+1}_{X,Y}\colon \Ext^{k+1}_{\A}(X,Y) \xrightarrow[]{} \Ext^{k+1}_{\B}(\mathsf{i}X,\mathsf{i}Y)
\]
is a monomorphism for all $X, Y \in \A$.
\end{rem}

Thus, assuming that the functor $\mathsf{i}\colon \A \rightarrow \B$ is a $k$-homological embedding, proving that it is a $(k+1)$-homological embedding, requires showing that the map $\mathsf{i}^{k+1}_{X,Y}\colon \Ext^{k+1}_{\A}(X,Y) \xrightarrow[]{} \Ext^{k+1}_{\B}(\mathsf{i}X, \mathsf{i}Y)$ is an epimorphism. We use this fact to prove how to lift a homological embedding between abelian categories to equivariant ones.

\begin{prop}
\label{equivarianthomembedding}
Let $\mathsf{i}\colon \A \rightarrow \B$ be an exact functor between abelian categories. Assume that $\mathsf{i}$ is a $G$-functor. If the functor $\mathsf{i}$ is a k-homological embedding, then the equivariant functor $\mathsf{i}^G\colon \A^G \rightarrow \B^G$ is also a k-homological embedding. 
\begin{proof}
We prove that $\mathsf{i}^G$ induces an isomorphism 
\[
\mathsf{i}^G_{{(X,\chi),(Y,\psi)}}\colon \Ext^n_{\A^G}((X,\chi),(Y,\psi)) \simeq \Ext^n_{\B^G}(\mathsf{i}^G(X,\chi),\mathsf{i}^G(Y,\psi)) 
\]
for all $(X,\chi),(Y,\psi)$ in $\A^G$ and $0 \leq n \leq k$. Note that $\Ext^n_{\B^G}(\mathsf{i}^G(X,\chi),\mathsf{i}^G(Y,\psi)=\Ext^n_{\B^G}((\mathsf{i}X,\chi'),(\mathsf{i}Y,\psi'))$ where $ \chi', \psi'$ are induced by $\mathsf{i}^G$, see Lemma~\ref{gfunctorinducesequivariant}. First we show that $\mathsf{i}^G$ is well defined, i.e.\ for $\theta \sim \theta'$ we prove that $\mathsf{i}^G(\theta) \sim \mathsf{i}^G(\theta')$. Suppose we have two equivalent extensions $\theta \sim \theta'$:
\[
\begin{tikzcd}[column sep=small]
        \theta: & 0 \arrow[r] & (Y, \psi) \arrow[r] & (A_n, a_n)  \arrow[r] & \cdots \arrow[r] & (A_1, a_1) \arrow[r] & (X, \chi)  \arrow[r] & 0 \\
       \theta'': \arrow[u, "\mu"'] \arrow[d, "\nu"] & 0 \arrow[r] & (Y, \psi) \arrow[u, equal]\arrow[d, equal] \arrow[r] & (A''_n, a_n) \arrow[u, "\mu_n"']\arrow[d, "\nu_n"] \arrow[r] & \cdots \arrow[r] & (A''_1, a_1) \arrow[u, "\mu_1"'] \arrow[d, "\nu_1"]\arrow[r] & (X, \chi) \arrow[d, equal] \arrow[u, equal]\arrow[r] & 0 \\
      \theta' : & 0 \arrow[r] & (Y,\psi) \arrow[r] & (A'_n, a'_n) \arrow[r] & \cdots \arrow[r] & (A'_1, a'_1) \arrow[r] & (X, \chi) \arrow[r] & 0
    \end{tikzcd}
\]
We claim that their images under $\mathsf{i}^G$ are also equivalent via some $\mu'$ and $\nu'$. Set $(\mu'_i)_{1 \leq i \leq n}:=(\mathsf{i}^G \mu_i)_{1 \leq i \leq n} $ and $(\nu'_i)_{1 \leq i \leq n}:=(\mathsf{i}^G \nu_i)_{1 \leq i \leq n} $. Since the functor $\mathsf{i}^G$ is exact, we obtain following commutative diagram:
\[
\begin{tikzcd}[column sep=small]
     \mathsf{i}^G(\theta): & 0 \arrow[r] & \mathsf{i}^G(Y,\psi) \arrow[r] & \mathsf{i}^G(A_n, a_n) \arrow[r] & \cdots \arrow[r] & \mathsf{i}^G(A_1, a_1) \arrow[r] & \mathsf{i}^G(X,\chi) \arrow[r] & 0 \\
      \mathsf{i}^G(\theta''): \arrow[d, "\nu'"] \arrow[u, "\mu'"'] & 0 \arrow[r] & \mathsf{i}^G(Y,\psi) \arrow[d, equal]\arrow[u, equal] \arrow[r] & \mathsf{i}^G(A_n'', a_n) \arrow[d, "\mathsf{i}^G\nu_n"] \arrow[u, "\mathsf{i}^G\mu_n"'] \arrow[r] & \cdots \arrow[r] & \mathsf{i}^G(A''_1, a_1) \arrow[d, "\mathsf{i}^G\nu_1"] \arrow[u, "\mathsf{i}^G\mu_1"'] \arrow[r] & \mathsf{i}^G(X,\chi) \arrow[d, equal] \arrow[u, equal]\arrow[r] & 0 \\
      \mathsf{i}^G(\theta') : & 0 \arrow[r] & \mathsf{i}^G(Y,\psi) \arrow[r] & \mathsf{i}^G(A'_n,a'_n) \arrow[r] & \cdots \arrow[r] & \mathsf{i}^G(A'_1,a'_1) \arrow[r] & \mathsf{i}^G(X,\chi) \arrow[r] & 0
\end{tikzcd}   
\]
We infer that the map $\mathsf{i}^G$ is well defined.

If $\mathsf{i}$ is fully faithful, then $\mathsf{i}^G$ is fully faithful by Lemma~\ref{ffequiv}. Also, the functor $\mathsf{i}$ being a $1$-homological embedding, implies that $\mathsf{i}^G$ is a $1$-homological embedding. Indeed, $\mathsf{i}^G$ is fully faithful and one can show that its image is also closed under extensions using the same arguments appearing in Remark~\ref{SerreSubcatEquivariant}. Thus, by Remark~\ref{remhomemb}, to prove that $\mathsf{i}^G$ is a $k$-homological embedding, it suffices to prove that $(\mathsf{i}^G)^n$ is an epimorphism for $2 \leq n \leq k$. Let $\theta' \in \Ext^n_{\B^G}((\mathsf{i}X,\chi'),(\mathsf{i}Y, \psi'))$ be represented by 
\[
0 \xrightarrow[]{} (\mathsf{i}Y, \psi') \rightarrow  (A'_n, a'_n) \rightarrow \dots \rightarrow (A'_1, a'_1) \rightarrow (\mathsf{i}X, \chi') \rightarrow 0  
\]
Applying the exact forgetful functor we have that $\mathsf{For}(\theta') \in \Ext^n_{\B^G}(\mathsf{i}X,\mathsf{i}Y)$. Since $\mathsf{i}$ is a k-homological embedding, we obtain an extension 
\[
\xi\colon \ 0 \rightarrow Y \rightarrow A_n \rightarrow \dots \rightarrow A_1 \rightarrow X \rightarrow 0 
\]
such that $\mathsf{i}(\xi) = \mathsf{For}(\theta')$. Combining the above we can write $\theta'$ as follows:
\[
 0 \xrightarrow[]{} (\mathsf{i}Y, \psi') \rightarrow  (\mathsf{i}A_n, a'_n) \rightarrow \dots \rightarrow (\mathsf{i}A_1, a'_1) \rightarrow (\mathsf{i}X, \chi') \rightarrow 0  
\]
By Lemma~\ref{epilinear}, we have that for all $(\mathsf{i}A_i, a'_i)$ there exists $(A_i, a_i)$ in $\A^G$ such that $\mathsf{i}^G(A_i, a_i) = (\mathsf{i}A_i, a'_i) $. Since $\mathsf{i}^G$ is fully faithful, we have that maps between two consecutive terms of $\theta'$ are images of maps of $\A^G$. In particular, if $f'\colon (\mathsf{i}A'_i, a'_i) \xrightarrow[]{} (\mathsf{i}A_{i+1}, a'_{i+1})$ (where $A'_i$ might be $\mathsf{i}X$ and $A'_{i+1}$ might be $\mathsf{i}Y$), there exists $f\colon (A_i, a_i) \rightarrow (A_{i+1}, a_{i+1})$ such that $\mathsf{i}(f) =f'$.
We have proved that there exists some $\theta \in \mathrm{Ext}^n_{\A^G}((X,\chi),(Y, \psi))$, consisting of the above $(A_i, a_i) \in \A^G$ and their maps, such that $\mathsf{i}^G(\theta) = \theta'$.
\end{proof}
\end{prop}

In the next section, we will examine the converse in the context of abelian recollements, that is, whether the equivariant functor $\mathsf{i}^G$ being a k-homological embedding implies that the functor $\mathsf{i}\colon \A \rightarrow \B$ is also a k-homological embedding.

\section{Equivariant Derived Categories}
\label{equivariant derived cats}

In this section we investigate triangulated recollements of bounded derived categories induced by abelian recollements that lift to $G$-equivariant recollements. Throughout this section all groups are finite and all actions are right ones.

\subsection{Action on Derived Categories}\label{actiononderivedcats}

By subsection~\ref{actionontriangcats} we already know that the equivariant category of a triangulated category is not always triangulated. The question that arises is: what happens in the derived category scenario?
If $\A$ is an abelian category with enough injectives, then a categorical action of $G$ on $\A$ induces a categorical action of $G$ on the bounded derived category $\mDb(\A)$ but its equivariant category is not a priori triangulated. 
%{\color{red}since cones are not functorial and there is no canonical way to define linearization of a cone of a morphism between equivariant objects}. 
In the next subsection we describe briefly the tools we need in the sequel of our work. For a more in depth view to this problem we refer you to \cite{XiaoWuChen,Elagin,Sosna}.

\begin{rem}\label{Gactiononderivedcat}
    The action of $G$ on an abelian category $\A$ extends to an admissible action on $\mDb(\A)$ via the derived functors of each $\rho_g$, which we denote also by $\rho_g$, because each $\rho_g\colon  \A \xrightarrow[]{} \A$ is exact, thus its derived functor is computed component-wise as follows:
\begin{equation}\label{actiononD^b}
\rho_g(\cdots \xrightarrow{} X^n \xrightarrow[]{\delta^n} X^{n+1} \xrightarrow[]{} \cdots) = \ \ \cdots \xrightarrow[]{} (X^n)^g \xrightarrow{(\delta^n)^g} (X^{n+1})^g \xrightarrow[]{} \cdots 
\end{equation}
We write $ \rho_g(X^{\bullet}) \coloneqq (X^{\bullet})^g$ for the above action on chain complexes.

The composition isomorphisms $\theta_{g,h}$ of the induced action on chain complexes are also induced canonically, i.e.\ the family is induced by component-wise composition isomorphisms and it satisfies the 2-cocycle condition~$(\ref{cocycle})$.
 We write $(\mDb(\A), \rho, \theta)$ for the (right) $G$-action induced by the $G$-action on $(\A, \rho, \theta)$, to emphasize that it is induced by it. However, we omit this notation when it is clear. We write $(X^\bullet, \chi^\bullet)$ for the $G$-equivariant complexes.
\end{rem}

It is natural to consider the equivariant derived category $\mDb(\A)^G$, which is canonically pre-triangulated, by Proposition~\ref{canonicalpretriangstructure}. Notice that since $\A^G$ is an abelian category, there exists also its derived category $\mDb(\A^G)$. There exists a natural functor $K_{\A} \colon  \mDb(\A^G) \xrightarrow{} \mDb(\A)^G$, called the \textbf{comparison functor}, which maps 
$(X,\chi)^{\bullet} \coloneqq \cdots \xrightarrow[]{} (X^n, \chi^n) \xrightarrow{\delta^n}  (X^{n+1},\chi^{n+1}) \xrightarrow[]{} \cdots$ 
to $(X^{\bullet}, \chi^{\bullet})$.

The following result has been proved independently by  Xiao-Wu Chen \cite[Proposition~4.5]{XiaoWuChen2} and Elagin \cite[Theorem~7.1]{Elagin}. 
Here we want to point out that while both used a special case of Beck's theorem (see \cite[Chapter~VI,~7.1]{MacLane}), Chen gives a full description of this comparison functor but proves it when the action is strict. However, Chen comments that the results probably holds without the strictness condition. Elagin proved this result for any action, but without defining precisely the comparison functor. 
%(since he did not need it for the rest of his work). 
Combining both their results, we can see that Chen's Theorem works also without the strictness assumption. Namely, having a strict action implies that $\A^G$ is {\em isomorphic} to $\A_M$, which is the category of modules defined by the monad on $\A$ induced by the adjunction $\mathsf{Ind} \dashv \mathsf{For}$. If the action is not strict, we still have an {\em equivalence} of categories $\A^G \simeq \A_M$ by results of Elagin.

\begin{thm}\label{derivedequivariantequiv}\textnormal{(\!\!\cite[Proposition~4.5]{XiaoWuChen2}, \cite[Theorem~7.1]{Elagin})}
 Let G be a finite group acting on an abelian category $\A$ and $|G|$ is invertible in $ \A$. Then $\mathsf{D^b}(\A)^G$ is canonically triangulated and $K_{\A}\colon  \mDb(\A^G) \xrightarrow{} \mDb(\A)^G$ is a triangle equivalence.
\end{thm}

Note that the assumption on $|G|$ is necessary, see \cite[Remark~4.6,~(2)]{XiaoWuChen2} for a counter example in the modular case of a field $k$ of characteristic dividing $|G|$.

\begin{rem}
Theorem~\ref{derivedequivariantequiv} has many interesting consequences. For instance, if $X$ is a quasi-projective variety over a field $k$ with $G$-action as described in Examples~\ref{exampleonsheaves1} and~\ref{exampleonsheaves2} such that $\mathrm{char}(k)$ does not divide $|G|$, then $  \mathsf{\mDb}(\mathsf{Coh}[X/G])\simeq \mDb(\mathsf{Coh}^G(X)) \simeq \mDb(X)^G$.
Notice that if we work over a field of characteristic zero, then this result follows directly for any finite group $G$.
\end{rem}

\subsection{Equivariant Derived Functors}
In this subsection, we investigate the existence of the derived functor of an exact $G$-equivariant functor between abelian categories as well as certain homological properties. This subsection constists of many verifications of natural constructions which will be useful also in the next subsection. The key result is Proposition~\ref{finitenessoflocaldimension}.
Throughout this section, all derived categories are endowed with the (right) action of the finite group $G$ induced from the underlying abelian category.

\begin{lem}\label{D^b(e)Gfunctor}
Let $F\colon \B \to \C$ be an exact functor between abelian categories. Let also $G$ be a finite group acting on $\B$ and $\C$ with $|G|$ invertible in both categories such that $(F, \sigma^F)$ is a $G$-functor.
There exists a canonical family $\{ \sigma^{\mDb(F)}_g \}_{g \in G}$ so that  $(\mDb(F), \sigma^{\mDb(F)}) \colon \mDb(\B) \to \mDb (\C)$ is a $G$-functor. 
\end{lem}
\begin{proof}
We show that $\{ \sigma^{\mDb(F)}_g \}_{g \in G}$ is the family of isomorphisms induced by the $G$-functor structure of $F$, i.e.\ the natural transformations $\sigma^F_g \colon   F \circ \rho_g \xrightarrow[]{\sim} \rho_g \circ F $. Indeed, let $X^{\bullet} \in \mDb(\B)$ and $g \in G$, then we have the following commutative diagram:
\begin{equation}\label{natransformobjects}
\begin{tikzcd}
\mDb(F) \big( (X^{\bullet})^g \big) \arrow[dd,"\big(\sigma^{\mDb(F)}_{g} \big)^{X^{\bullet}}"] & \cdots \arrow[r] & F((X^n)^g) \arrow[r, " F((\delta^n)^g)"] \arrow[dd, "\big(\sigma^{F}_{g}\big)^{X^n}"'] &  F((X^{n+1})^g) \arrow[r]  \arrow[dd, "\big(\sigma^{F}_{g}\big)^{X^{n+1}}"] & \cdots \\
\\
 \big( \mDb(F) X^{\bullet} \big)^g  & \cdots \arrow[r] & (FX^n)^g \arrow[r, " (F\delta^n)^g"'] &  (FX^{n+1})^g \arrow[r] & \cdots
\end{tikzcd}
\end{equation}
The middle (and every other) square commutes since $\sigma_g^F$ is a natural isomorphism. Therefore, $\big(\sigma^{\mDb(F)}_{g} \big)^{X^{\bullet}}$ is an isomorphism of complexes.

Now let $X^{\bullet}, Y^{\bullet} \in \mDb(\B)$ and $f \in \Hom _{\mDb(\B)} ( X^{\bullet}, Y^{\bullet})$, that is, there exists a quasi-isomorphism  $q \colon W^{\bullet} \to X^{\bullet}$ such that the morphism $f$ is represented by the roof $X ^{\bullet} \xleftarrow{q} W^{\bullet}  \xrightarrow{f}~Y^{\bullet} $. Note that both $\rho_g$ and $F$, being exact functors, preserve quasi-isomorphisms. We have the following commutative square (independent of the choice of $W^{\bullet}$) which proves that $\{ \sigma^{\mDb(F)}_g \}_{g \in G}$ is a family of natural isomorphisms. 

\begin{equation}\label{natransformmorphisms}
\begin{tikzcd}
 \mDb(F) \big( (X^{\bullet})^g \big) \arrow[rr,"\big(\sigma^{\mDb(F)}_{g}\big)^{X^{\bullet}}"]& &   \big(\mDb(F)(X^{\bullet})\big)^g  \\
 \\
 \mDb(F) \big( (W^{\bullet})^g \big) \arrow[rr," \big(\sigma^{\mDb(F)}_{g}\big) ^{W^{\bullet}}"]  \arrow[uu, " \mDb(F)(f^g)"] \arrow[dd, " \mDb(F) (q^g)"'] & & \big( \mDb(F)(W^{\bullet}) \big)^g \arrow[dd, "  \big(\mDb(F)(f)\big)^g"]  \arrow[uu, " \big(\mDb(F)(q) \big)^g"']\\
 \\
 \mDb(F) \big((Y^{\bullet})^g \big)  \arrow[rr, "\big( \sigma^{\mDb(F)}_{g}\big)^{Y^{\bullet}}"] &&  \big( \mDb(F)(Y^{\bullet}) \big)^g 
\end{tikzcd}
\end{equation}
%\smallskip

To see the commutativity of the bottom square of the above diagram, one would have to see the complexes commute component-wise. Let $f \colon  W^\bullet \to Y^\bullet$ be the chain map and consider the following diagram:  

\begin{equation*}
\begin{tikzcd}
\dots \arrow[r, dashrightarrow] & F((W^n)^g) \arrow[dr, "(\sigma^F_{g})^{W_n}"] \arrow[dd, dashrightarrow, near start, "F((f^n)^g)" description] \arrow[rr, dashrightarrow, "F(( \delta^n)^g)"] && F(( W^{n-1})^g) \arrow[dr, "(\sigma^F_{g})^{W_{n-1}}"] \arrow[dd, dashrightarrow, near start, "F((f^{n-1})^g)" description]  \arrow[rr, dashrightarrow] && \dots \\
\dots \arrow[rr] & & (F W^n)^g \arrow[dd, near start, "(Ff^n)^g"] \arrow[rr, near start, "(F \delta^n)^g" description] & & (F W^{n-1})^g \arrow[dd, crossing over, near start, "(Ff^{n-1})^g"] \arrow[r]  & \dots \\
\dots \arrow[r, dashrightarrow]  & F(( Y^n)^g) \arrow[dr, "(\sigma^F_{g})^{Y_n}" description] \arrow[rr,dashrightarrow, near start, "F(( \delta^n)^g)"] & & F(( Y^{n-1})^g) \arrow[dr, "(\sigma^F_{g})^{Y_{n-1}}" description] \arrow[rr, dashrightarrow] & & \dots \\
\dots \arrow[rr] & & (F Y^n)^g \arrow[rr, "(F \delta^n)^g"] & & (F Y^{n-1})^g  \arrow[r]  & \dots
\end{tikzcd}
\end{equation*}
The top and botom faces commute due to $\sigma^F_g$ being a natural isomorpshism. The front and back faces commute since $(Ff)^g$ and $F(f^g)$ respectively are chain maps. The left and right faces commute since $\sigma^F_g$ is a natural isomorphism.
 
The commutativity of the upper square is obtained similarly.
This family satisfies the associativity condition~$(\ref{associativity})$, since it is induced by $\sigma^{F}$.
\end{proof}

Note that since $\mDb(F)$ is naturally a $G$-functor, its equivariant functor $\mDb(F)^G$ is triangulated functor by Lemma~\ref{triangadjointequiv}. In the following Lemma we show that $\mDb(F)^G$ and $\mDb(F^G)$ are ``equivalent" up to the comparison functor.

\begin{lem}\label{D(F^G)andD(F)^G}
    With the same hypothesis as in Lemma~\ref{D^b(e)Gfunctor}, we have the following commutative diagram:
    \begin{equation*}
        \begin{tikzcd}
            \mDb(\B^G) \arrow[d,"K_{\B}"'] \arrow[r, "\mDb(F^G)"] & \mDb(\C^G) \arrow[d,"K_{\C}"] \\
            \mDb(\B)^G \arrow[d, "\mathsf{For}"'] \arrow[r, "\mDb(F)^G"] & \mDb(\C)^G \arrow[d, "\mathsf{For}"] \\
            \mDb(\B) \arrow[r, "\mDb(F)"] & \mDb(\C)
        \end{tikzcd}
    \end{equation*}

\begin{proof}
Because the functor $F$ is exact, its derived functor $\mDb(F)$ is computed by applying $F$ component-wise on complexes, thus we write
\[ \mDb(F)(X^\bullet) = \cdots \xrightarrow{} FX^n \xrightarrow{F\delta^n} FX^{n+1} \xrightarrow{} \cdots \eqqcolon (FX)^\bullet \]
Similar notation is used for the functors $\mDb(F^G)$ and $\mDb(F)^G$. 

Therefore, verifying the commutativity of the upper square reduces to checking it component-wise on the underlying objects and their linearizations. Since both functors $\mDb(F^G)$ and $\mDb(F)^G$ are defined by applying $F$ and $\sigma^F$ to each component of a complex, they naturally coincide. 

To formalize this, let $(X, \chi)^\bullet$ be an object in $\mDb(\B^G)$. Applying the comparison functor yields $K_{\B}(X, \chi)^\bullet = (X^\bullet, \chi^\bullet)$. Next, applying the derived $G$-functor gives:
\[ \mDb(F)^G(X^\bullet, \chi^\bullet) = (\mDb(F)X^\bullet, \mDb(F)\chi^\bullet) = ((FX)^\bullet, (F\chi)^\bullet) \]
where $(F\chi)^\bullet$ is the linearization induced by $(\mDb(F), \sigma^{\mDb(F)})$; see Lemma~\ref{gfunctorinducesequivariant}. As shown in the proof of Lemma~\ref{D^b(e)Gfunctor}, each $\sigma^{\mDb(F)}_g$ is constructed component-wise from $\sigma^F$, which justifies using the notation $(F\chi)^\bullet$ for the resulting linearization.

Following the other path around the square, we first apply the derived equivariant functor:
\[ \mDb(F^G)(X, \chi)^\bullet = (F^G(X, \chi))^\bullet = (FX, F\chi)^\bullet \]
where $F\chi$ is the linearization induced by $(F, \sigma^F)$. Finally, applying the comparison functor yields:
\[ K_{\C}(FX, F\chi)^\bullet = ((FX)^\bullet, (F\chi)^\bullet) \]
This proves the commutativity of the upper square. The commutativity of the lower square is already known from diagram~$(\ref{comdiagramFor})$.
\end{proof}
\end{lem}

We mention the following remark that can be proven by observing that the part of \cite[Corollary~3.13]{ChaoSun}  regarding the existence of enough projectives (resp.\ injectives) holds for any idempotent complete exact category. 

\begin{rem}\label{enoughprojinj}
Let $\B$ be an abelian category with enough projectives (resp.\ injectives). Assume a finite group $G$ acts on $\B$ with $|G|$ invertible in $\B$. Then $\B^G$ has enough projectives (resp.\ injectives). 
\end{rem}

We need the above remark for the final two results of this subsection for which we also recall the following notion from \cite{Psaroud}.

\begin{defn}
A left (resp.\ right) exact functor $F\colon  \B \to \C$ between abelian categories, where $\B$ has enough injectives (resp.\ projectives), is \textbf{of locally finite cohomological} (resp.\ \textbf{homological}) \textbf{dimension} if for every $X \in \B$ there exists some $n_X\geq 0$ such that the right derived functor $\mathbb{R}^nF(X)=0$ (resp.\ the left derived functor $\mathbb{L}^nF(X)=0$) for all $n > n_X$.
\end{defn}

\begin{prop}\label{finitenessoflocaldimension}
Let $\B$ and $\C$ be a abelian categories where $\B$ has enough injectives (resp.\ projectives). Assume a finite group $G$ acts on both $\B$ and $\C$ such that $|G|$ is invertible in both categories. Given a left (resp.\ right) exact $G$-functor \linebreak $F\colon  \B\to\C $, the right (resp.\ left) derived functor of the induced equivariant functor $F^G\colon  \B^G \to \C^G$ exists. Moreover, the following hold:
\begin{itemize}
\item[(i)] Assume that $X \in \B$ admits a linearization $(X, \chi) \in \B^G$. Then there exists some $n \in \mathbb{N}$ such that $\mathbb{R}^{n} F(X) = 0$ if and only if $\mathbb{R}^{n} F^G (X, \chi) = 0 $ (resp.\ for left derived functors).

\item[(ii)] There exists some $N \in \mathbb{N}$ such that $\mathbb{R}^{N}F(X) = 0$ for all $X \in \B$ if and only if $\mathbb{R}^N F^G (X,\phi) = 0$ for all $(X, \chi) \in \B^G$ (resp.\  for left derived functors). In particular, $\mathbb{R}F$ is bounded if and only if $\mathbb{R}F^G$ is bounded (resp.\ for bounded left derived functors).
\item[(iii)] The functor $F$ is of locally finite cohomological (resp.\ homological) dimension if and only if the equivariant functor $F^G$ is.
\end{itemize}
\end{prop}

\begin{proof}
We prove statements (i), (ii) and (iii) for the case of right derived functors (the other case is analogous).

First, recall that by Remark~\ref{enoughprojinj}, the equivariant category $\B^G$ has enough injectives. Let $(X, \chi)$ be an object in $\B^G$ and consider an injective resolution:
\[
(X,\chi) \to (I^0, \i^0) \to (I^1, \i^1) \to \dots
\]
Applying the equivariant functor $F^G$, which is also left exact, yields the complex:
\[
(FI^0, F\i^0) \to (FI^1, F\i^1) \to \dots
\]
where $F\i^n$ denote the linearizations obtained as in Lemma~\ref{gfunctorinducesequivariant}.
The $n$-th right derived functor $\mathbb{R}^n F^G(X, \chi)$ is defined as the $n$-th cohomology of this complex computed in $\C^G$.

Recall that the forgetful functor $\mathsf{For} \colon \B^G \to \B$ is exact and preserves injectives (as the right adjoint of the exact induction functor). Therefore, applying $\mathsf{For}$ to the above resolution yields an injective resolution of $X$ in $\B$:
\[
X \to I^0 \to I^1 \to \dots
\]
Consequently, the underlying object of the complex $F^G(I^\bullet, \iota^\bullet)$ is simply $F(I^\bullet)$. 
Since cohomology in the equivariant category $\C^G$ is computed using the underlying objects (kernels, images and cokernels in $\C^G$ are formed in $\C$ and then linearized, see Remark~\ref{0-object-equiv}), we have an isomorphism of underlying objects:
\[
\mathsf{For}\left(\mathbb{R}^n F^G(X, \chi)\right) \cong \mathrm{H}^n(F(I^\bullet)) \cong \mathbb{R}^n F(X).
\]

To prove (i), observe that from the isomorphism above, $\mathbb{R}^n F^G(X, \chi) = 0$ in $\C^G$ implies the underlying object $\mathbb{R}^n F(X)$ is zero in $\C$. Conversely, if $\mathbb{R}^n F(X) = 0$, then the equivariant object $\mathbb{R}^n F^G(X, \chi)$ is zero because the zero object has a unique linearization, see Remark~\ref{0-object-equiv}.

The ``only if" of (ii) follows immediately from (i). For the ``if" direction, assume that there exists a (fixed) $N$ such that $\mathbb{R}^N F^G(Y, \psi) = 0$ for all $(Y, \psi) \in \B^G$. For any $X \in \B$, consider its induction $\mathsf{Ind}(X) \in \B^G$. By hypothesis, we have $\mathbb{R}^N F^G(\mathsf{Ind}(X)) = 0$, hence $\mathsf{For}(\mathbb{R}^N F^G(\mathsf{Ind}(X)))=0 $. Using the  isomorphism above and the additivity of derived functors we have:
\[
\mathsf{For}(\mathbb{R}^N F^G(\mathsf{Ind}(X))) \cong \mathbb{R}^N F(\mathsf{For}(\mathsf{Ind}(X))) = \mathbb{R}^N F\left(\bigoplus_{g \in G} X^g\right) \cong \bigoplus_{g \in G} \mathbb{R}^N F(X^g).
\]
Since this direct sum is zero, each summand must be zero. In particular, for the unit $1_G$, we have $X^{1_G} \simeq X$, which implies that $  \mathbb{R}^N F(X) = 0$.

The ``only if" of (iii) is also immediate from (i).
The ``if" part follows by a similar proof as in (ii), by applying the induction functor on each object of $\B$ and using the hypothesis. 
\end{proof}

\begin{rem}
In the course of the proof above, we have described the form of the derived equivariant functor. Specifically, for any object $(X, \chi) \in \B^G$, we have a natural isomorphism:
\[
\mathbb{R}^n F^G(X, \chi) \cong ( \mathbb{R}^n F(X), \psi )
\]
where the linearization $\psi$ is canonically induced. In particular, it is the canonical linearization of the cokernel $\Image \big( F^G(I^{n-1}, \i^{n-1}) \to F^G(I^n, \i^n) \big)$ into the kernel $\Ker (F^G(I^n,\i^n) \to F^G(I^{n+1}, \i^{n+1}) )$, which are canonically linearized objects by Remark~\ref{0-object-equiv}.
\end{rem}

\begin{rem}
\label{counterexampleforchar}
Consider the recollement~$(\ref{modulerecollement})$. The above result provides a comparison between the left derived functors of $\mathsf{l}= -\otimes_{eRe}eR\colon \Mod{eRe}\to \Mod{R}$ and  $\mathsf{l}^G= -\otimes_{e'RGe'}e'RG\colon \Mod{e'RGe'}\to \Mod{RG}$. In particular, assuming that $\Tor^{eRe}_i(X, eR)=0$ for $i>n$, then we obtain that the $\Tor^{e'RGe'}_i(X, e'RG)=0$ for $i>n$. So, as probably expected, there is an interplay between homological properties of the recollement~$(\ref{modulerecollement})$ and the equivariant one $(\ref{equivariantmodulerecollement})$.
\end{rem}

We show below the converse of Proposition~\ref{equivarianthomembedding} for abelian recollements.

\begin{prop}\label{homologicalembedrecollements}
Let $\mathsf{R_{ab}} (\A, \B, \C)$ be a recollement of abelian categories. Assume that $\B$ has enough projective or injective objects. Assume also that a finite group $|G|$ acts on $\B$ with $|G|$ invertible in $\B$ and that the recollement lifts to a $G$-equivariant recollement. The following statements are equivalent:
\begin{itemize}
\item[(i)] The functor $\mathsf{i} \colon \A \to \B$ is a $k$-homological embedding.
\item[(ii)]The functor $\mathsf{i}^G \colon \A^G \to \B^G$ is a $k$-homological embedding.
\end{itemize}
\begin{proof}
We provide a proof for the case that $\B$ has enough projectives. The other case is similar. By \cite[Theorem~3.9]{Psaroud}) we have that $\mathsf{i} \colon \A \to \B$ is a $k$-homological embedding if and only if $\mathbb{L}^n \mathsf{q}(\mathsf{i}(X))=0$ for all $X \in \A$ and $1 \leq n \leq k$. This is equivalent to $\mathbb{L}^n \mathsf{q}^G(\mathsf{i}^G(X,\chi))=0$ for all $(X,\chi) \in \A^G$ and $1 \leq n \leq k$, by Proposition~\ref{finitenessoflocaldimension}. The latter is equivalent to $\mathsf{i}^G \colon \A^G \to \B^G$ being a $k$-homological embedding.
\end{proof}
\end{prop}

\subsection{Commutative Diagrams of Equivariant Derived Categories}

In this section we lift \cite[Theorem~7.2]{Psaroud} to the equivariant setting which we split into three parts and then prove the commutativity of a natural diagram of recollements, which constitutes the main result of this section.

\begin{prop}\label{derivedrecol-left-part}
    Let $\mathsf{R}_{\mathsf{ab}}(\A,\B,\C)$ be a recollement of abelian categories and assume that $\B$ and $\C$ have enough projective and injective objects. Assume that a finite group $G$ acts on $\B$ with $|G|$ be invertible in $\B$ and that the recollement lifts to a $G$-equivariant recollement.
   Then the following statements are equivalent:
\begin{itemize}
 \item[(a)] The functor $\mathsf{i}\colon  \A \to \B$ is a homological embedding, the functor $\mathsf{q}\colon  \B \to \A$ is of locally finite homological dimension and $\mathsf{p}\colon  \B \to \A$ is of locally finite cohomological dimension. 
\item[(b)] There exists a recollement of triangulated categories 
            \begin{equation*}
            \begin{tikzcd}
             \mDb(\A) \arrow[rr, "\mDb(\mathsf{i})" description]& & \mDb(\B) \arrow[ll, bend left, "\mathbb{R}\mathsf{p}"] \arrow[ll, bend right, "\mathbb{L}\mathsf{q}"'] \arrow[rr, "\mDb(\mathsf{e})" description] & & \mDb(\C) \arrow[ll, bend left, "\mathsf{r}'"] \arrow[ll, bend right, "\mathsf{l}'"']
            \end{tikzcd}
            \end{equation*} 
\item[(a-G)] The equivariant functor $\mathsf{i}^G\colon  \A^G \to \B^G$ is a homological embedding, the equivariant functor $\mathsf{q}^G\colon  \B^G \to \A^G$ is of locally finite homological dimension and also $\mathsf{p}^G\colon  \B^G \to \A^G$ is of locally finite cohomological dimension. 
\item[(b-G)] There exists a recollement of triangulated categories 
            \begin{equation*}
            \begin{tikzcd}
             \mDb(\A^G) \arrow[rr, "\mDb(\mathsf{i}^G)" description]& & \mDb(\B^G) \arrow[ll, bend left, "\mathbb{R}(\mathsf{p}^G)"] \arrow[ll, bend right, "\mathbb{L}(\mathsf{q}^G)"'] \arrow[rr, "\mDb(\mathsf{e}^G)" description] & & \mDb(\C^G) \arrow[ll, bend left, "\mathsf{r}'"] \arrow[ll, bend right, "\mathsf{l}'"']
            \end{tikzcd}
            \end{equation*}
\end{itemize}
\begin{proof} 
The equivalence between (a-G) and (b-G) follows from the equivalence between (a) and (b), which this is \cite[Theorem~7.2 (i)]{Psaroud}. Note that, by Remark~\ref{invertibilityonsubcats}, since $|G|$ is invertible in $\B$, then it is invertible in all three abelian categories and by Remark~\ref{enoughprojinj} the abelian categories $\B^G, \C^G$ also have enough projective and injective objects.
%the categories $\B, \C$ being abelian means that $\B^G, \C^G$ have also enough projective and injective objects by Remark~\ref{enoughprojinj}.
%Note also that the functors $i^G$ and $e^G$ being exact, imply that their derived functors $\mDb(i)$ and $\mDb(e)$ are $G$-functors by Lemma~\ref{D^b(e)Gfunctor} and thus its adjoints are also $G$-functors by Lemma~\ref{adjointequiv}.

(a)$\iff$(a-G):
By Proposition~\ref{finitenessoflocaldimension} we have that $\mathsf{q}$ is of locally finite homological dimension if and only if $\mathsf{q}^G$ is of locally finite homological dimension and $\mathsf{p}$ is of locally finite cohomological dimension if and only $\mathsf{p}^G$ is of locally finite cohomological dimension.  By Proposition~\ref{homologicalembedrecollements},  $\mathsf{i}$ is a homological embedding if and only if $\mathsf{i}^G$ is homological embedding. 
\end{proof}
\end{prop}

\begin{prop}\label{derivedrecol-right-part}
        Let $\mathsf{R}_{\mathsf{ab}}(\A,\B,\C)$ be a recollement of abelian categories and assume that $\B$ and $\C$ have enough projective and injective objects. Assume that a finite group $G$ acts on $\B$ with $|G|$ be invertible in $\B$ and that the recollement lifts to a $G$-equivariant recollement.
Then the following statements are equivalent:
\begin{itemize}
            \item[(a)] The functor $\mathsf{i}\colon  \A \to \B$ is a homological embedding, the functor $\mathsf{l}\colon  \C \to \B$ is of locally finite homological dimension and $\mathsf{r}\colon  \C \to \B$ is of locally finite cohomological dimension. 
            \item[(b)] There exists a recollement of triangulated categories 
            \begin{equation*}
            \begin{tikzcd}
             \mDb(\A) \arrow[rr, "\mDb(\mathsf{i})" description]& & \mDb(\B) \arrow[ll, bend left, "\mathsf{p}'"] \arrow[ll, bend right, "\mathsf{q}'"'] \arrow[rr, "\mDb(\mathsf{e})" description] & & \mDb(\C) \arrow[ll, bend left, "\mathbb{R}\mathsf{r}"] \arrow[ll, bend right, "\mathbb{L}\mathsf{l}"']
            \end{tikzcd}
            \end{equation*}
\item[(a-G)] The equivariant functor $\mathsf{i}^G\colon  \A^G \to \B^G$ is a homological embedding, the equivariant functor $\mathsf{l}^G\colon  \C^G \to \B^G$ is of locally finite homogical dimension and also $\mathsf{r}^G\colon  \C^G \to \B^G$ is of locally finite cohomological dimension. 
\item[(b-G)] There exists a recollement of triangulated categories 
            \begin{equation*}
            \begin{tikzcd}
             \mDb(\A^G) \arrow[rr, "\mDb(\mathsf{i}^G)" description]& & \mDb(\B^G) \arrow[ll, bend left, "\mathsf{p}'"] \arrow[ll, bend right, "\mathsf{q}'"'] \arrow[rr, "\mDb(\mathsf{e}^G)" description] & & \mDb(\C^G) \arrow[ll, bend left, "\mathbb{R}(\mathsf{r}^G)"] \arrow[ll, bend right, "\mathbb{L}(\mathsf{l}^G)"']
            \end{tikzcd}
            \end{equation*}
 \end{itemize}
\begin{proof}
The proof is similar as above using \cite[Theorem~7.2 (ii)]{Psaroud}.
\end{proof}
\end{prop}

\begin{cor}\label{derivedrecol-both-parts}
 Let $\mathsf{R}_{\mathsf{ab}}(\A,\B,\C)$ be a recollement of abelian categories and assume that $\B$ and $\C$ have enough projective and injective objects. Assume that a finite group $G$ acts on $\B$ with $|G|$ be invertible in $\B$ and that the recollement lifts to a $G$-equivariant recollement.
Then the following statements are equivalent:
\begin{itemize}
\item[(a)] The functors $\mathsf{i}\colon  \A \to \B$ is a homological embedding, $\mathsf{q}\colon  \B \to \A$ and $\mathsf{l}\colon \C\to \B$ are of locally finite homological dimension and $\mathsf{p}\colon  \B \to \A$ and $\mathsf{r}\colon \C \to \B$ are of locally finite cohomological dimension. 
\item[(b)] There exists a recollement of triangulated categories 
            \begin{equation*}
            \begin{tikzcd}
             \mDb(\A) \arrow[rr, "\mDb(\mathsf{i})" description]& & \mDb(\B) \arrow[ll, bend left, "\mathbb{R}\mathsf{p}"] \arrow[ll, bend right, "\mathbb{L}\mathsf{q}"'] \arrow[rr, "\mDb(\mathsf{e})" description] & & \mDb(\C) \arrow[ll, bend left, "\mathbb{R}\mathsf{r}"] \arrow[ll, bend right, "\mathbb{L}\mathsf{l}"']
            \end{tikzcd}
            \end{equation*}
 \item[(a-G)] The functor $\mathsf{i}^G\colon  \A^G \to \B^G$ is a homological embedding, the functors $\mathsf{q}^G\colon  \B^G \to \A^G$ and $\mathsf{l}^G\colon  \C^G \to \B^G$ are of locally finite homological dimension and $\mathsf{p}^G\colon  \B^G \to \A^G$ and $\mathsf{r}^G\colon  \C^G \to \B^G$ are of locally finite cohomological dimension. 
\item[(b-G)] There exists a recollement of triangulated categories 
            \begin{equation*}
            \begin{tikzcd}
             \mDb(\A^G) \arrow[rr, "\mDb(\mathsf{i}^G)" description]& & \mDb(\B^G) \arrow[ll, bend left, "\mathbb{R}(\mathsf{p}^G)"] \arrow[ll, bend right, "\mathbb{L}(\mathsf{q}^G)"'] \arrow[rr, "\mDb(\mathsf{e}^G)" description] & & \mDb(\C^G) \arrow[ll, bend left, "\mathbb{R}(\mathsf{r}^G)"] \arrow[ll, bend right, "\mathbb{L}(\mathsf{l}^G)"']
            \end{tikzcd}
            \end{equation*}
\end{itemize}
\begin{proof}
The result follows from Proposition~\ref{derivedrecol-left-part} and Proposition~\ref{derivedrecol-right-part} using also the fact that adjoints are unique up to natural isomorphism.
\end{proof}
\end{cor}

We finish this section with the following result.

\begin{thm}
\label{main5}
    Let $\mathsf{R}_{\mathsf{ab}}(\A,\B,\C)$ be a recollement of abelian categories and assume that $\B$ and $\C$ have enough projective and injective objects. Assume that a finite group $G$ acts on $\B$ with $|G|$ be invertible in $\B$ and that the recollement lifts to a $G$-equivariant recollement. In either setup of the Propositions~\ref{derivedrecol-left-part}, \ref{derivedrecol-right-part} or \ref{derivedrecol-both-parts}, we have a commutative diagram of recollements of the form:
\[
\begin{tikzcd}[scale cd=0.9]
& \mathsf{R}_{\mathsf{ab}}(\A,\B,\C) \arrow[dl, dotted] \arrow[dr, dotted] \\
\mathsf{R}_{\mathsf{tr}}(\mDb(\A) ,\mDb(\B), \mDb(\C) ) \arrow[d, dotted] & & \mathsf{R}_{\mathsf{ab}}(\A^G, \B^G, \C^G) \arrow[d, dotted] \\
\mathsf{R}_{\mathsf{tr}}(\mDb(\A)^G ,\mDb(\B)^G, \mDb(\C)^G ) \arrow[rr, "\simeq"] & & \mathsf{R}_{\mathsf{tr}}(\mDb(\A^G) ,\mDb(\B^G), \mDb(\C^G) )
\end{tikzcd}
\]
where the horizontal equivalence of recollements is induced by the comparison functors $K$ of Proposition~\ref{derivedequivariantequiv} and the dotted arrows are indicating the two different ways to construct the left and right recollements of triangulated categories.

\begin{proof}
We prove it in the setup of Proposition~\ref{derivedrecol-left-part}. The other two cases are similar.
Starting from the recollement $\mathsf{R}_{\mathsf{ab}}(\A, \B, \C)$, Proposition~\ref{derivedrecol-left-part} yields the top left dotted arrow and Theorem~\ref{main1} yields the top right dotted arrow. The right vertical dotted arrow exists also by Proposition~\ref{derivedrecol-left-part}. The $G$-action on $\mathsf{R}_{\mathsf{ab}}(\A,\B,\C)$ extends naturally to an action of derived categories by Remark~\ref{actiononderivedcats}. Since $\mathsf{e}$ is an exact functor, $\mDb(\mathsf{e})$ is naturally a $G$-functor by Lemma~\ref{D^b(e)Gfunctor}. Thus, by Theorem~\ref{main2}, we have the left vertical dotted arrow.
We have the following diagram that we need to show that commutes (up to natural isomorphism):

\begin{equation*}
\begin{tikzcd}
\mDb(\A)^G \arrow[ddd, "K_{\A}"] \arrow[rr, "\mDb(\mathsf{i})^G" description]& & \mDb(\B)^G \arrow[ddd, "K_{\B}"] \arrow[ll, bend left, "(\mathbb{R}\mathsf{p})^G"] \arrow[ll, bend right, "(\mathbb{L}\mathsf{q})^G"'] \arrow[rr, "\mDb(\mathsf{e})^G" description] & & \mDb(\C)^G \arrow[ddd, "K_{\C}"] \arrow[ll, bend left, "\mathsf{r}'"] \arrow[ll, bend right, "\mathsf{l}'"'] \\
 & \\
 \\
\mDb(\A^G) \arrow[rr, "\mDb(\mathsf{i}^G)" description]& & \mDb(\B^G) \arrow[ll, bend left, "\mathbb{R}(\mathsf{p}^G)"] \arrow[ll, bend right, "\mathbb{L}(\mathsf{q}^G)"'] \arrow[rr, "\mDb(\mathsf{e}^G)" description] & & \mDb(\C^G) \arrow[ll, bend left, "\mathsf{r}'"] \arrow[ll, bend right, "\mathsf{l}'"']
\end{tikzcd}
\end{equation*}
which means that the equivalence $K_{-}\colon  \mDb(-)^G \to \mDb(-^G) $ commutes with the corresponding adjoints.
We just need to show the commutativity with the middle functors, i.e.\ $\mDb(\mathsf{i})^G$ and $\mDb(\mathsf{e})^G$, and then the commutativity up to natural isomorphism with the adjoints follows immediately, since, for example $K_{\A}^{-1} \circ \mathbb{L}q^G \circ K_{\B}$ is a left adjoint of $\mDb(\mathsf{i})^G$ and by uniqueness of adjoints we have that it is naturally isomorphic to $ (\mathbb{L}\mathsf{q})^G$. The desired commutativity has been proved in Lemma~\ref{D(F^G)andD(F)^G}. 
%we get the desired commutativity and this completes the proof.
%The desired commutativity with the middle equivariant derived functors of $i$ and $e$ has been proved in Lemma~\ref{D(F^G)andD(F)^G}.
\end{proof}
\end{thm}

\section{Equivariant Singularity Categories}
\label{singularcats}

In this section we investigate singularity categories in equivariant abelian recollements. In particular, given a recollement of abelian categories $\mathsf{R}_{\mathsf{ab}}(\A, \B, \C)$ and a finite group action such that the recollement lifts to an equivariant recollement $\mathsf{R}_{\mathsf{ab}}(\A^G, \B^G, \C^G)$, in the sense of Theorem~\ref{main1},  we investigate when the quotient functor $\mathsf{e}^G$ induces a singular equivalence between the singularity categories $\mDsg(\B^G)$ and $\mDsg (\C^G)$. Moreover, we explore how this singular equivalence is related to the singular equivalence of $\mDsg(\B)$ and $\mDsg(\C)$ induced by $\mathsf{e}$. The case of the singular equivalence induced by $\mathsf{e}$ has been characterized in \cite[Theorem~5.2]{PSS}, which we will also use and extend in the case of equivariant categories. We begin by recalling some useful facts about singularity categories.

Buchweitz \cite{Buch} defined the singularity category $\mDsg(R)$ for a noetherian ring $R$ to be the Verdier quotient of the bounded derived category $\mDb(\smod R)$ modulo the homotopy category $\mKb (\proj R)$ of finitely generated projective $R$-modules, i.e.\ $\mDb(\smod R) / \mKb (\proj R)$. Later Orlov \cite{Orlov} defined the singularity category $\mDsg(X)$ of an algebraic variety $X$ as the Verdier quotient of the bounded derived category $\mDb (X)$ of coherent sheaves on $X$ quotient by the full subcategory of perfect complexes $\mathsf{perf}(X)$. In both settings, the singularity category captures many geometric properties.
%Given a smooth algebraic variety $X$ its singularity category is trivial (but not the converse) so the singularity category captures a lot of geometric properties of the variety. With this point of view, in the algebraic setup, the singularity category of a ring $R$ measures the singularities of the spectrum of the ring $R$. 
In the algebraic setup, the singularity category $\mDsg(\Lambda)$ of an Artin algebra $\Lambda$ is trivial if and only if $\Lambda$ has finite global dimension. 
We recall below the definition of the singularity category in the context of abelian categories.

\begin{defn}\label{dsgdef}
Let $\A$ be an abelian category with enough projectives. Denote by $\mKb (\Proj\A)$ the homotopy category of bounded complexes of projectives in $\A$. The \textbf{singularity category} of $\A$ is the Verdier quotient:
\[
\mDsg (\A) \coloneqq \mDb (\A) / \mKb(\Proj\A)
\]
\end{defn}

The singularity category carries a unique triangulated structure such that the quotient functor $Q_{\A} \colon \mDb(\A) \to \mDsg(\A) $ is triangulated. The objects of the singularity category are the objects of the bounded derived category and morphisms $f\colon X^{\bullet} \to Y^{\bullet}$ are equivalence classes of fractions $(X^{\bullet} \leftarrow L^{\bullet}  \rightarrow Y^{\bullet} )$ such that the cone of $L^{\bullet}  \to X^{\bullet} $ lies in $\mKb( \Proj \A)$. The exact triangles in $\mDsg (\A)$ are images via $Q_{\A}$ of exact triangles in $\mDb(\A)$.  %triangles such that their image via $Q_{\A}$ is isomorphic to an exact triangle in $\mDb(\A)$.

Note that, like in the previous sections, all actions are right ones.

%%%%%%%%%%%%%%%%%%%%%%%%%%%%%%%%%%%%%%%%%%%%%%%%%%%%%%%%%%%
\subsection{Equivariant Singularity Categories}
%%%%%%%%%%%%%%%%%%%%%%%%%%%%%%%%%%%%%%%%%%%%%%%%%%%%%%%%%%%
In this section we investigate how the action on $\A$ induces naturally an action on $\mDsg(\A)$ and how the equivariant singularity category $\mDsg(\A)^G$ is related to the singularity category $\mDsg (\A^G)$.
We begin this section with two useful remarks on projectives and injectives in the equivariant case.
Recall that the forgetful functor $\mathsf{For}$ and the induction $\mathsf{Ind}$, being bi-adjoint, both preserve projective and injective objects. 

\begin{rem}
Let $\A$ be an abelian category and $G$ a finite group acting on $\A$. Then we have that %projective objects of the equivariant category $\A^G$ are 
$\Proj(\A^G) = \mathsf{Add}\{ \mathsf{Ind}(P) \ | \ P \in \Proj\A \}$, where by $\mathsf{Add}$ we mean summands of coproducts. 
%So $(P, \pi)$ is projective if and only if $P$ is projective. 
We also have the same description for injectives, i.e.\ $\Inj(\A^G) = \mathsf{Add}\{ \mathsf{Ind}(I) \ | \ I \in \Inj\A \}$.

Indeed, for every $(X, \chi) \in \A^G$ we have an epimorphism $\mathsf{Ind} \circ \mathsf{For} (X, \chi) \twoheadrightarrow (X,\chi)$ induced by the counit of the adjunction $\mathsf{Ind} \dashv  \mathsf{For}$. It is an epimorphism since the forgetful functor  $\mathsf{For}$ is faithful. So, if $(P, \pi)$ is projective in $\A^G$ then the above map splits and thus we get the desired description of projectives. 
%$P$ would have to be a direct summand of the projective $\mathsf{Ind}(P)$. 
\end{rem}

The following Lemma describes the projectives and the injectives of the equivariant category. This lemma was proven in the general context of exact categories but we state it for abelian categories, to match the context of this paper.

\begin{lem}\label{equivariantprojinj}\textnormal{(\!\!\cite[Lemma~3.12]{ChaoSun})}
Let $\A$ be an abelian category together with an action of a finite group $G$ such that $|G|$ is invertible in $\A$. Then $\Proj(\A^G) = (\Proj\A)^G$ and $\Inj(\A^G) = (\Inj\A)^G$.
\end{lem}

Elagin  \cite{Elagin} and Chen  \cite{XiaoWuChen2} had already noticed that their result (see Theorem~\ref{derivedequivariantequiv}) holds for unbounded derived categories if we assume idempotent completion of the derived equivariant category. Chao Sun generalised their results in \cite[Examples~3.19,~3.20]{ChaoSun} in the following way. 

Let $\A$ be an additive category and let $\mc ^*(\A)$ (resp.\ $\mK^*(\A)$ and, if $\A$ is abelian with enough projectives, $\mD^*(\A)$) with $* \in \{\varnothing, +, -, \mathsf{b} \}$ denote the category of complexes (resp.\ homotopy and derived category) of unbounded, bounded bellow, bounded above and bounded complexes of $\A$ respectively. Let $G$ be a finite group acting on $\A$ with $|G|$ invertible in $\A$. There is a natural $G$-action on $\mc^*(\A)$ (induced by acting component-wise, as in Remark~\ref{Gactiononderivedcat}) and the comparison functor (defined similarly as in subsection~\ref{actiononderivedcats}) yields a canonical equivalence $K \colon \mc ^*(\A^G) \xrightarrow{\simeq} \mc^*(\A)^G$.
There is also an induced admissible $G$-action on $\mK^*(\A)$, with $\mK^*(\A)^G$ canonically triangulated, and Sun proved that the functor $K$ induces the following triangle equivalence up to retracts:
\[
K \colon \mK^* (\A^G) \to \mK^*(\A)^G
\]
If $\A$ is abelian with enough projectives, then $\mD^*(\A)$ carries an admissible $G$-action with $\mD^*(\A)^G$ being canonically triangulated and the above comparison functor induces a triangle equivalence up to retracts $K\colon \mD^*(\A^G) \to \mD^*(\A)^G$. Notice that, when the derived categories involved are bounded, they are always idempotent complete by \cite[Corollary~2.10]{BalSchil}.

Let $\A$ be an abelian category. Then the projective objects $\Proj\A$ is an idempotent complete additive category. Thus, by \cite[Lemma~2.3]{XiaoWuChen2},  $(\Proj\A)^G$ is also an idempotent complete additive category, which implies that $\mKb ((\Proj \A)^G) $ is idempotent complete by \cite[Theorem~2.9]{IyamaYang}. In this case, Lemma~\ref{equivariantprojinj} and the above comparison functor yield the following equivalence:
\begin{equation}\label{Kbprojcomparison}
\mKb (\Proj \A^G) = \mKb ((\Proj \A)^G) \xrightarrow[]{\simeq} \mKb (\Proj \A)^G
\end{equation}
This last equivalence can be interpreted as the restriction of the comparison functor on $\mDb(\A^G)$ to the full subcategory $ \mKb (\Proj \A^G)$. Notice that the canonical triangulated structure of $\mKb(\Proj \A)^G $ can be obtained also by the fact that  $\mKb(\Proj \A)$ is a $G$-invariant subcategory of $\mDb (\A)$. Indeed, $\mKb(\Proj \A)^G$ admits a unique canonical triangulated structure such that the inclusion
$\mKb (\Proj \A)^G \hookrightarrow \mDb (\A)^G $ is triangulated functor.
Thus, the Verdier quotient $\mDb(\A)^G / \mKb (\Proj \A)^G$ is defined. Therefore, we have an equivalence realised by the comparison functor:
\begin{equation}\label{comparisonfuctoronquotients}
K \colon \mDb (\A^G) / \mKb (\Proj \A^G) \xrightarrow{\simeq} \mDb (\A)^G / \mKb (\Proj\A)^G 
\end{equation}
Note that the left hand side is the singularity category of $\A^G$ (see Definition~\ref{dsgdef}). Thus, we have that the comparison functor induces an equivalence:
\begin{equation}
K \colon \mDsg (\A^G) \xrightarrow{\simeq} \mDb (\A)^G / \mKb (\Proj\A)^G 
\end{equation}

In the next result, we show how the $G$-action on $\A$ induces a $G$-action on the singularity category $\mDsg(\A)$.

\begin{prop}\label{equivsing=singequiv}
   Let $\A$ be an abelian category with enough projectives. Assume that a finite group $G$ acts on $\A$ with $|G|$ invertible in $\A$. Then the singularity category $\mDsg(\A)$ carries an admissible $G$-action, its equivariant category is canonically triangulated and there is a triangle equivalence up to retracts:
\[ 
F \colon \mDb (\A)^G / \mKb(\Proj\A)^G \xrightarrow[]{}  (\mDb (\A) / \mKb (\Proj \A))^G \eqqcolon \mDsg (\A)^G
\]

\begin{proof}
By the preceding discussion, we have a natural $G$-action on $\mDb(\A)$ for which $\mKb(\Proj \A)$ is a $G$-invariant subcategory. Thus, Theorem~\ref{quotienttheorem} yields that there exists a natural admissible $G$-action on the quotient $\mDb (\A) / \mKb (\Proj\A)$ and the equivariant category $(\mDb (\A) / \mKb (\Proj\A))^G$ is canonically triangulated.
The triangle functor $F$ is induced by the quotient functor 
$$Q' \colon \mDb (\A)^G \to \mDb (\A)^G/\mKb(\Proj\A)^G$$ 
in the following way. By Remark~\ref{quotientGfunctor}, the quotient functor $Q_{\A} \colon \mDb (\A) \to \mDsg(\A)$ is a $G$-functor and $\mKb (\Proj\A)$ lies in its kernel. Hence  $\mKb (\Proj \A)^G$ lies in the kernel of $Q_{\A}^G \colon \mDb (\A)^G \to (\mDb (\A)/\mKb (\Proj \A))^G $ by Lemma~\ref{Ginvariantsubcats}. Using the universal property for the quotients we have the following commutative diagram
\begin{equation}\label{factorization of equivariant quotient}
    \begin{tikzcd}
        \mDb (\A)^G \arrow[rr, "Q' "] \arrow[dr, "Q_{\A}^G"']& &  \mDb (\A)^G/\mKb (\Proj \A)^G \arrow[dl, dashrightarrow, "F"] \\
         & (\mDb (\A)/\mKb (\Proj \A))^G
    \end{tikzcd}
\end{equation}
where the functor $F$ is the unique functor that completes this into a commutative diagram and it is an equivalence up to retracts (for more details see the proof of \cite[Theorem~3.9]{ChaoSun}).
\end{proof}
\end{prop}
Combining the equivalence~$(\ref{comparisonfuctoronquotients})$, the definition of singularity category and Proposition~\ref{equivsing=singequiv} we obtain the following.

\begin{prop}\label{singequivuptoretracts}
Let $\A$ be an abelian category with enough projectives. Assume a finite group $G$ acts on $\A$ with $|G|$ invertible in $\A$. Then we have that
$$\mDsg (\A^G) \coloneqq  \mDb (\A^G) / \mKb(\Proj \A^G) \simeq \mDb (\A)^G / \mKb(\Proj\A)^G \to \mDsg (\A)^G  $$
where the rightmost map is an equivalence up to retracts.
Moreover, there exists the following commutative diagram:
\begin{equation}\label{commutativityofdiagramswithsingularities}
    \begin{tikzcd}
        \mDb (\A^G) \arrow[d, "K"'] \arrow[r, "Q_{\A^G}"] & \frac{\mDb(\A^G)}{\mKb (\Proj\A^G)} \arrow[d, "K"'] \arrow[r, "="] & \mDsg (\A^G) \arrow[d, dashrightarrow, "F'"]\\
         \mDb(\A)^G \arrow[r, "Q'"'] & \frac{\mDb(\A)^G}{\mKb (\Proj\A)^G} \arrow[r, "F"'] &  \mDsg (\A)^G
    \end{tikzcd}
\end{equation}
\begin{proof}
The left and middle vertical arrows are the equivalences induced by the comparison functor. The right vertical arrow $F'$ is induced by $F$ and is the equivalence up to retracts described in Proposition~\ref{equivsing=singequiv}. The commutativity of the left square is evident since we have the horizontal quotient functors and that $\mKb (\Proj\A^G) \simeq \mKb (\Proj\A)^G $ by equation~$(\ref{Kbprojcomparison})$.
\end{proof}
\end{prop}

We close this subsection with a careful examination of the action on $\mDsg(\A)$. 
\begin{rem}\label{Gactiononsingularity}
The action of $G$ on objects is the same as on objects of $\mDb (\A)$ (see Remark~\ref{actiononD^b}). The action on morphisms has to be well defined - that is the reason why the subcategory we quotient with has to be $G$-invariant; so that Theorem~\ref{quotienttheorem} works. Another key point is that $G$ is acting admissibly:

A morphism $f\colon X^{\bullet} \to Y^{\bullet}$ in $\mDsg (\A) $ is of the form $  X^{\bullet} \xleftarrow[]{f_1} W^{\bullet}  \xrightarrow[]{f_2} Y^{\bullet} $, where $\mathsf{Cone}(f_1) \in \mKb(\Proj\A)$. 
So, the morphism $f^g\colon (X^{\bullet})^g \to (Y^{\bullet})^g$ is of the form $ (X^{\bullet})^g \xleftarrow[]{f_1^g} (W^{\bullet})^g  \xrightarrow[]{f_2^g} (Y^{\bullet})^g$. Now, we have that $\mathsf{Cone}(f_1)^g \simeq \mathsf{Cone}(f_1^g)$. Indeed, we have the following diagram of exact triangles:
\begin{equation*}
    \begin{tikzcd}
        (W^{\bullet})^g \arrow[r, "f_1^g"] \arrow[d, equal ] & (X^{\bullet})^g \arrow[d, equal] \arrow[r] & \mathsf{Cone}(f_1^g) \arrow[r] \arrow[d, dashrightarrow, ""] & (W^{\bullet})^g[1] \arrow[d, "\simeq"] \\
        (W^{\bullet})^g \arrow[r, "f_1^g"] & (X^{\bullet})^g \arrow[r] & \mathsf{Cone}(f_1)^g \arrow[r] & (W^{\bullet}[1])^g
    \end{tikzcd}
\end{equation*}
The first two equalities are by definition of the action and the  rightmost isomorphism is due to the $G$-action being admissible (i.e.\ $\rho_g[1] \simeq~[1]\rho_g$). Since there are three vertical morphisms, we can complete the missing one (dashing arrow) and since they are isomorphisms so is this one. Thus $\mathsf{Cone}(f_1^g) \simeq  \mathsf{Cone}(f_1)^g$ and since $\mathsf{Cone}(f_1) \in \mKb(\Proj\A)$ and this is a $G$-invariant subcategory, we have that $\mathsf{Cone}(f_1^g) \in \mKb(\Proj\A)$. 
\end{rem}

\subsection{Equivariant Singular Equivalences}
In this subsection we investigate singular equivalences in equivariant recollements of abelian categories. We first recall the following result. 

\begin{thm}\label{pss}\textnormal{(\!\!\cite[Theorem~5.2]{PSS})}
    Let $\mathsf{R}_{\mathsf{ab}}(\A, \B, \C)$ be a recollement of abelian categories and $\B$ and $\C$ have enough projectives. The following are equivalent:
    \begin{itemize}
        \item[(i)] $\pd_{\B}\mathsf{i}(A) < \infty$ and $\pd_{\C}\mathsf{e}(P) < \infty$ for all $A \in \A$ and $P\in \Proj\B$.
        \item[(ii)]  The functor $\mathsf{e}\colon\B \to \C$ induces a singular equivalence between $\B$ and $\C$: 
\[        
\mDsg (\mathsf{e})\colon \mDsg (\B)\xrightarrow[]{\simeq} \mDsg (\C)
\]
\end{itemize}
\end{thm}

The existence of $\mDsg(\mathsf{e})$ is related to the following commutative diagram:
\begin{equation}\label{commutativediagramofsingularcats}
\begin{tikzcd}
 \mDb(\B) \arrow[d, "\mDb(\mathsf{e})"'] \arrow[r, "Q_{\B}"] & \mDsg (\B) \arrow[d, "\mDsg (\mathsf{e})"] \\ 
 \mDb(\C) \arrow[r, "Q_{\C}"'] & \mDsg (\C)
\end{tikzcd}
\end{equation}

We aim to extend the above commutative square to a commutative square of equivariant categories and functors.  Note that even if the square of the underlying functors is commutative, the equivariant square does not always commute. This is due to the extra stucture of the $G$-functors.
We start with the following Lemma, which is an analogue of Lemma~\ref{D^b(e)Gfunctor}.

\begin{lem}\label{D_sg(e)Gfunctor}
    Let $\mathsf{e}\colon\B \to \C$ be an exact functor between abelian categories with enough projectives such that the singularity functor $\mDsg (\mathsf{e}) \colon \mDsg (\B) \to \mDsg (\C)$ exists. Let also $G$ be a finite group acting on $\B$ and $\C$ with $|G|$ invertible in both categories such that $(\mathsf{e}, \sigma^\mathsf{e})$ is a $G$-functor. There exist canonical $G$-actions on $\mDsg(\B)$ and $\mDsg(\C)$ and a family $\{ \sigma^{\mDsg (\mathsf{e})} \}_{g\in G}$ so that  $(\mDsg (\mathsf{e}), \sigma^{\mDsg (\mathsf{e})})$ is a $G$-functor. 
\end{lem}
\begin{proof}
The canonical $G$-actions on $\mDsg(\B)$ and $\mDsg(\C)$ are described in Remark~\ref{Gactiononsingularity}. Because $\mathsf{e}$ is exact we write $\mDb(\mathsf{e}) \eqqcolon \mathsf{e}$ for its derived functor and recall that it is a $G$-functor with family of natural isomorphisms $\sigma^{\mDb(\mathsf{e})}$ induced by applying $\sigma^\mathsf{e}$ component-wise. For $\{ \sigma^{\mDsg (\mathsf{e})}\}_{g\in G}$, we have that it is induced by $\sigma^\mathsf{e}$ on objects in the same way as in diagram~$(\ref{natransformobjects})$. For a morphism $f\colon X^{\bullet} \to Y^{\bullet}$ in the singularity category, we have the class of equivalent morphisms of the form:
\begin{equation*}
\begin{tikzcd}
X^\bullet &  \mathsf{e}((X^{\bullet})^g) \arrow[r, "\sigma^e"] & (\mathsf{e}(X^{\bullet}))^g\\
W^{\bullet} \arrow[u, "f_1"] \arrow[d, "f_2"'] & \mathsf{e}((W^{\bullet})^g) \arrow[d, "\mathsf{e}(f_2^g)"'] \arrow[u, "\mathsf{e}(f_1^g)"] \arrow[r, "\sigma^\mathsf{e}"] & (\mathsf{e}(W^{\bullet}))^g  \arrow[u, "(\mathsf{e}f_1)^g"'] \arrow[d, "(\mathsf{e}f_2)^g"]\\
Y^{\bullet} & \mathsf{e}((Y^{\bullet})^g) \arrow[r, "\sigma^\mathsf{e}"] & (\mathsf{e}Y^{\bullet})^g
\end{tikzcd}
\end{equation*}
where each square above commutes (same as we showed in the commutative diagram of complexes in diagram~$(\ref{natransformmorphisms})$). Note that we use the notation $\sigma^\mathsf{e}$  since the natural isomorphisms $\{ \sigma^{\mDsg (\mathsf{e})} \}_{g\in G}$ are induced by $\sigma^\mathsf{e}$ which is applied component-wise. 
\end{proof}

\begin{prop}
    Let $\mathsf{e}\colon \B \to \C$ be an exact functor of abelian categories with enough projectives such that there exist an induced functor between their singularity categories rendering the following square commutative:
\begin{equation}
\begin{tikzcd}
\mDb(\B) \arrow[r, "Q_{\B}"] \arrow[d, "\mDb(\mathsf{e})"']& \mDsg (\B) \arrow[d, "\mDsg (\mathsf{e})"] \\
\mDb(\C) \arrow[r, "Q_{\C}"'] & \mDsg (\C)
\end{tikzcd}
\end{equation}
Assume that a finite group $G$ is acting on both $\B$ and $\C$ with $|G|$ invertible in $\B$ and $\C$. Assume further that $(\mathsf{e}, \sigma^\mathsf{e})$ is a $G$-functor. Then there exist canonical families of natural isomorphisms making all the functors in the above square $G$-functors and the following square of equivariant categories commutes:
\begin{equation}\label{equivsquare2}
\begin{tikzcd}
\mDb(\B)^G \arrow[r, "Q_{\B}^G"] \arrow[d, "\mDb(\mathsf{e})^G"']& \mDsg (\B)^G \arrow[d, "\mDsg (\mathsf{e})^G"] \\
\mDb(\C)^G \arrow[r, "Q_{\C}^G"'] & \mDsg (\C)^G
\end{tikzcd}
\end{equation}
\end{prop}

\begin{proof}
Recall that by Remark~\ref{Gactiononderivedcat} and Proposition~\ref{equivsing=singequiv}, the $G$ action on $\B$ and $\C$ extends canonically to a $G$ action on their respective derived  and singularity categories. 
By Remark~\ref{quotientGfunctor} we know that $(Q_{\B}, \sigma^{\B})$ and $(Q_{\C}, \sigma^{\C})$ are $G$-functors and thus we have $Q_{\B}^G \colon \mDb(\B)^G \xrightarrow{} \mDsg (\B)^G$ and  $Q_{\C}^G \colon \mDb(\C)^G \xrightarrow{} \mDsg (\C)^G$.
By Lemmas~\ref{D^b(e)Gfunctor} and~\ref{D_sg(e)Gfunctor} we have that
$\mDb(\mathsf{e})$ and $\mDsg (\mathsf{e})$ are endowed with $G$-functor structure in a canonical way (induced by $\sigma^\mathsf{e}$). Thus, we have induced equivariant functors $\mDb(\mathsf{e})^G \colon \mDb(\B)^G \xrightarrow[]{} \mDb(\C)^G$ and $\mDsg (\mathsf{e})^G\colon \mDsg (\B)^G \xrightarrow[]{} \mDsg (\C)^G$.
\par
We have to verify that the square of equivariant categories~$(\ref{equivsquare2})$ is commutative.
By Remark~\ref{compositionGfunctors}, have to verify the equality $(\sigma^{Q_{\C}} \mDb(\mathsf{e}))\circ (Q_{\C}\sigma^{\mDb(\mathsf{e})} )= \linebreak ( \sigma^{\mDsg (\mathsf{e})}Q_{\B} ) \circ (\mDsg (\mathsf{e})\sigma^{Q_{\B}})$. For simplicity we write $\sigma_1$ for the left hand side of the desired equality and $\sigma_2$ for the right hand side. In order for them to be equal one would have to check that they coincide on objects, i.e.\ $\sigma_1^{X^{\bullet}} = \sigma_2^{X^{\bullet}}$. 
We can immediately see that 
$\sigma_{i,g}^{X^{\bullet}} \colon \mathsf{e} ((X^{\bullet})^g) \xrightarrow[]{\simeq}  (\mathsf{e}(X^{\bullet}))^g$ 
are the same isomorphism for all $X^{\bullet} \in \Dsg(\B)$, $g \in G$ and $i=1,2$ since $\sigma^{Q_{\B}}, \sigma^{Q_{\C}} $ and $Q_{\B}, Q_{\C}$ are identities on objects and both $\mDb(\mathsf{e}), \mDsg (\mathsf{e})$ and their families of natural isomorphisms are induced by $\mathsf{e}$ as well as its family of natural isomorphisms as described in Lemma~\ref{D_sg(e)Gfunctor}.
\end{proof}

\begin{cor}\label{equivrecollsquare}
Let $\mathsf{R}_{\mathsf{ab}}(\A, \B, \C)$ be a recollement of abelian categories that lifts to a $G$-equivariant recollement with $|G|$ invertible in $\B$. Assume also that $\B$ and $\C$ have enough projectives and that $\mDsg(\mathsf{e}) \colon \mDsg(\B) \to \mDsg(\C)$ exists.
 Then there exists the following commutative square of equivariant categories:
\begin{equation}
\begin{tikzcd}
\mDb(\B)^G \arrow[r, "Q_{\B}^G"] \arrow[d, "\mDb(\mathsf{e})^G"']& \mDsg (\B)^G \arrow[d, "\mDsg (\mathsf{e})^G"] \\
\mDb(\C)^G \arrow[r, "Q_{\C}^G"'] & \mDsg (\C)^G
\end{tikzcd}
\end{equation}

\end{cor}

We are ready to formulate the first main result of this section which yields the desired relation between the singularity functors $\mDsg(\mathsf{e})$,  $\mDsg(\mathsf{e}^G)$ and $\mDsg(\mathsf{e})^G$.

\begin{thm}\label{main3}
    Let $\mathsf{R}_{\mathsf{ab}}(\A, \B, \C)$ be a recollement of abelian categories that lifts to a $G$-equivariant recollement with $|G|$ invertible in $\B$. 
    Let also $\B$ and $\C$ have enough projectives. Assume that $\mDsg(\mathsf{e})\colon \mDsg(\B) \to \mDsg(\C)$ exists. Then there exists a singularity functor $\mDsg (\mathsf{e}^G)\colon \mDsg (\B^G) \xrightarrow[]{} \mDsg (\C^G) $ rendering the following diagram commutative. 
    \begin{equation}\label{inducedsignularequiv}
    \begin{tikzcd}
    \mDsg (\B^G) \arrow[r, "F'_{\B}"] \arrow[d, dotted, "\mDsg (\mathsf{e}^G)"']& \mDsg (\B)^G \arrow[d, "\mDsg (\mathsf{e})^G"] \\
    \mDsg (\C^G) \arrow[r, "F'_{\C}"'] & \mDsg (\C)^G
    \end{tikzcd}
    \end{equation}
    If additionally,  $\mDsg (\mathsf{e})\colon \mDsg (\B) \xrightarrow[]{} \mDsg (\C)$ is an equivalence and both $\mDsg(\B^G)$ and $\mDsg(\C^G)$ are idempotent complete categories, then $\mDsg (\mathsf{e}^G)$ is a singular equivalence.
\end{thm}

\begin{proof}
Recall that $|G|$ being invertible in $\B$ implies that it is also invertible in $\C$ by Remark~\ref{invertibilityonsubcats}. Thus $\B^G$ and $\C^G$ have enough projectives by Remark~\ref{enoughprojinj} and the action extends 
%to an admissible action in their respective derived categories by Remark \ref{Gactiononderivedcat} and 
to an action on the singularity categories by Proposition~\ref{equivsing=singequiv}.
Recall also that by diagram~$(\ref{commutativityofdiagramswithsingularities})$ for the abelian category $\B$ (resp.\ $\C$)  we have that there is triangle equivalence up to retracts $F'_{\B} \colon \mDsg(\B^G) \xrightarrow[]{} \mDsg (\B)^G$ (resp.\ $F'_{\C}$).

We prove that in our recollement setup there exists $\mDsg (\mathsf{e}^G)$ rendering the following diagram commutative:
\begin{equation*}
\begin{tikzcd}
\mDb(\B)^G \arrow[ddd, bend right=60, "\mDb(\mathsf{e})^G"'] \arrow[d, leftarrow, "K_{\B}"'] \arrow[r, "Q_{\B}^G"] & \mDsg (\B)^G \arrow[ddd, bend left = 60, "\mDsg (\mathsf{e})^G"] \\
\mDb(\B^G) \arrow[r, "Q_{\B^G}"] \arrow[d, "\mDb(\mathsf{e}^G)"']& \mDsg (\B^G) \arrow[d, dotted, "\mDsg (\mathsf{e}^G)"] \arrow[u, "F'_{\B}"']\\
\mDb(\C^G) \arrow[r, "Q_{\C^G}"] & \mDsg (\C^G) \arrow[d, "F'_{\C}"] \\
\mDb(\C)^G \arrow[u, leftarrow, "K_{\C}"] \arrow[r, "Q_{\C}^G"] & \mDsg (\C)^G
\end{tikzcd}
\end{equation*}
Commutativity of the outer bended square follows by Corollary~\ref{equivrecollsquare}. Commutativity of top and bottom squares is due to diagram~$(\ref{commutativityofdiagramswithsingularities})$ combined with the commutativity of diagram~$(\ref{factorization of equivariant quotient})$.
The left bent diagram commutes by Lemma~\ref{D(F^G)andD(F)^G}.

We claim that there is an induced functor $\mDsg (\mathsf{e}^G)$ such that the middle square is commutative. Indeed, notice that since $\mDb(\mathsf{e}) (\mKb(\Proj\B)) \subseteq \mKb(\Proj\C)$ and both $\mKb(\Proj \B)$ and $\mKb(\Proj \C)$ are $G$-invariant subcategories, we have, by Lemma~\ref{Ginvariantsubcats}, that $\mDb(\mathsf{e})^G \big(\mKb(\Proj\B)^G\big)$ lies in $ \mKb(\Proj\C)^G$. From the fact that $\mKb(\Proj\B^G) \simeq \mKb(\Proj\B)^G$ and $\mKb(\Proj\C^G) \simeq \mKb(\Proj\C)^G$ via the comparison functor, and using the above diagram we obtain that 
\begin{align*}
\mDb(\mathsf{e}^G)(\mKb(\Proj\B^G)) &= K_{\C}^{-1}\mDb(\mathsf{e})^G(K_{\B}(\mKb(\Proj\B^G))) \\ 
										& \simeq K_{\C}^{-1}\mDb(\mathsf{e})^G(\mKb(\Proj\B)^G) \\
& \subseteq   K_{\C}^{-1} \mKb(\Proj\C)^G \\
& \simeq   \mKb(\Proj\C^G)
\end{align*}
This yields that $Q_{\C^G} \circ \mDb(\mathsf{e}^G)$ annihilates $\mKb(\Proj(\B^G))$ and thus $\mDsg (\mathsf{e}^G)$ is induced rendering the middle square commutative. From the commutativity of the diagram so far, we have that $F'_{\C} \circ \mDsg(\mathsf{e}^G) \circ Q_{\B^G} = \mDsg(\mathsf{e})^G \circ F'_{\B} \circ Q_{\B^G}$. Using the universal property of the quotient functor $Q_{\B^G}$, we obtain that  $F'_{\C} \circ \mDsg(\mathsf{e}^G) = \mDsg(\mathsf{e})^G \circ F'_{\B} $, i.e.\ the desired commutativity of diagram~$(\ref{inducedsignularequiv})$.

If, moreover, $\mDsg (\mathsf{e})$ is a singular equivalence, then we have that $\mDsg (\mathsf{e})^G$ is also a singular equivalence by Corollary~\ref{G functor equivalence}.
Assuming that  $\mDsg(\B^G)$ and $\mDsg(\C^G)$ are idempotent complete, then by Proposition $\ref{singequivuptoretracts}$ we have the following equivalences $F'_{\B} \colon \mDsg (\B^G) \xrightarrow[]{\simeq} \mDsg (\B)^G$ and $F'_{\C} \colon\mDsg (\C^G) \xrightarrow[]{\simeq}~\mDsg(\C)^G$. We infer that $\mDsg (\mathsf{e}^G)$ is an equivalence by the commutativity of diagram~$(\ref{inducedsignularequiv})$.
\end{proof}

%{\color{red} From Theorem~\ref{main3} we have the following homological properties.
%
%\begin{rem}
%In the setup of Theorem~\ref{main3}, we have that for any object $(P,\pi)$ in $\mKb (\Proj \B^G)$, there is some $n_0 $ such that $\Ext^n_{\C^G}(e^G(P,\pi), (C,c))=0$ for all $n>n_0$. Moreover, by Theorem~\ref{pss} we have
%$\pd_{\B^G} i^G(A, \alpha) < \infty $ and $\pd_{\C^G} e^G (P,\phi) < \infty$ for any $(A, \alpha) \in \A^G$ and $(P,\phi) \in \Proj(\B^G)$.
%\end{rem}}

It is natural to ask if the converse of Theorem~\ref{main3} holds. Indeed we have that it holds under some mild assumptions. 

\begin{thm}\label{main4}
Let $\mathsf{R}_{\mathsf{ab}}(\A, \B, \C)$ be a recollement of abelian categories that lifts to a $G$-equivariant recollement with $|G|$ invertible in $\B$. 
    Let also $\B$ and $\C$ have enough projectives. Assume that $\mDsg(\mathsf{e}^G)\colon \mDsg(\B^G) \to \mDsg(\C^G)$ exists. Then there exists a singularity functor $\mDsg (\mathsf{e})\colon \mDsg (\B) \xrightarrow[]{} \mDsg (\C) $.
    If additionally,  $\mDsg (\mathsf{e}^G)$ is a singular equivalence, then $\mDsg (\mathsf{e})$ is a singular equivalence.
\begin{proof}
The existence of $\mDsg(\mathsf{e}^G)\colon \mDsg(\B^G) \to \mDsg(\C^G)$ implies $\pd_{\C^G}\mathsf{e}^G(P,\pi) <\infty $ for all $(P, \pi) \in \Proj (\B^G)$. We will show  that $\pd_{\C}\mathsf{e}(P) < \infty$ for all $P \in \Proj \B$ which again implies the existence of $\mDsg (\mathsf{e})\colon \mDsg (\B) \xrightarrow[]{} \mDsg (\C)$ . Indeed, both $\mathsf{Ind} $ and $\mathsf{For}$ are exact and preserve projectives since they are biajoint functors, hence $\mathsf{e}^G \mathsf{Ind}(P)$ has finite projective dimension. By the equivalence $\mathsf{e}^G \mathsf{Ind} \simeq \mathsf{Ind} \, \mathsf{e}$ we infer that $\mathsf{Ind}(\mathsf{e}P)$ has finite projective dimension. 
Applying the forgetful functor, we have that $\mathsf{For} ( \mathsf{Ind}(\mathsf{e}P) ) = \oplus (\mathsf{e}P)^g$ has finite projective dimension because it preserves the finite projective resolution of $\mathsf{Ind}(\mathsf{e}P)$. Hence, $ (\mathsf{e}P)^{1_G} \simeq \mathsf{e}(P)$ is a summand of an object with finite projective dimension and thus it has finite projective dimension.

Assume that $\mDsg (\mathsf{e}^G)\colon \mDsg (\B^G) \xrightarrow[]{} \mDsg (\C^G)$ is an equivalence. By Theorem~\ref{pss}, this is equivalent to  $\pd_{\C^G}\mathsf{e}^G(P,\pi) < \infty $ and $\pd_{\B^G}\mathsf{i}^G(A,\alpha) < \infty $ for all $(P, \pi)$ in $\Proj (\B^G)$ and $(A,\alpha)$ in $\A^G$, respectively. It suffices to show that  $\pd_{\C}\mathsf{e}(P) < \infty$ for all $P \in \Proj \B$ and  $\pd_{\C}\mathsf{i}(A) < \infty$ for all $A \in \A$. The first part has been already proved. For the second part we have that $\mathsf{Ind}(\mathsf{i}A)$ has finite projective dimension since it is isomorphic to $\mathsf{i}^G \mathsf{Ind}(A)$. By the same interplay of induction and forgetful functor, we obtain that $ (\mathsf{i}A)^{1_G} \simeq \mathsf{i}(A)$ is a summand of  some object of finite projective dimension and thus it has finite projective dimension.
\end{proof}
\end{thm}

We can combine Theorem~\ref{main3} and Theorem~\ref{main4} and formulate the following main result regarding equivariant singular equivalences.

\begin{thm}\label{main6}
Let $\mathsf{R}_{\mathsf{ab}}(\A, \B, \C)$ be a recollement of abelian categories that lifts to a $G$-equivariant recollement with $|G|$ invertible in $\B$. Assume that $\B$ and $\C$ have enough projectives. Consider the following statements:
    \begin{itemize}
        \item[(i)] The functor $\mathsf{e}\colon \B \to \C$ induces a singular equivalence.
        \vspace{1mm}
        \item[(ii)] The functor $\mathsf{e}^G\colon \B^G \to \C^G$ induces a singular equivalence.
    \end{itemize}
Then \textnormal{(ii)} $\Longrightarrow$ \textnormal{(i)}.
Moreover, if $\mDsg(\B^G)$ and $\mDsg(\C^G)$ are idempotent complete then \textnormal{(i)} $\Longrightarrow$ \textnormal{(ii)}. In this case we have the following commutative diagram:
\begin{equation*}
\begin{tikzcd}
\mDsg( \B^G) \arrow[r, "\mDsg(\mathsf{e}^G)"]  \arrow[d, "F'_{\B}"'] & \mDsg(\C^G)  \arrow[d, "F'_{\C}" ] \\
\mDsg(\B)^G \arrow[r, "\mDsg(\mathsf{e})^G "]  \arrow[d, "\mathsf{For}"'] & \mDsg(\C)^G  \arrow[d, "\mathsf{For}"] \\
\mDsg(\B) \arrow[r, "\mDsg(\mathsf{e})"] & \mDsg(\C)
\end{tikzcd}
\end{equation*}
where $F'$ is the equivalence up to retracts of Proposition~\ref{singequivuptoretracts}.
\begin{proof}
The implication (ii) $\Longrightarrow$ (i) follows by Theorem~\ref{main4} and the converse is implied by Theorem~\ref{main3}. The commutativity of the diagram is immediate from the proof of Theorem~\ref{main3} and diagram~$(\ref{comdiagramFor})$ for the forgetful functor.
\end{proof}
\end{thm}

\begin{rem}
Let $R$ be a ring and $n=\sum_{i=1}^{n} 1_R$. Assume that $R$ is commutative and let $\A$ be an additive $R$-linear category. If $n$ is invertible in $R$, then it is invertitble in $\A$ since for any morphism $f$ we can set $f' \coloneq n^{-1}f$. Recall that, even for a non commutative ring $R$, $\Mod{R}$ is naturally $\mathsf{Z}(R)$-linear category where $\mathsf{Z}(R)$ is the center of $R$. Now, if $n$ is invertible in $\Mod{R}$, consider $\mathsf{Id}_R$ as an $R$-module homorphism, then $\mathsf{Id}=nf'$ for some unique $f' \in \Hom_R(R,R)$. Hence, $1_R=\mathsf{Id}(1_R)=nf'(1_R)$, i.e.\ $n$ is invertible in $R$. We have proved that $n=|G|$ is invertible in $\Mod{R}$ if and only if $n$ is invertible in $R$. 
\end{rem}

Recall that the singularity category of an Artin algebra $\Lambda$ is defined to be the Verdier quotient $\mDb(\smod \Lambda)/\mKb(\proj (\Lambda))$. It is a result of Chen that the singularity category of an Artin algebra $\Lambda$ is idempotent complete, see \cite[Corollary~2.4]{XiaoWuChen3}). Thus the results of Theorem~\ref{main6} can be applied to module categories over Artin algebras.  
%We recall the following important result of Chen on the idempotent completion of .
%\begin{cor}\label{artinalgebraidempotentcomplete}\textnormal{(\!\!\cite[Corollary 2.4]{XiaoWuChen3})}
%    The singularity category $\mDsg (\Lambda)$ of an Artin algebra $\Lambda$ is idempotent complete. 
%\end{cor}
%This is useful when working with module categories over some Artin algebra $\Lambda$ over a commutative ring $R$. 
If we have a finite group $G$ acting on $\smod\Lambda$ by automorphisms (see subsection~\ref{skewgroupalgebra}), then $(\smod\Lambda)^G \simeq \smod \Lambda G$ and, moreover, $\Lambda G$ is an Artin algebra. 
%if $\Lambda$ is finitely generated as an $R$-module then so is $\Lambda G$ since $G$ is finite. 
Hence, the singularity category $\mDsg (\Lambda G)$ is idempotent complete, thus so is $\mDsg ((\smod \Lambda)^G)$ since these two are equivalent categories.
We have the following corollary which provides an equivariant version of \cite[Theorem~8.1]{PSS}.

\begin{cor}\label{recollementofmodulesandequiv}
	Let $\Lambda$ be an Artin algebra and $(\smod \Lambda/\Lambda e\Lambda, \smod \Lambda, \smod e \Lambda e)$ be the recollement induced by an idempotent $e \in R^G$ and $G$ a finite group acting by algebra automorphisms with $|G|$ invertible in $\Lambda$. Then the following are equivalent:
	\begin{itemize}
        \item[(i)] The functor $\mathsf{e} \colon \smod \Lambda  \to \smod e\Lambda e$ induces a singular equivalence.
        \item[(ii)] The functor $\mathsf{e}^G\colon (\smod \Lambda)^G \to( \smod e\Lambda e)^G$ induces a singular equivalence.
        \item[(iii)] The functor $\mathsf{e'}\colon \smod \Lambda G \to \smod e'\Lambda G e'$ induces a singular equivalence, where $e'=e1_G \in \Lambda G$.
        \vspace{1.5mm}
        \item[(iv)] $\pd_{\Lambda}(\frac{\Lambda /\Lambda e\Lambda}{\mathsf{rad}(\Lambda /\Lambda e\Lambda)}) < \infty$ and $\pd_{e\Lambda e}e\Lambda < \infty$.
        \vspace{1.5mm}
        \item[(v)] $\pd_{\Lambda}(\frac{\Lambda G/\Lambda Ge'\Lambda G}{\mathsf{rad}(\Lambda G/\Lambda Ge'\Lambda G)}) < \infty$ and $\pd_{e'\Lambda Ge'}e'\Lambda G < \infty$.
    \end{itemize}
\end{cor}

\begin{rem}
It is known that $\Lambda$ and $\Lambda G$ are separably equivalent. We refer to \cite[Section 3]{Bergh} for the definition and see the proof of \cite[Theorem 4.1]{Bergh} for the separability. In this setup, it is natural to ask if $\Lambda$ and $\Lambda G$ are singularly equivalent. By \cite[Theorem 2.5]{ZhaoSun}, a pair of separably equivalent algebras is equivalent to a separable pair of functors. This is given by the bi-adjoint pair $(\mathsf{Ind}, \mathsf{For})$, i.e.\  induction and restriction of scalars between $\smod\Lambda$ and $\smod\Lambda G$. In \cite[Theorem 3.9]{ZhaoSun} the authors characterize when a separable pair induces a singular equivalence. There are examples showing that this is not the case, see for instance \cite[Section~4]{ChenLu} which deals with singularity categories of skewed-gentle algebras.
\end{rem}

\section{Applications and Examples}
\label{sectionexamples}

This section provides applications and examples illustrating the results developed in the previous sections. It consists of three subsections, the first one is of geometric nature, the second one deals with triangular matrix rings and the third one is about singular equivalences with level and singular Hochschild cohomology.
%n particular, we have three subsections, one for each example. These include a geometric example, one involving triangular matrix rings and one regarding singular equivalences with level.

\subsection{Geometric Example}\label{geomexample}

In this section $X$ will be a scheme, $\mD(X)$ is the derived category of sheaves of $\mathcal{O}_X$-modules and $\mD_\mathsf{qc}(X)$ its full subcategory consisting of complexes with quasi-coherent cohomology. For a quasi-compact quasi-separated scheme $X$ with a closed subscheme $Z \subset X$ such that its complement $U = X \setminus Z$ is quasi-compact, J\o{}rgensen \cite{Jorgensen} proved that there exists the following recollement:
\begin{equation}\label{jorgensen}
\begin{tikzcd}
\mD_{\mathsf{qc}} (U)  \arrow[rr, "\mathbb{R}u_*" description]& &\mD_{\mathsf{qc}}(X) \arrow[ll, bend left, ""] \arrow[ll, bend right, "\mathbb{L}u^*"'] \arrow[rr] & & \mD_{\mathsf{qc},Z}(X)  \arrow[ll, bend left, ""] \arrow[ll, bend right, "v"']
\end{tikzcd}
\end{equation}
where $v$ is the inclusion of the full subcategory $\mD_{\mathsf{qc},Z}(X)$ of complexes with quasi-coherent cohomology supported on $Z$ and $u \colon U \to X$ the inclusion of schemes.

In fact, for a scheme $X$ which is either quasi-compact with affine diagonal or noetherian, the natural forgetful functor $ \Psi_X \colon \mathsf{D}(\mathsf{Qcoh}(X)) \to \mD_\mathsf{qc}(X)$ is an equivalence of categories, see \cite[Corollary~5.5]{BokstedtNemman} and discussion preceding \cite[Theorem~1.2]{hnr}.
Under either of the above assumptions on $X$, the open subscheme $U$ satisfies the same conditions: if $X$ is noetherian, then $U$ is noetherian; if $X$ has affine diagonal, then $U$ also has affine diagonal. Note that if $X$ and $U$ are assumed to be quasi-compact in the second case, then recollement $(\ref{jorgensen})$ exists as we indicate in the next remark:

\begin{rem}\label{jorgensen2}
Let $X$ be a scheme and $Z \subset X$ be a closed subscheme with complement $U$. If $X$ is either noetherian (without assumptions on $U$) or quasi-compact with affine diagonal and $U$ is quasi-compact, then there exists the next recollement:
\begin{equation*}
\begin{tikzcd}
\mD (\mathsf{Qcoh}(U))  \arrow[rr, "\mathbb{R}u_*" description]& &\mD(\mathsf{Qcoh}(X)) \arrow[ll, bend left, ""] \arrow[ll, bend right, "\mathbb{L}u^*"'] \arrow[rr] & & \mD_Z(\mathsf{Qcoh}(X))  \arrow[ll, bend left, ""] \arrow[ll, bend right, "v"']
\end{tikzcd}
\end{equation*}
which is equivalent to the recollement $(\ref{jorgensen})$. In particular, $\Psi_X$ restricts to an equivalence $\Psi_X \big|_Z \colon \mD_Z(\mathsf{Qcoh}(X)) \to \mD_{\mathsf{qc},Z}(X)$, $v$ denotes the inclusion $\mD_Z(\mathsf{Qcoh}(X)) \hookrightarrow \mD(\mathsf{Qcoh}(X))$ and $\mathbb{L}u^*$ and $\mathbb{R}u_*$ denote the derived functors induced by the inclusion $u \colon U \to X$ which commute with the equivalence of recollements induced by the triple $(\Psi_U, \Psi_X, \Psi_X\big|_Z)$.
\end{rem}

Let $X$ be a $k$-scheme and $G$ be a finite subgroup of $\mathsf{Aut}(X)$. Recall from Example~\ref{exampleonsheaves1}  that $G$ induces a right action on $\mathsf{Qcoh}(X)$ by $g^*$ (i.e.\ pullbacks of automorphisms). Since $U$ is a $G$-invariant subspace, then the $G$-action restricts to $\mathsf{Qcoh}(U)$. Moreover, the pullback $u^*$, which is the restriction to $U$, is a $G$-functor:
\begin{equation}\label{u*Gfunctor}
 u^*g^* \simeq  (gu)^* = (ug)^* \simeq g^* u^* 
\end{equation}
where the middle equality holds since $U$ is $G$-invariant.

The $G$-action on $\mathsf{Qcoh}(X)$ (resp.\ $\mathsf{Qcoh}(U)$) extends to an admissible $G$-action on $\mD(\mathsf{Qcoh}(X))$ (resp.\ $\mD(\mathsf{Qcoh}(U))$ by derived pullbacks $\mD(g^*)$ since each $g^*$ is an exact functor. Moreover, following the proof of Lemma~\ref{D^b(e)Gfunctor} we can show that $\mathbb{L}u^*$ is a $G$-functor. Indeed, $u^*$ is an exact functor (because open immersions are flat), thus the derived functor $\mathbb{L}u^*$ is applied component-wise which also holds for $\mD(g^*)$ since $g^*$ are also exact auto-equivalences for each $g \in \mathsf{Aut}(X)$. We conclude using the equivalence~\ref{u*Gfunctor} which implies that the $G$-functor $u^*$ lifts to a $G$-functor between the derived categories. This implies that $\mathbb{R}u_*$ is a $G$-functor, by Lemma~\ref{adjointequiv}. Note that $\mD(\mathsf{Qcoh}(X))$ admits a (unique) dg-enhancement, see for example \cite[Theorem~4.6]{SC}. This implies that $\mD(\mathsf{Qcoh}(X))^G$ is canonically triangulated, by Remark~\ref{trianequivrem}, assuming that $|G|$ is invertible in $k$.  Then, by Theorem~\ref{main2}, we have that there exists a recollement of triangulated categories

\begin{equation}\label{equivariantgeometricrecollement}
\begin{tikzcd}
\mD (\mathsf{Qcoh}(U))^G  \arrow[rr, "(\mathbb{R}u_*)^G" description]& & \mD (\mathsf{Qcoh}(X))^G \arrow[ll, bend left, ""] \arrow[ll, bend right, "(\mathbb{L}u^*)^G"'] \arrow[rr] & & \mD_Z(\mathsf{Qcoh}(X))^G  \arrow[ll, bend left, ""] \arrow[ll, bend right, "v^G"']
\end{tikzcd}
\end{equation}
where $v^G$ is the induced inclusion of subcategories. We remark that we also obtain the equivariant recollement of recollement~$(\ref{jorgensen})$ using Lemma~\ref{inducedequivariantequivalence} and the equivalences $\Psi_U$, $\Psi_X$ and  $\Psi_X \big|_Z$.

It is natural to ask whether the above categories are equivalent to unbounded derived categories of equivariant quasi-coherent sheaves. Moreover, what is the relation they have with the quasi-coherent sheaves of the quotient varieties? To examine this question, we need the following.

\begin{rem}\label{icofunboundedcats}
Recall that an abelian category is Grothendieck if it satisfies AB-5 and has a generator. It is well known that any Grothendieck category has exact products (i.e.\ is AB-4). It is also well known (see \cite[Proposition~077P]{StacksProject}) that for any scheme, $\mathsf{Qcoh}(X)$ is a Grothendieck category, thus it has exact products.

Following \cite[Example~3.20]{ChaoSun}), we have that for any abelian category $\A$ on which a finite group $G$ acts with $|G|$ invertible in $\A$ the comparison functor $K \colon \mD (\A^G) \xrightarrow{} \mD (\A)^G$ is an equivalence up to retracts. If $\A$ satisfies AB-4 (resp.\ AB-4*), then $\A^G$ is also AB-4 (resp.\ AB-4*). Then $\mD (\A^G)$ has arbitrary direct coproducts (resp.\ products) by \cite[Lemma~1.5]{BokstedtNemman} and thus is idempotent complete by \cite[Proposition~1.6.8]{Neeman} (resp.\ \cite[Remark~1.6.9]{Neeman}). Therefore, the derived category $ \mD (\A^G)$ is idempotent complete and the functor $K$ is an equivalence.
\end{rem}

By the above remark, the derived categories $\mD (\mathsf{Qcoh}^G(X))$ and $\mD (\mathsf{Qcoh}^G(U))$ are idempotent complete and therefore, assuming that $|G|$ is invertible in $k$,  we have the following triangle equivalences:
\[
 K_X \colon \mD (\mathsf{Qcoh}^G(X)) \xrightarrow{\simeq} \mD (\mathsf{Qcoh}(X))^G
\]
 and 
\[
K_X \big|_U  = K_U \colon \mD (\mathsf{Qcoh}^G(U)) \xrightarrow{\simeq} \mD (\mathsf{Qcoh}(U))^G
\]
where $K_X$ and $K_U$ are the comparison functors.

Thus, we can identify $\mD(\mathsf{Qcoh}^G(X))$ and $\mD(\mathsf{Qcoh}(X))^G$ using this equivalence. Since we have a full embedding $v^G \colon  \mD_Z(X)^G \to \mD(\mathsf{Qcoh}(X))^G$ we can identify $ \mD_Z(\mathsf{Qcoh}(X))^G$ with the full subcategory  $\mD_Z(\mathsf{Qcoh}^G(X))$  by restricting the comparison functor, i.e.\ we have the following commutative square:
\begin{equation}\label{restrictionofcomparrisonfunctortoZ}
\begin{tikzcd}
\mD(\mathsf{Qcoh}^G(X)) \arrow[r, "K_X"] & \mD(\mathsf{Qcoh}(X))^G \\
\mD_Z(\mathsf{Qcoh}^G(X)) \arrow[u, hookrightarrow, "v'"] \arrow[r, "K_X \big|_Z"] & \mD_Z(\mathsf{Qcoh}(X))^G \arrow[u, hookrightarrow, "v^G"]
\end{tikzcd}
\end{equation}
where the vertical arrows are the canonical inclusions. Hence, $\mD_Z(\mathsf{Qcoh}^G(X))$ is the full subcategory, which consists of complexes $(E,\phi)^{\bullet}$ whose cohomology is supported on $Z$ if and only if $E^{\bullet}$ has cohomology supported on $Z$.
\par
Hence, the equivariant recollement~$(\ref{equivariantgeometricrecollement})$ is equivalent via the comparison functors $(K_U, K_X,K_X \big|_Z)$ to the following recollement:
\begin{equation}\label{derivedequivariantrecollementofqcoh}
\begin{tikzcd}
\mD (\mathsf{Qcoh}^G(U))  \arrow[rr, "\mathbb{R}(u_*^G)" description]& & \mD (\mathsf{Qcoh}^G(X)) \arrow[ll, bend left, ""] \arrow[ll, bend right, "\mathbb{L}((u^*)^G)"'] \arrow[rr] & & \mD_Z(\mathsf{Qcoh^G}(X))  \arrow[ll, bend left, ""] \arrow[ll, bend right, "v'"']
\end{tikzcd}
\end{equation}
where $ \mathbb{L}((u^*)^G) = K_U^{-1} \circ (\mathbb{L}u^*)^G \circ K_X $ which is easy to observe using the definition of the comparison functor, the exactness of $u^*$ and $g^*$ and equation~$(\ref{u*Gfunctor})$. %since $u^*$ is exact (open immersions are flat) and also $g^*$ are exact for all $g \in G$, thus $(u^*)^G$ is exact and $\mathbb{L}((u^*)^G)$ is the restriction to $U$. 
Moreover, since $(\mathbb{L}u^*)^G$ admits a right adjoint, so does $ \mathbb{L}((u^*)^G)$ and this adjoint will be $\mathbb{R}(u^G_*) $ by uniqueness of adjoints.
The fact that $v'= K_X^{-1} \circ v^G \circ K_X \big|_Z$ follows by the commutativity of diagram~$(\ref{restrictionofcomparrisonfunctortoZ})$. Similarly we can complete the diagram~$(\ref{derivedequivariantrecollementofqcoh})$ into a recollement by using the adjoints of the recollement~$(\ref{equivariantgeometricrecollement})$. Notice that the comparison functors induce the equivalence $ \mD(\mathsf{Qcoh}(X))^G/ \mD(\mathsf{Qcoh}(U))^G \simeq \mD(\mathsf{Qcoh}^G(X))/ \mD(\mathsf{Qcoh}^G(U))$.
Now we can apply this to the quotient scheme case.

\begin{prop}\label{derivedcategoryofclosedsubscheme}
Let $X$ be a scheme over a field $k$ and $Z$ is a closed subscheme with complement $U$. Assume that either $X$ is noetherian or that $X$ is quasi-compact with affine diagonal and $U$ is quasi compact. Let also $G$ be a finite subgroup of $\mathsf{Aut}(X)$ and assume that $G$ is acting freely on $X$ with $|G|$ invertible in $k$ and $U$ is $G$-invariant. Assume also that the quotient scheme $X/G$ exists. Then we have that $\mD_Z(\mathsf{Qcoh}^G(X))$ is equivalent to $\mD_{Z/G}(\mathsf{Qcoh}(X/G))$. 
\end{prop}
\begin{proof}
Recall by Example~\ref{exampleonsheaves2} that the quotient scheme $X/G$ exists if and only if each $G$-orbit is contained in some affine open subset of $X$. Then, for the $G$-invariant open $U$, the quotient $U/G$ is an open subscheme of $X/G$. Moveore, $Z/G$ also exists and it is the complement of $U/G$ in $X/G$.
We have the following sequence of triangle equivalences:
\begin{flalign*}
\mD_Z(\mathsf{Qcoh}^G(X))  & \simeq \mD_Z(\mathsf{Qcoh}(X))^G  && \text{Induced by $K_X \big|Z$} \\
 &  \simeq  \frac{\mD(\mathsf{Qcoh}(X))^G}{\mD(\mathsf{Qcoh}(U))^G} && \text{By the equivariant recollement~(\ref{equivariantgeometricrecollement}}) \\
 &  \simeq \frac{\mD(\mathsf{Qcoh}^G(X))}{\mD(\mathsf{Qcoh}^G(U))} && \text{Induced by the comparison functors} \\
 & \simeq \frac{\mD(\mathsf{Qcoh}(X/G))}{\mD(\mathsf{Qcoh}(U/G))} && \text{}
\end{flalign*}
Where the last equivalence is induced by the fact that since $G$ is acting freely and the quotient scheme $X/G$ exists, then, by Examples~\ref{exampleonsheaves1} and~\ref{exampleonsheaves2}, we have that  $\pi^* \colon \mathsf{Qcoh}^G(X) \xrightarrow{\simeq} \mathsf{Qcoh}(X/G)$ is the equivalence given by pullbacks of sheaves along the quotient map $\pi \colon X \to X/G$ - similarly we have that $\pi_{|_U} \colon U \to U/G$ induces $(\pi_{|_U})^* \colon \mathsf{Qcoh}^G(U) \xrightarrow{\simeq} \mathsf{Qcoh}(U/G)$.

It is easy to prove, using the quotient map $\pi \colon X \to X/G$ that, if $X$ is noetherian, then so is $X/G$, hence also $U/G$, being an open subscheme of $X/G$. If $X$ (resp.\  $U$) is quasi-compact, then $X/G$ (resp.\  $U/G$) is quasi-compact, since it is a surjective image of a quasi-compact space via the quotient map. Finally, for any affine $V \subset X/G$, we have that $\pi^{-1}(V)$ is affine, we can show that if $X$ has affine diagonal (hence $U$), then so does $X/G$ (hence $U/G$).
Therefore, by Remark~\ref{jorgensen2}, there exists the recollement $\mathsf{R}_{\mathsf{tr}}(\mD (\mathsf{Qcoh}(U/G)), \mD (\mathsf{Qcoh}(X/G)) , \mD_{Z/G} (\mathsf{Qcoh}(X/G) ) )$. 
This implies that there is an equivalence $\frac{\mD(\mathsf{Qcoh}(X/G))}{\mD(\mathsf{Qcoh}(U/G))} \simeq \mD_{Z/G} (\mathsf{Qcoh}(X/G))$.
\end{proof}

Recall by Example~\ref{exampleonsheaves2} that $X/G$ exists when $X$ is a quasi-projective algebraic variety. In fact, the open subvariety $U$ is also quasi-projective and any quasi-projective variety is noetherian.

\begin{cor}
Let $X$ be a quasi-projective algebraic variety over a field $k$ and $Z$ a closed subscheme with complement $U$. Let also $G$ be a finite subgroup of $\mathsf{Aut}(X)$ and assume that $G$ is acting freely on $X$ with $|G|$ invertible in $k$ and $U$ is $G$-invariant. Then we have $\mD_Z(\mathsf{Qcoh}^G(X)) \simeq \mD_{Z/G}(\mathsf{Qcoh}(X/G))$. 
\end{cor}

\begin{rem}
If the action is not free or the quotients scheme $X/G$ does not exist, then recollement~$(\ref{derivedequivariantrecollementofqcoh})$ implies the following recollement for the quotient stacks:
\[
\mathsf{R_{tr}} (\mD (\mathsf{Qcoh}[U/G]), \mD (\mathsf{Qcoh}[X/G]) , \mD_Z (\mathsf{QCoh} (X))^G )
\]
To show the equivalence $\mD_Z(\mathsf{Qcoh}^G(X)) \simeq \mD_{[Z/G]}(\mathsf{Qcoh}([X/G]))$ one would have to use \cite[Theorem~1.2]{hnr} for the quotient stacks $[X/G]$ and $[U/G]$ and the existence of a recollement of the form~$(\ref{jorgensen})$ to obtain the analogue of Remark~\ref{jorgensen2}.
\end{rem}

\subsection{Equivariant Categories of Modules of Triangular Matrix Rings} \label{triangmatrixexample}

In this subsection we examine group actions on module categories of triangular matrix rings and their recollements which are induced by some idempotent. We show that the induced skew group ring is a Morita equivalent to a triangular matrix ring, and then we extend a known corollary about singular equivalences.
We begin by recalling some useful machinery that we will need throughout this section.

Let $F \colon \B \to \A$ be an additive functor between abelian categories. To this functor we associate the \textbf{comma category} denoted by $(F \downarrow \mathsf{Id} )$. Its objects are triples $(A,B, f)$, where $f\colon FB \to A$ is a morphism in $\A$. A morphism $(\a, \b) \colon (A,B,f) \to (A',B',f')$ between triples consists of morphisms $\a \colon A \to A'$ in $\A$ and $\b \colon B \to B'$ in $\B$ such that the following diagram commutes:
\begin{equation*}
\begin{tikzcd}
FB \arrow[r, "f"] \arrow[d, "F\b"']& A \arrow[d, "\a"] \\
FB' \arrow[r, "f'"'] & A'
\end{tikzcd}
\end{equation*}

The comma category is abelian when the functor $F$ is right exact, see \cite{FosGriRei}. We recall the following functors:
\begin{itemize}
\item[(i)] $\mathsf{T}_{\B} \colon \B \to (F \downarrow \mathsf{Id} )$ is defined by $ \mathsf{T}_{\B}B = (FB, B, \mathsf{Id}_{F(B)} ) $ on objects and given a morphism $\b \colon B \to B'$  then $ \mathsf{T}(\b) = (F\b , \b)$.

\item[(ii)] $ \mathsf{U}_{\B} \colon (F \downarrow \mathsf{Id} ) \to \B $ is defined by $\mathsf{U}_{\B} (A,B, f) = B$ and given a morphism $(\a, \b) \colon (A,B,f) \to (A',B',f')$ then $\mathsf{U}_{\B}(\a, \b) = \b $. Similarly we define the functor $ \mathsf{U}_{\A} \colon (F \downarrow \mathsf{Id} ) \to \A $.

\item[(iii)] $\mathsf{Z}_{\B} \colon (F \downarrow \mathsf{Id} ) \to \B$ is defined by $\mathsf{Z}_{\B}B = (0,B, 0) $ on objects and given a morphism $ \b \colon B \to B'$ then $\mathsf{Z}_{\B}(\b) = (0, \b)$. Similarly we define the functor $\mathsf{Z}_{\A} \colon (F \downarrow \mathsf{Id} ) \to \A$.

\item[(iv)] $\mathsf{q} \colon (F \downarrow \mathsf{Id} ) \to  \A $ is defined by $\mathsf{q} (A,B,f) = \mathsf{Coker}(f)$ on objects and, if $(\a, \b) \colon (A,B,f) \to (A',B',f')$ is a morphism, then $\mathsf{q} (\a,\b)$ is the induced morphism $\mathsf{Coker}(f) \to \mathsf{Coker}(f)$ between the cokernels.
\end{itemize}

Then from \cite{Psaroud:survey}, these functors give rise to a recollement:
\begin{equation}\label{1strecolofcomma}
\xymatrix@C=0.5cm{
\A \ar[rrr]^{\mathsf{Z}_{\A}} &&& (F \downarrow \mathsf{Id} )   \ar[rrr]^{\mathsf{U}_{\B}} \ar @/_1.5pc/[lll]_{\mathsf{q}}  \ar @/^1.5pc/[lll]^{\mathsf{U}_{\A}} &&& \B 
\ar @/_1.5pc/[lll]_{\mathsf{T}_{\B}} \ar
 @/^1.5pc/[lll]^{\mathsf{Z}_{\B}}
}
\end{equation}

Notice that $\mathsf{U}_{\A} \mathsf{T}_{\B} \simeq F$ and $\mathsf{q} \mathsf{Z}_{\B} = 0$. Moreover, $\mathsf{U}_{\A}$ and $\mathsf{Z}_{\B}$ are exact. When $F$ has a right adjoint $H\colon \A \to \B$ with unit $\eta \colon \mathsf{Id}_{\B} \to HF$ and counit $\epsilon \colon FH \to \mathsf{Id}_{\A} $ we have the following two additional functors:
\begin{itemize}
\item[(vii)] $\mathsf{H}_{\A} \colon \A \to  (F \downarrow \mathsf{Id} ) $ is defined by $\mathsf{H}_{\A} A = (A, HA, \epsilon_A) $ on objects and on morphisms $\a : A \to A'$ then $\mathsf{H}_{\A}(\a) = (\a, H\a)$.
\item[(viii)] $\mathsf{p} \colon  (F \downarrow \mathsf{Id} ) \to \B $ is defined by $\mathsf{p}(A,B,f) = \mathsf{Ker}( \eta_B \circ Hf)$ on objects and on morphisms is the induced morphism between the kernels.
\end{itemize}
With these two extra functors we have the recollement:
\begin{equation}\label{2ndrecolofcomma}
\xymatrix@C=0.5cm{
\B \ar[rrr]^{\mathsf{Z}_{\B}} &&& (F \downarrow \mathsf{Id} )   \ar[rrr]^{\mathsf{U}_{\A}} \ar @/_1.5pc/[lll]_{\mathsf{U}_{\B}}  \ar @/^1.5pc/[lll]^{\mathsf{p}} &&& \A 
\ar @/_1.5pc/[lll]_{\mathsf{Z}_{\A}} \ar
 @/^1.5pc/[lll]^{\mathsf{H}_{\A}}
}
\end{equation}
In this case we have that $\mathsf{U}_{\B}\mathsf{H}_{\A} \simeq H$.
Moreover, notice that the existence of the right adjoint $H$ yields a right adjoint of $\mathsf{U}_{\A}$, namely $\mathsf{H}_{\A}$, and a right adjoint of $\mathsf{Z}_{\B}$, namely $\mathsf{p}$. Dual to the above construction, we can define the comma category $(\mathsf{Id}\downarrow F)$ where $F\colon \B \to \A$  is a left exact functor.
Recollements of comma categories are canonically related to recollements of abelian categories with certain extra properties. Franjou and Pirashvili proved in \cite[Proposition~8.9]{FranjPir}) the next result, see also \cite[Proposition~3.1]{GaoKoenigPsaroudakis} for a different proof.

\begin{prop}
\label{recolemenetswithcommacats}
Let $\mathsf{R}_{\mathsf{ab}}(\A, \B, \C)$ be a recollement of abelian categories. Assume that $\mathsf{p}$ is exact and $\B$ and $\C$ have enough projectives. Then the recollements $\mathsf{R}_{\mathsf{ab}}(\A, \B, \C)$ and $\mathsf{R}_{\mathsf{ab}}(\A, (\mathsf{pl} \downarrow \mathsf{Id}), \C)$ are equivalent.

Dually, if $\mathsf{q}$ is exact and $\B$ and $\C$ have enough injectives, then the recollements $\mathsf{R}_{\mathsf{ab}}(\A, \B, \C)$ and $\mathsf{R}_{\mathsf{ab}}(\A, ( \mathsf{Id} \downarrow \mathsf{qr}), \C)$ are equivalent.
\end{prop}

Comma categories of module categories are related to module categories of triangular matrix rings. More precisely, consider the triangular matrix ring 
\[
\Lambda = 
\begin{pmatrix}
R & 0 \\
_SN_R & S \\
\end{pmatrix}
\]
where $R$ and $S$ are rings and $_SN_R$ an $S$-$R$-bimodule. It is known that there exists an equivalence between $\Mod \Lambda$ and the comma category $(- \otimes _S N  \downarrow \mathsf{Id})$ where $- \otimes_{S}  N \colon \Mod S \to \Mod R$, see \cite[Chapter~III]{AusReiSmal}. This means that a right $\Lambda$-module is given by a triple $(X,Y,f)$, where $f \colon Y \otimes_S N \to X$ is a morphism in $\Mod R$. More concretely, a right $\Lambda$-module corresponding to $(X,Y,f)$ is the additive group $X \oplus_f Y$ with the right action:
$$ (x,y)  
\begin{pmatrix}
r & 0 \\
n & s \\
\end{pmatrix} 
= (xr + f(y\otimes n) , ys)$$
Morphisms of right $\Lambda$-modules are of the form $(\chi, \psi) \colon(X,Y,f) \to (X', Y', f')$ where $\chi\colon X\to X' $ in $\Mod R$ and $ \psi \colon Y \to Y'$ in $\Mod S$, such that $\chi \circ f = f' \circ (\psi \otimes \mathsf{Id}_N )$, that is the following diagram commutes:
\begin{equation*}
    \begin{tikzcd}
        Y\otimes_S N \arrow[d, "\psi \otimes \mathsf{Id}_N "'] \arrow[r,"f"] & X \arrow[d, "\chi"]\\
        Y' \otimes_S N \arrow[r, "f'"'] & X'
    \end{tikzcd}
\end{equation*}

There is a natural way to construct a recollement induced by an idempotent $e = \begin{psmallmatrix} 0 & 0\\0 & 1\end{psmallmatrix}$ as in subsection~\ref{examplewithidempotents}. In this case we have that $S \simeq e \Lambda e$ and and $R \simeq (1-e)\Lambda (1-e)$. We have that the induced recollement is equivalent to 
\begin{equation}\label{modRmodS}
\begin{tikzcd}
\Mod R \arrow[rr, "\mathsf{Z}_{R}" description]& & \Mod \Lambda \arrow[ll, bend left, "\mathsf{U}_{R}=\mathsf{1-e}"]  \arrow[ll, bend right, "\mathsf{q}"'] \arrow[rr, "\mathsf{e}=\mathsf{U}_{S}"] & & \Mod S \arrow[ll, bend left, "\mathsf{Z}_{S}"] \arrow[ll, bend right, "\mathsf{T}_{S}"']
\end{tikzcd}
\end{equation}
where $\mathsf{e}$ (resp.\  $\mathsf{1-e}$) is the functor induced by the idempotent $e$ (resp.\  $1-e$) and $\mathsf{U}_R \mathsf{T}_S = - \otimes_S N $. Observe that $-\otimes_S N$ admits a right adjoint, thus we obtain the following recollement, which is the recollement induced by the idempotent $1-e$:

\begin{equation} \label{modSmodR}
\begin{tikzcd}
\Mod S \arrow[rr, "\mathsf{Z}_{S}" description]& & \Mod \Lambda \arrow[ll, bend left, "\mathsf{p}"]  \arrow[ll, bend right, "\mathsf{U}_{S}"'] \arrow[rr, "\mathsf{U}_{R}" description] & & \Mod R \arrow[ll, bend left, "\mathsf{H}_{R}"] \arrow[ll, bend right, "\mathsf{Z}_{R}"']
\end{tikzcd}
\end{equation}

Notice that in this setup, $\mathsf{U}_R$ admits a right adjoint, thus it is exact, and similarly $\mathsf{U}_S$ admits a left adjoint. This means that Proposition~\ref{recolemenetswithcommacats} applies on recollement~\ref{modRmodS} for $\mathsf{p} = \mathsf{U}_R$ and on recollement~\ref{modSmodR} for $\mathsf{q}= \mathsf{U}_S$. In fact we can show that any recollement of module categories satisfying either of these properties is a recollement of the above form, i.e.\ of a triangular matrix ring. We summarize this result in the following remark to apply it later on.

\begin{rem}\label{recollementofmodulesistriangular}
Consider a recollement $\mathsf{R}_{\mathsf{ab} }(\Mod R, \Mod T, \Mod S)$. When the functor $\mathsf{p}$ has a right adjoint, it is exact. By Proposition~\ref{recolemenetswithcommacats} the above recollement is equivalent to the recollement $\mathsf{R}_{\mathsf{ab}}(\Mod R, (\mathsf{pl} \downarrow \mathsf{Id}), \Mod S)$, where $\mathsf{pl} \colon \Mod S \to \Mod R$. Moreover, since both $\mathsf{l}$ and $\mathsf{p}$ have right adjoints, their composition $\mathsf{pl}$ admits a right adjoint. Thus, by the Eilenber-Watts theorem, there exists a bimodule $_S N _R$ such that $\mathsf{pl} \simeq - \otimes_S N$. Hence  $(\mathsf{pl} \downarrow \mathsf{Id})$ is equivalent to $\Mod \Lambda$ for some triangular matrix ring $\Lambda$ and the recollement we started with is equivalent to the recollement~\ref{modRmodS}. 
If $\mathsf{q}$ admits a left adjoint, we obtain, similarly, that the recollement of $\Mod T$ is equivalent to a recollement of $(\mathsf{Id} \downarrow \mathsf{qr})$ by Proposition~\ref{recolemenetswithcommacats}. The latter is equivalent to a recollement of some triangular matrix ring.
\end{rem}

We turn our attention to group actions on a triangular matrix ring $\Lambda$ that induces group actions by automorphisms on $\Mod \Lambda$, as in subsection~\ref{examplewithidempotents}. 
Consider a finite group $G$ together with a group homomorphism $G \to \mathsf{Aut}(\Lambda)^{op}$ such that $e \in \Lambda^G$.
Then we have the induced action by automorphisms on $\Mod \Lambda$. By Proposition~\ref{equivariant recollement induced by an idempotent}, we have the next recollement of skew group rings induced by the idempotent $e1_G$ of $\Lambda G$:

\begin{equation}\label{G-triangularmatrix}
\begin{tikzcd}
\Mod RG \arrow[rr, "\mathsf{Z}_{R}^G"]& & \Mod \Lambda G \arrow[ll, bend left, "\mathsf{U}_{R}^G"]  \arrow[ll, bend right, "\mathsf{q}^G"'] \arrow[rr, "\mathsf{e}^G=\mathsf{U}_{S}^G"] & & \Mod SG \arrow[ll, bend left, "\mathsf{Z}_{S}^G"] \arrow[ll, bend right, "\mathsf{T}_{S}^G"']
\end{tikzcd}
\end{equation}
where $\mathsf{e}^G$ is (identified with) the functor induced by the idempotent $e1_G$. 
Note that $\Mod \Lambda G \simeq (\Mod \Lambda)^G$ and in the presence of the above recollement we prove that $\Lambda G$ is Morita equivalent to a triangular matrix ring. 

\begin{prop}\label{triangularmatrixequivariant}
Let $\Lambda$ be a triangular matrix ring. Let $G$ be a finite group acting by automorphisms such that $e \in \Lambda^G$. Then the skew group ring $\Lambda G$ is Morita equivalent to the triangular matrix ring $\Lambda'= \begin{psmallmatrix} RG & 0\\N' & SG\end{psmallmatrix}$ with $N' = \Hom_{\Lambda G}( \Lambda G , \Lambda G / SG  )$.
\end{prop}
\begin{proof}
We use Remark~\ref{recollementofmodulesistriangular} to show that the recollement~$(\ref{G-triangularmatrix})$ is equivalent to a recollement of a triangular matrix ring. Indeed, we know that $\mathsf{H}_R$ is a right adjoint to $\mathsf{U}_R$. Thus, by Lemma~\ref{adjointequiv}, the functor $\mathsf{H}_R$ is also a $G$-functor and $\mathsf{H}_R^G$ is right adjoint to $\mathsf{U}_R^G$. Thus, the comma category $(\mathsf{U}_R^G \mathsf{T}_S^G \downarrow \mathsf{Id})$ is equivalent to the category of modules over the triangular matrix ring 
$\Lambda'= \begin{psmallmatrix} RG & 0\\N' & SG\end{psmallmatrix}$,
where $N'$ is the $SG$-$RG$-bimodule that corresponds to $\mathsf{U}_R^G \mathsf{T}_S^G \simeq -\otimes_{SG} N' \colon \Mod SG \to \Mod RG$. For $N'$ we apply the functor $\mathsf{U}_R^G \mathsf{T}_S^G$ on $SG$ and we obtain that
$N'$ is $\Hom_{\Lambda G}(SG \otimes_{SG} \Lambda G , \Lambda G / SG  ) = \Hom_{\Lambda G}( \Lambda G , \Lambda G / SG  )$
which is naturally a right $RG$-module. It is also a left $SG$-module if we define $sg \cdot n' \coloneqq (\mathsf{U}_R^G \mathsf{T}_S^G)(sg(-))n'$, where $sg(-)\colon SG \to SG$ is the right $SG$-module homomorphism. 
%given by multiplication by $sg$. 
\end{proof}

We have shown that the recollement $\mathsf{R_{ab}}(\Mod RG, \Mod \Lambda G, \Mod SG)$ is equivalent to $\mathsf{R_{ab}}(\Mod RG, \Mod \Lambda', \Mod SG)$ which is the recollement induced by the idempotent $e'=~\begin{psmallmatrix} 0 & 0\\0 & 1_{SG}\end{psmallmatrix}$, that is the idempotent $e1_G$ of $\Lambda G$ corresponds to the idempotent $e'$ of $\Lambda'$.
Moreover, note that the idempotent $(1-e)$ is also $G$-invariant. Thus the recollement~$(\ref{modSmodR})$ induces also a {$G$-equivariant} recollement, which is 
\begin{equation} 
\begin{tikzcd}
\Mod SG \arrow[rr, "\mathsf{Z}_{S}^G"]& & \Mod \Lambda G \arrow[ll, bend left, "\mathsf{p}^G"]  \arrow[ll, bend right, "\mathsf{U}_{S}^G"'] \arrow[rr, "\mathsf{(1-e)}^G=\mathsf{U}_{R}^G"] & & \Mod RG \arrow[ll, bend left, "\mathsf{H}_{R}^G"] \arrow[ll, bend right, "\mathsf{Z}_{R}"']
\end{tikzcd}
\end{equation}
and using the dual part of Remark~\ref{recollementofmodulesistriangular} we obtain, similarly, that $\Lambda G$ is Morita equivalent to the triangular matrix ring $\Lambda'$. Note also that $\mathsf{(1-e)}^G$ is the functor induced by the idempotent $(1-e)G$ of $\Lambda G$ and corresponds to $1-e'$ of $\Lambda'$.

We end this section examining the singular equivalence in this context. 
%We recall that by \cite[Chapter III, Proposition 2.1]{AusReiSmal} $\Lambda$ is Artin if and only if $R$ and $S$ are $k$-Artin for some commutative ring $k$ and $N$ is $R$-finitely generated over $k$ which acts centrally on $N$. 
Recall by \cite[Theorem~1.3]{RR} that $\mathsf{gl.dim}(R) = \mathsf{gl.dim}(RG)$, where $R$ is an Artin algebra, and therefore $\mathsf{gl.dim}(R)<\infty$ if and only if $\mathsf{gl.dim}(RG)<\infty$.

\begin{cor}\label{globaldimensionoftriangmatrixrings}
Let $\Lambda$ be a triangular matrix Artin algebra over a commutative ring $k$. Let $G$ be a finite group acting by automorphisms such that $e \in \Lambda^G$ and $|G|$ is invertible in $k$. Denote by 
$e'= \begin{psmallmatrix} 0 & 0\\0 & 1_{SG}\end{psmallmatrix}$
%\begin{pmatrix}
%0 & 0 \\
%0 & 1_{SG} \\
%\end{pmatrix}
%$
the idempotent of $\Lambda'= \begin{psmallmatrix} RG & 0\\N' & SG\end{psmallmatrix}$ that corresponds to $e 1_G$ of $\Lambda G$. Denote by $\mathsf{e, \, e', \, (1-e)}$ and $\mathsf{(1-e')}$ the functors induced by the idempotents $e$, $e'$, $1-e$ and $1-e'$, respectively. We have the following:
\begin{itemize}
\item[(i)] If $\mathsf{gl.dim.}(R)<\infty$, then $\mDsg(\mathsf{e'}) \colon \mDsg(\Lambda') \xrightarrow{\simeq} \mDsg(SG)$.
\item[(ii)] If $\mathsf{gl.dim.}(S)<\infty$ and $\pd_R N< \infty$, then $\mDsg \mathsf{(1-e')} \colon \mDsg(\Lambda') \xrightarrow{\simeq} \mDsg(RG)$.
\end{itemize}
\begin{proof}
By the preceding discussion and \cite[Corollary~8.19]{PSS}, (i) follows immediately. 
For (ii) we have that $\mDsg( \mathsf{1-e}) \colon \mDsg(\Lambda) \xrightarrow{\simeq} \mDsg(R)$ by \cite[Corollary~8.17]{PSS}.
Since the recollement induced by $1-e$ lifts to a $G$-equivariant recollement~$(\ref{modSmodR})$, Corollary~\ref{recollementofmodulesandequiv} yields that $\mDsg ( \mathsf{(1-e)}^G ) \colon \mDsg(\Lambda G) \xrightarrow{\simeq} \mDsg(RG)$. By Proposition~\ref{triangularmatrixequivariant} we have that $\Lambda G$ is Morita equivalent to the triangular matrix ring $\Lambda '$ and thus we have the singular equivalence $\mDsg(\mathsf{1-e'}) \colon \mDsg(\Lambda') \xrightarrow{\simeq} \mDsg(RG)$.
\end{proof}
\end{cor}

\begin{rem}
A ring $R$ is strongly indecomposable if there is no non-trivial idempotent $e$ of $R$ such that $(1-e) Re=0$.
For triangular matrix rings where $R, S$ are strongly indecomposable, by \cite[Theorem~3.2]{AhnWyk}, the group of automorphisms of the ring $\Lambda$ is a subgroup of $\mathsf{Aut}(R)\times \mathsf{Aut}(S)\times N \times \mathsf{Aut}(_SN_R)$ where the summand $N$ is viewed only as an abelian group. Namely, for an element $(\rho, \kappa, n_0, \lambda)\in \mathsf{Aut}(\Lambda)$: 
$$(\rho, \kappa, n_0, \lambda)
\begin{pmatrix}
r & 0 \\
n & s \\
\end{pmatrix}
=
\begin{pmatrix}
\rho(r) & 0 \\
n_0\rho(r)+\lambda(n)-\kappa(s)n_0 & \kappa(s) \\
\end{pmatrix}
$$
Then for a finite group $G$ together with a group homomorphism $G \to \mathsf{Aut}(\Lambda)^{op}$, i.e.\ the codomain is a subgroup of $\mathsf{Aut}(R)^{op}\times \mathsf{Aut}(S)^{op}\times N^{op} \times \mathsf{Aut}(_SN_R)^{op}$, we observe that the idempotent $e = \begin{psmallmatrix} 0 & 0\\0 & 1\end{psmallmatrix}$ inducing recollement~$(\ref{modRmodS})$ is $G$-invariant. Note that the induced action by automorphisms on $\Mod S$ stems from the restriction of $G \to \mathsf{Aut}(\Lambda)^{op}$ to the summand $\mathsf{Aut}(S)^{op}$. Similarly, the action on $\Mod R $ is induced by the restriction to $\mathsf{Aut}(R)^{op}$.  
\end{rem}

%%%%%%%%%%%%%%%%%%%%%%%%%%%%%%%%%%%%%%%%%%%%%%%%%%%%%%%%%%%%%%%%%%%%%%%%%%%%%%%%%%%%%%%%%%%%%
\subsection{Singular Equivalence with Level}\label{equivariantmoritatypewithlvl}
%%%%%%%%%%%%%%%%%%%%%%%%%%%%%%%%%%%%%%%%%%%%%%%%%%%%%%%%%%%%%%%%%%%%%%%%%%%%%%%%%%%%%%%%%%%%%

For this subsection let $\Lambda$ be a finite dimensional associative algebra over a field $k$. Denote by $\mathsf{rad}(\Lambda)$ the Jacobson radical of $\Lambda$. The semisimple quotient $\Lambda/\mathsf{rad}(\Lambda)$ is separable if $\Lambda/\mathsf{rad}(\Lambda)$ remains semisimple under extension of scalars to a field containing $k$. Let also $\Lambda^e = \Lambda \otimes_k \Lambda^{op}$ be the enveloping algebra of $\Lambda$. We identify $\Lambda$-$\Lambda$-bimodules with left $\Lambda^e$-modules. 

Denote by $\umod \Lambda$ the stable module category of $\smod  \Lambda$ modulo morphisms factoring through projective modules. Denote also by $\mathrm{\Omega}_{\Lambda} $ the sygyzy endofunctor of $\umod \Lambda$. 
%which, for each module $M$, is the (unique in $\umod \Lambda$) kernel of the surjection of some projective $P$ onto $M$, i.e.\ $\mathrm{\Omega}_{\Lambda}(M) = \Ker (P \twoheadrightarrow M)$.
When $\Lambda$ is self-injective, it follows from Happel \cite{Happel} that $\umod \Lambda$ is triangulated and, moreover, it is equivalent to the singularity category $\mDsg(\Lambda)$ by the work of Rickard \cite{JRick}. 
%In particular, when $\Lambda$ is self-injective, $\mathrm{\Omega}_{\Lambda} $ is an equivalence and its inverse is the shift operator of $\umod \Lambda$. 
We recall the next notion due to Wang \cite{Wang}. %defined the following notion.

\begin{defn}
Let $A$ and $B$ be finite dimensional algebras over a field $k$.
    Let $_AM_B$ and $_BN_A$ be an $A$-$B$-bimodule and a $B$-$A$-bimodule, respectively, and let $n \geq 0$. We say that $(M,N)$ defines a \textbf{singular equivalence of Morita type with level} $n$, if the following conditions are satisfied:
    \begin{enumerate}
        \item The modules  $_AM$, $M_B$, $_BN$ and $N_A$ are finitely generated and projective.
        \item There are isomorphisms $M \otimes_B N \simeq \mathrm{\Omega}^n_{A^e}(A)$ and $N \otimes_A M \simeq \mathrm{\Omega}^n_{B^e}(B)$  in $\umod A^e$ and in $\umod B^e$, respectively.
    \end{enumerate}
\end{defn}
It is well known that if $(M,N)$ defines a singular equivalence of Morita type with level $n$, then the functor $M\otimes_B - $ induces a singular equivalence between $A$ and $B$, i.e.\ the functor $M\otimes_B -\colon \mathsf{D}_{\mathsf{sg}}(B) \to \mathsf{D}_{\mathsf{sg}}(A)$ is a triangulated equivalence.

In the following we recall Qin's theorem on singular equivalences of Morita type with level in a recollement situation over finite dimensional algebras.

\begin{thm} \label{moritatypewithlvlofidempotent}
\textnormal{(\!\!\cite[Theorem~4.1]{YQ})}
Let $\Lambda$ be a finite-dimensional algebra over a field $k$ such that $\Lambda/\mathsf{rad}(\Lambda)$ is separable over $k$ and let $e \in \Lambda$ be an idempotent. If $\mathsf{e}\colon  \smod \Lambda \to \smod e\Lambda e$ induces a singular equivalence,  then $\Lambda$ and $e \Lambda e$ are singularly equivalent of Morita type with level.  
\end{thm} 

Now we can use the machinery of Corollary~\ref{recollementofmodulesandequiv} and obtain the following.

\begin{cor}\label{equivariantmoritatypewithlvlcor}
Let $\Lambda$ be a finite-dimensional algebra over a field $k$ and let $e \in \Lambda$ be an idempotent. Let also $G$ be a finite group acting on $\smod \Lambda$ by $k$-algebra automorphisms such that $|G|$ is invertible in $k$. The following are equivalent:
\begin{itemize}
 	\item[(i)] The functor $\mathsf{e} \colon \smod \Lambda \to  \smod e\Lambda e$ induces a singular equivalence.
        \item[(ii)] The functor $\mathsf{e}^G\colon ( \smod \Lambda)^G \to (\smod e\Lambda e)^G$ induces a singular equivalence.
\end{itemize}
If we additionally assume that $\Lambda /\mathsf{rad}(\Lambda)$ is separable over $k$, then we have the following properties:
\begin{itemize}
        \item[(1)] $\Lambda$ and $e \Lambda e$ are singularly equivalent of Morita type with level.
        \item[(2)] $\Lambda G$ and $e'\Lambda Ge'$ are singularly equivalent of Morita type with level.
    \end{itemize}
\begin{proof}
    The proof is just a combination of Corollary~\ref{recollementofmodulesandequiv}, Theorem~\ref{moritatypewithlvlofidempotent} and of the following easy remark.
\end{proof}
\end{cor}

\begin{rem}
If $\Lambda /\mathsf{rad}(\Lambda)$ is separable over $k$, then $\Lambda G/\mathsf{rad}(\Lambda G)$ is also separable over $k$. Indeed, for any field extension $k' /k$  we have 
$$k' \otimes_k \Lambda G/\mathsf{rad}(\Lambda G) \simeq k' \otimes_k (\Lambda /\mathsf{rad}(\Lambda)) G \simeq  (k' \otimes_k \Lambda/\mathsf{rad}(\Lambda)) \otimes_k kG  $$
and since $k' \otimes_k \Lambda /\mathsf{rad}(\Lambda)$ is semisimple, so is $k' \otimes_k \Lambda G/\mathsf{rad}(\Lambda G)$. The fact that $\Lambda G/\mathsf{rad}(\Lambda G)\simeq (\Lambda /\mathsf{rad}(\Lambda))G$ for Artin algebras with finite group actions and $|G|$ invertible in $k$ can be found at \cite[Paragraph~1.5]{RR}.
\end{rem}

\begin{rem} 
We would like to mention that work of Asashiba and Pan \cite{AP} provide us with a $G$-invariant equivalence of Morita type (resp.\ stable or singular or singular with level) if and only if there exists a $G$-graded  equivalence of Morita type (resp.\ stable or singular or singular with level). This type of equivalences uses a 2-categorical context and extends previous work of Asashiba \cite{HA}, \cite{HA2} on a 2-categorical Cohen-Montgomery duality. Their 2-categorical setting does not readily specialise in the setting of this paper and, moreover, they use a modified definition for Morita type equivalences. For more details we refer to the above papers of Asashiba. 
%However this does not mean that there is not a suitable notion for this purpose, so this might be worth exploring. 
\end{rem}

We finish this paper with the next corollary on singular Hochschild cohomology. For the precise definition and further details we refer to \cite{Keller}.

\begin{cor}
\label{HHC}
Let $\Lambda$ be a finite-dimensional algebra over a field $k$ and let $e$ an idempotent element of $\Lambda$. Assume that $\Lambda /\mathsf{rad}(\Lambda)$ is separable over $k$ and let $G$ be a finite group acting on $\smod \Lambda$ by $k$-algebra automorphisms such that the order $|G|$ is invertible in $k$. Assume that the functor $\mathsf{e} \colon \smod \Lambda \to  \smod e\Lambda e$ induces a singular equivalence. Then there is an isomorphism of Gerstenhaber algebras between the singular Hochschild cohomology:
\[
\mathsf{HH}^*_{\mathsf{sg}}(\Lambda) \cong \mathsf{HH}^*_{\mathsf{sg}}(e\Lambda e)
\]
and
\[
\ \ \mathsf{HH}^*_{\mathsf{sg}}(\Lambda G) \cong \mathsf{HH}^*_{\mathsf{sg}}(e'\Lambda G e')
\]
\begin{proof}
Our setup together with Corollary~\ref{equivariantmoritatypewithlvlcor} implies that there is a singular equivalence of Morita type with level between $\Lambda$ and $e \Lambda e$ are between $\Lambda G$ and $e'\Lambda Ge'$. Then the desired Gerstenhaber algebra isomorphisms of the singular Hochschild cohomology are direct consequences of \cite[Theorem~6.2]{Wang:SingularHH}.
\end{proof}
\end{cor}

\end{document}